\documentclass[twosided,b4paper]{amsart}
\usepackage{mypackage}
\usepackage{mathdots}
\usepackage{hyperref}\hypersetup{colorlinks}
\newcommand{\mtrx}[1]{\left (\begin{matrix}#1\end{matrix}\right)}

\usepackage{color} 

\definecolor{darkred}{rgb}{1,0,0} 
\definecolor{darkgreen}{rgb}{0,0.8,0}
\definecolor{darkblue}{rgb}{0,0,1}

\hypersetup{colorlinks,
linkcolor=darkblue,
filecolor=darkgreen,
urlcolor=darkred,
citecolor=darkgreen}

\title[Maximal representations into $\SO_0(2,n)$]{The geometry of maximal representations of surface groups into $\SO_0(2,n)$}

\author{Brian Collier}
\address{Department of Mathematics \\
University of Maryland \\
4176 Campus Drive \\
College Park, MD 20742}
\email{briancollier01@gmail.com}

\author{Nicolas Tholozan}
\address{
DMA, \'Ecole Normale Sup\'erieure \\
CNRS -- PSL Research University \\
45 rue d'Ulm\\
75230 Paris Cedex 5 \\
France}
\email{nicolas.tholozan@ens.fr}

\author{J\'{e}r\'{e}my Toulisse}
\address{
Department of Mathematics \\
Universit\'e C\^{o}te d’Azur \\
LJAD, Nice F-06000\\France}
\email{ jtoulisse@unice.fr }
\date{\today}

\begin{document}
\begin{abstract}
In this paper, we study the geometric and dynamical properties of maximal representations of surface groups into Hermitian Lie groups of rank $2$. Combining tools from Higgs bundle theory, the theory of Anosov representations, and pseudo-Riemannian geometry, we obtain various results of interest.

We prove that these representations are holonomies of certain geometric structures, recovering results of Guichard and Wienhard. We also prove that their length spectrum is uniformly bigger than that of a suitably chosen \emph{Fuchsian} representation, extending a previous work of the second author. Finally, we show that these representations preserve a unique minimal surface in the symmetric space, extending a theorem of Labourie for Hitchin representations in rank $2$.
\end{abstract}
\maketitle
\tableofcontents

\section{Introduction}

In the past decades, two major theories have allowed various breakthroughs in the understanding of surface group representations into semi-simple Lie groups, leading to what is sometimes called \emph{higher Teichm\"uller theory}. 

On one side, non-abelian Hodge theory gives a bijective correspondence between conjugacy classes of representations of the fundamental group of a closed Riemann surface into a semi-simple Lie group and holomorphic objects on the Riemann surface called \emph{Higgs bundles}. This theory, developed by Hitchin, Simpson, Corlette and many others, has proven very useful in describing the topology of representation varieties of surface groups (see \cite{hitchinselfduality}, \cite{hitchin} or \cite{gothen}). 

On the other side, Labourie showed in \cite{labouriehyperconvex} that many surface group representations share a certain dynamical property called the \emph{Anosov} property. 
This property has strong geometric and dynamical implications similar to the \emph{quasi-Fuchsian} property for surface group representations into $\PSL(2,\C)$.

A recent trend in the field is to try to link these two approaches to higher Teichm\"uller theory (see for instance \cite{AlessandriniLi,BaragliaThesis,CollierLi,LabourieWentworth}) by extending previous works which highlighted the importance of harmonic maps in Teichm\"uller theory (see for instance \cite{WolfHarmonicTeichmuller}).
Currently, such links are far from being fully understood. 
For example, there is no known Higgs bundle characterization of Anosov representations.
 The main obstacle is that finding the representation associated to a given Higgs bundle involves solving a highly transcendental system of PDEs called the \emph{self-duality equations}.

However, in some cases the self-duality equations simplify, and one can hope to reach a reasonably good understanding of their solutions. These simplifications happen when the Higgs bundle is \emph{cyclic}. Unfortunately, not every Higgs bundle is cyclic. Nevertheless, it turns out that restricting to cyclic Higgs bundles is enough to study representations into most Lie groups of real rank $2$. This was used by Labourie \cite{labouriecyclic} to study Hitchin representations into $\PSL(3,\R)$, $\PSp(4,\R)$ and $\mathrm{G}_2$, and by the first author and Alessandrini \cite{colliercyclic,PSp4maximalRepsAC} to study maximal representations into $\PSp(4,\R)$.

The goal of this paper is to derive from Higgs bundle theory several geometric properties of representations of surface groups into Hermitian Lie groups of rank $2$. According to the work of Burger, Iozzi and Wienhard \cite{burgeriozziwienhard}, it is enough to restrict to representations into the Lie groups $\SO_0(2,n+1)$, $n\geq 1$ (see Remark \ref{rmk:LieGroupsRank2}).

\subsection*{Geometrization of maximal representations}
Hitchin representations into split real Lie groups \cite{labouriehyperconvex} and maximal representations into Hermitian Lie groups \cite{BILW} are two important families of Anosov representations.
One nice feature of Anosov representations is that they are holonomies of certain geometric structures on closed manifolds. More precisely, for every Anosov representation $\rho$ of a hyperbolic group $\Gamma$ into a semi-simple Lie group $G$, Guichard and Wienhard \cite{wienhardanosov} construct a $\rho$-invariant open domain $\Omega$ in a certain \emph{flag manifold} $G/P$ on which $\rho(\Gamma)$ acts properly discontinuously and co-compactly. 

In our setting, their result can be reformulated as follows. Let $\R^{2,n+1}$ denote the vector space $\R^{n+3}$ equipped with the quadratic form
\[\mathbf{q}(\mathbf{x}) = x_1^2+x_2^2 - x_3^2 - \ldots - x_{n+3}^2~.\]
We denote by $\Ein^{1,n}$ the space of isotropic lines in $\R^{2,n+1}$ and by  $\Pho(\R^{2,n+1})$ the space of \emph{photons} in $\Ein^{1,n}$ or, equivalently, of totally isotropic planes in $\R^{2,n+1}$. By Witt's theorem, $\SO_0(2,n+1)$ acts transitively on both $\Ein^{1,n}$ and $\Pho(\R^{2,n+1})$.

\begin{theo}[Guichard--Wienhard \cite{wienhardanosov}] \label{t:GuicharWienhard}
Let $\Gamma$ be the fundamental group of a closed oriented surface $\Sigma$ of genus at least two. If $\rho: \Gamma \to \SO_0(2,n+1)$ is a maximal representation $(n\geq 2)$, then there exists a non-empty open domain $\Omega_\rho$ in $\Pho(\R^{2,n+1})$ on which $\Gamma$ acts properly discontinuously and co-compactly via $\rho$.
\end{theo}

In particular, the representation $\rho$ is the holonomy of a \emph{photon structure} on the closed manifold $\rho(\Gamma) \backslash \Omega_\rho$ (see Definition \ref{d:FiberedPhotonStructure}). One drawback of the construction of Guichard--Wienhard is that, a priori, it gives neither the topology of the domain $\Omega_\rho$ nor the topology of its quotient by $\rho(\Gamma)$. In forthcoming work \cite{guichardwienhardtopology}, a clever -- but indirect -- argument is used to describe this topology in the case of maximal representations into the symplectic group $\Sp(2n,\R)$. In an earlier paper \cite{guichardwienhardsl4}, they focus on Hitchin representations into $\SO_0(2,3)$\footnote{To be more accurate, Guichard and Wienhard study Hitchin representations into $\PSL(4,\R)$ and, in particular, into $\PSp(4,\R)$, and their action on the projective space $\ProjR{3}$. By a low dimensional isomorphism, $\PSp(4,\R)$ is isomorphic to $\SO_0(2,3)$ and $\ProjR{3}$ identifies (as a $\PSp(4,\R)$-homogeneous space) with $\Pho(\R^{2,3})$.} and give a more explicit parametrization of (the two connected components of) $\Omega_\rho$ by triples of distinct points in $\ProjR{1}$, thus identifying $\rho(\Gamma) \backslash \Omega_\rho$ with the unit tangent bundle of $\Sigma$. In this parametrization, however, the circle bundle structure of the manifold is not apparent.

Here, we will construct photon structures on certain fiber bundles over $\Sigma$ whose holonomy is any prescribed maximal representation into $\SO_0(2,n+1)$ and in such a way that the fibers are ``geometric''. We will show that these photon structures coincide with the Guichard--Wienhard structures, and thus describe the topology of Guichard--Wienhard's manifolds in this setting.

\begin{MonThm}\label{t:photontructures}
Let $\Gamma$ be the fundamental group of a closed oriented surface $\Sigma$ of genus at least two. If $\rho: \Gamma \to \SO_0(2,n+1)$ is a maximal representation $(n\geq 2)$, then there exists a fiber bundle $\pi : M \to \Sigma$ with fibers diffeomorphic to $\mathrm{O}(n)/\mathrm{O}(n-2)$, and a $\Pho(\R^{2,n+1})$-structure on $M$ with holonomy $\rho \circ \pi_*$. Moreover, the developing map of this photon structure induces an isomorphism from each fiber of $\pi$ to a copy of $\Pho(\R^{2,n}) \subset \Pho(\R^{2,n+1})$.

Conversely, if $\pi : M \to \Sigma$ is a fiber bundle with fibers diffeomorphic to $\mathrm{O}(n)/\mathrm{O}(n-2)$, then any photon structure on $M$ whose developing map induces an isomorphism from each fiber of $\pi$ to a copy of $\Pho(\R^{2,n}) \subset \Pho(\R^{2,n+1})$ has holonomy $\rho\circ \pi_*$, where $\rho:\Gamma \to \SO_0(2,n+1)$ is a maximal representation.
\end{MonThm}

\begin{MonCoro}
The manifold $\rho(\Gamma) \backslash \Omega_\rho$ in Guichard--Wienhard's Theorem \ref{t:GuicharWienhard} is diffeomorphic to a $\mathrm{O}(n)/\mathrm{O}(n-2)$-bundle over $\Sigma$.
\end{MonCoro}

\begin{rmk}
The proof of Theorem \ref{t:photontructures} in Section \ref{s:Geometrization} gives additional information on the topology of the fiber bundle $M$, which depends on certain topological invariants of the representation $\rho$. 
\end{rmk}

Hitchin representations into $\SO_0(2,3)$ are the special class of maximal representations that also have a Guichard--Wienhard domain of discontinuity in $\Ein^{1,2}$. In a manner similar to \cite{guichardwienhardsl4}, this domain can be parametrized by triples of distinct points in $\ProjR{1}$ so that its quotient by $\rho(\Gamma)$ is homeomorphic to the unit tangent bundle to $\Sigma$. Here, we recover this $\Ein^{1,2}$-structure (referred to as a conformally flat Lorentz structure) on the unit tangent bundle to $\Sigma$ in such a way that the fibers are ``geometric'':

\begin{MonThm}\label{t:einsteinstructures}
Let $\Gamma$ be the fundamental group of a closed oriented surface $\Sigma$ of genus at least two. Let $T^1\Sigma$ denote the unit tangent bundle to $\Sigma$ and $\pi : T^1\Sigma \to \Sigma$ the bundle projection. If $\rho: \Gamma \to \SO_0(2,3)$ is a Hitchin representation, then there exists a $\Ein^{1,2}$-structure on $T^1\Sigma$ with holonomy $\rho \circ \pi_*$. Moreover, the developing map of this $\Ein^{1,2}$-structure induces an isomorphism from each fiber of $\pi$ to a copy of $\Ein^{1,0} \subset \Ein^{1,2}$.
\end{MonThm}

For the group $\SO_0(2,2),$ Alessandrini and Li \cite{AlessandriniLi} used Higgs bundle techniques to construct anti-de Sitter structures on circle bundles over $\Sigma$, recovering a result of Salein and Gu\'eritaud and Kassel \cite{Salein,GueritaudKassel}.

\subsection*{Length spectrum of maximal representations in rank $2$}

Some Anosov representations of surface groups, such as Hitchin representations into real split Lie groups or maximal representations into Hermitian Lie groups, have the additional property of forming connected components of the whole space of representations. There have been several attempts to propose a unifying characterization of these representations (see \cite{MartoneZhang} and \cite{guichardwienhardpositivity}). Note that quasi-Fuchsian representations into $\PSL(2,\C)$  do not form components; indeed, they can be continuously deformed into representations with non-discrete image.

The property of lying in a connected component consisting entirely of Anosov representations seems to be related to certain geometric controls of the representation ``from below'' such as an upper bound on the entropy or a \emph{collar lemma}. To be more precise, let us introduce the \emph{length spectrum} of a representation.

\begin{defi} \label{d:LengthSpectrumIntro}
Let $\rho$ be a representation of $\Gamma$ into $\SL(n,\R)$, $n\geq 2$. Let $[\Gamma]$ denote the set of conjugacy classes in $\Gamma$. The \emph{length spectrum} of $\rho$ is the function
\[\function{L_\rho}{[\Gamma]}{\R_+}{\gamma}{\frac{1}{2} \log\left|\frac{\lambda_1(\rho(\gamma))}{\lambda_n(\rho(\gamma))}\right |~,}\]
where $\lambda_1(A)$ and $\lambda_n(A)$ denote the complex eigenvalues of $A$ with highest and lowest modulus respectively.
\end{defi}

\begin{rmk}
Since the eigenvalues of matrices in $\SO_0(2,n+1) \subset \SL(n+3,\R)$ are preserved by the involution $A \mapsto A^{-1}$, the above definition simplifies to
\[L_\rho(\gamma) = \log |\lambda_1(\gamma)|\]
for representations into $\SO_0(2,n+1)$.
\end{rmk}

The length spectrum of a representation captures many of its algebraic, geometric and dynamical properties. Several results suggest that the length spectra of Hitchin and maximal representations are somehow always ``bigger'' than that of a Fuchsian representation. The first of these results deals with the ``average behavior'' of the length spectrum.

\begin{defi}
Let $\rho$ be a representation of $\Gamma$ into $\SL(n,\R)$. The \emph{entropy} of $\rho$ is the number
\[h(\rho) = \limsup_{R\to +\infty} \frac{1}{R} \log \sharp \{\gamma \in [\Gamma] \mid L_\rho(\gamma) \leq R\}~.\]
\end{defi}

\begin{theo}[Potrie--Sambarino \cite{PotrieSambarino}]
If $\rho: \Gamma \to \SL(n,\R)$ is a Hitchin representation, then
\[h(\rho) \leq \frac{2}{n-1}~,\]
with equality if and only if $\rho$ is conjugate to $m_{irr} \circ j$, where $j:\Gamma \to \SL(2,\R)$ is a Fuchsian representation and $m_{irr} : \SL(2,\R) \to \SL(n,\R)$ is the irreducible representation.
\end{theo}

Another ``geometric control'' on Hitchin representations is a generalization of the classical \emph{collar lemma} for Fuchsian representations. It roughly says that, if $\gamma$ and $\eta$ are two essentially intersecting curves on $\Sigma,$ then $L_\rho(\gamma)$ and $L_\rho(\eta)$ cannot both be small.
Such a collar lemma was obtained by Lee and Zhang for Hitchin representations into $\SL(n,\R)$ \cite{LeeZhang} and by Burger and Pozzetti \cite{BurgerPozzetti} for maximal representations into $\Sp(2n,\R)$. More precisely, they prove:

\begin{theo}
There exists a constant $C$ such that, for any $\gamma$ and $\eta$ in $[\Gamma]$ represented by essentially intersecting curves on $\Sigma$ and for any Hitchin (resp. maximal) representation $\rho$ of $\Gamma$ into $\SL(n,\R)$ (resp. $\Sp(2n,\R)$), one has
\[\left(e^{L_\rho(\gamma)} -1\right)\cdot \left(e^{L_\rho(\eta)} -1\right) \geq C~.\]
\end{theo}

Motivated by a question of Zhang, the second author proved a stronger statement for Hitchin representations into $\SL(3,\R)$ which implies both results above:

\begin{theo}[Tholozan, \cite{TholozanConvex}]
If $\rho: \Gamma \to \SL(3,\R)$ is a Hitchin representation, then there exists a Fuchsian representation $j: \Gamma \to \SL(2,\R)$ such that
\[L_\rho \geq L_{m_{irr} \circ j}~.\]
\end{theo}

We will prove a similar statement for maximal representations into $\SO_0(2,n+1)$. A maximal representation $\rho: \Gamma \to \SO_0(2,n+1)$ is said to be \textit{in the Fuchsian locus} if $\rho(\Gamma)$ preserves a copy of $\R^{2,1}$ in $\R^{2,n+1}$ (see Definition \ref{d:FuchsianLocus}). 

\begin{MonThm}\label{t:domination}
Let $\Gamma$ be the fundamental group of a closed oriented surface $\Sigma$ of genus at least two. If $\rho: \Gamma \to \SO_0(2,n+1)$ is a maximal representation $(n\geq 0)$, then either $\rho$ is in the Fuchsian locus, or there exists a Fuchsian representation $j: \Gamma \to \SO_0(2,1)$ and a $\lambda >1$ such that
\[L_\rho \geq \lambda L_j~.\]
\end{MonThm}

As a direct consequence of the fact that Fuchsian representations into $\SO_0(2,1)$ have entropy $1$, we obtain the following:
\begin{MonCoro}
Let $\Gamma$ be the fundamental group of a closed oriented surface $\Sigma$ of genus at least two. If $\rho: \Gamma \to \SO_0(2,n+1)$ is a maximal representation $(n\geq 0)$, then the entropy $h(\rho)$ satisfies
\[h(\rho) \leq 1\]
with equality if and only if $\rho$ is in the Fuchsian locus.
\end{MonCoro}

 As a direct consequence of Theorem \ref{t:domination} and Keen's collar lemma \cite{KeenCollarLemma}, we can also deduce a sharp collar lemma for maximal representations into $\SO_0(2,n+1)$:
\begin{MonCoro}
Let $\Gamma$ be the fundamental group of a closed oriented surface $\Sigma$ of genus at least two and $\rho:\Gamma\to\SO_0(2,n+1)$ be a maximal representation. If $\gamma$ and $\eta$ are two elements in $[\Gamma]$ represented by essentially intersecting curves on $\Sigma$, then 
\[\sinh\left(\frac{L_\rho(\gamma)}{2}\right)\cdot \sinh\left(\frac{L_\rho(\eta)}{2}\right) > 1~.\]
\end{MonCoro}

\subsection*{Labourie's conjecture for maximal representations in rank $2$}

A drawback of non-abelian Hodge theory is that it parameterizes representations of a surface group in a way that depends on the choice of a complex structure on the surface. In particular, such parameterizations do not have a natural action of the mapping class group of $\Sigma.$ One would overcome this issue by finding a canonical way to associate a complex structure on the surface to a given surface group representation. To this intent, Labourie \cite{labourieenergy} suggested the following approach.

Let $\Teich(\Sigma)$ denote the Teichm\"uller space of marked complex structures on $\Sigma$. For each reductive representation $\rho$ of $\Gamma$ into a semi-simple Lie group $\mathrm{G},$ one can associate a function on $\Teich(\Sigma)$ called the \emph{energy function}.

\begin{defi}
The {\em energy function} $\E_\rho$ is the function that associates to a complex structure $J$ on $\Sigma$ the energy of the $\rho$-equivariant harmonic map from $(\widetilde{\Sigma}, J)$ to the Riemannian symmetric space $G/K$.
\end{defi}
The existence of such an equivariant harmonic map was proven by Corlette \cite{corlette}. 
By a theorem of Sacks and Uhlenbeck \cite{sacksuhlenbeck} and Schoen and Yau \cite{schoenyau}, $J$ is a critical point of $\E_\rho$ if and only if the $\rho$-equivariant harmonic map from $(\widetilde{\Sigma}, J)$ to $G/K$ is weakly conformal or, equivalently, if its image is a branched minimal surface in $G/K$.
Labourie showed in \cite{labourieenergy} that if the representation $\rho$ is Anosov, then its energy function is proper, and thus admits a critical point. 
He conjectured that, for Hitchin representations, this critical point is unique.

\begin{conj}[Labourie]
Let $\Gamma$ be the fundamental group of a closed oriented surface $\Sigma$ of genus at least two. If $\rho$ is a Hitchin representation of $\Gamma$ into a real split Lie group $G$, then there is a unique complex structure $J\in\Teich(\Sigma)$ on $\Sigma$ such that the $\rho$-equivariant harmonic map from $(\widetilde{\Sigma}, J)$ to $G/K$ is weakly conformal.
\end{conj}

Labourie's conjecture was proven independently by Loftin \cite{loftin} and Labourie \cite{labouriecubic} for $G= \SL(3,\R)$, and then recently by Labourie \cite{labouriecyclic} for other split real Lie groups of rank $2$ (namely, $\PSp(4,\R)$ and $\mathrm{G}_2$). Using the same strategy as Labourie, this was generalized by Alessandrini and the first author \cite{colliercyclic,PSp4maximalRepsAC} to all maximal representations into $\PSp(4,\R)$. Here we give a new proof of their result and extend it to any Hermitian Lie group of rank $2$.

\begin{MonThm} \label{t:LabourieConjecture}
Let $\Gamma$ be the fundamental group of a closed oriented surface $\Sigma$ of genus at least two. If $\rho$ is a maximal representation of $\Gamma$ into a Hermitian Lie group $G$ of rank $2$, then there is a unique complex structure $J\in\Teich(\Sigma)$ such that the $\rho$-equivariant harmonic map from $(\widetilde{\Sigma}, J)$ to $G/K$ is conformal. Moreover, this conformal harmonic map is an embedding.
\end{MonThm}

\begin{rmk} \label{rmk:LieGroupsRank2}
Theorem \ref{t:LabourieConjecture} reduces to a theorem concerning maximal representations into $\SO_0(2,n)$. Indeed, the Hermitian Lie groups of rank $2$ are (up to a cover): $\PU(1,n) \times \PU(1,n)$, $\PSp(4,\R)$, $\PU(2,n)$ and $\SO_0(2,n)$ ($n\geq 5$). By \cite{Toledo}, maximal representations into $\PU(1,n) \times \PU(1,n)$ are conjugate to maximal representations into $\mathrm{P}(\U(1,1) \times \U(n-1)) \times \mathrm{P}(\U(1,1) \times \U(n-1))$. By \cite{burgeriozziwienhard} and \cite{UpqHiggs}, maximal representations into $\PU(2,n)$ are all conjugate to maximal representations into $\mathrm{P}(\U(2,2)\times U(n-2))$ . Finally, $\PU(1,1)\times \PU(1,1)$ is isomorphic to $\PSO_0(2,2)$, $\PSp(4,\R)$ is isomorphic to $\SO_0(2,3)$ and $\PU(2,2)$ is isomorphic to $\PSO_0(2,4)$. 
\end{rmk}

Labourie's conjecture seems to be related to the property of lying in a connected component of Anosov representations. In particular, the conjecture does not hold for quasi-Fuchsian representations. Indeed, Huang and Wang \cite{HuangWang15} constructed quasi-Fuchsian manifolds containing arbitrarily many minimal surfaces. 

\subsection*{Maximal surfaces in $\H^{2,n}$ and strategy of the proof}

Let $\H^{2,n}$ be the space of negative definite lines in $\R^{2,n+1}$. The space $\H^{2,n}$ is an open domain in $\ProjR{n+2}$ on which $\SO_0(2,n+1)$ acts transitively, preserving a pseudo-Riemannian metric of signature $(2,n)$ with constant sectional curvature $-1$. The boundary of $\H^{2,n}$ in $\ProjR{n+2}$ is the space $\Ein^{1,n}$. The cornerstone of all the above results will be the following theorem:

\begin{MonThm} \label{t:ExistenceUniquenessMaximalSurface}
Let $\Gamma$ be the fundamental group of a closed oriented surface $\Sigma$ of genus at least two. If $\rho:\Gamma\to\SO_0(2,n+1)$ is a maximal representation, then there exists a unique $\rho$-equivariant maximal space-like embedding of the universal cover of $\Sigma$ into $\H^{2,n}$.
\end{MonThm}

This theorem generalizes a well-known result of existence of maximal surfaces in some anti-de Sitter $3$-manifolds. More precisely, for $n=1$, maximal representations are exactly the holonomies of globally hyperbolic Cauchy-compact anti-de Sitter $3$-manifolds (see \cite{Mess}). In this particular case, our theorem is due to Barbot, B\'eguin and Zeghib \cite{BBZ} (see also \cite{toulisse} for the case with cone singularities).

The existence part of Theorem \ref{t:ExistenceUniquenessMaximalSurface} will be proven in Section \ref{s:maximalsurface} using Higgs bundle theory. More precisely, we will see that, given a maximal representation $\rho:\Gamma\to\SO_0(2,n+1)$, any critical point of the energy function $\EE_\rho$ gives rise to a $\rho$-equivariant maximal space-like embedding of $\widetilde{\Sigma}$ with the same conformal structure. 
The uniqueness part of Theorem \ref{t:ExistenceUniquenessMaximalSurface} will then directly imply Theorem \ref{t:LabourieConjecture}. Our proof will use the pseudo-Riemannian geometry of $\H^{2,n}$ in a manner similar to \cite{BonsanteSchlenker}. 
\begin{rmk}
The idea of using cyclic Higgs bundles to construct equivariant maps and study geometric structures was first exploited by Baraglia in his thesis \cite{BaragliaThesis}. For the particular case of Hitchin representations into $\SO_0(2,3),$ equivariant maximal surfaces in $\mathbb{H}^{2,2}$ are constructed in \cite[Proposition 3.5.2]{BaragliaThesis} and Baraglia relates those to photon structures on the unit tangent bundle of $\Sigma.$ The existence part of Theorem \ref{t:ExistenceUniquenessMaximalSurface} generalizes these techniques to all maximal $\SO_0(2,n+1)$ cyclic Higgs bundles.
\end{rmk}

We show in Section \ref{ss:extremalsurfaces} that the $\rho$-equivariant minimal surface in the Riemannian symmetric space is the Gauss map of the maximal surface in $\H^{2,n}$. In the case $n=1$, this interpretation recovers the equivalence between the existence of a unique maximal surface in globally hyperbolic anti-de Sitter $3$-manifolds and the result of Schoen \cite{schoen} giving the existence of a unique minimal Lagrangian diffeomorphism isotopic to the identity between hyperbolic surfaces (this equivalence was proven in \cite{krasnovschlenker}).

Now, to each negative definite line $x\in\H^{2,n}$, one can associate a copy of $\Pho(\R^{2,n}) \subset \Pho(\R^{2,n+1})$ defined as the set of photons contained in $x^\perp$. Moreover, the copies of $\Pho(\R^{2,n})$ associated to two such lines $x$ and $y$ are disjoint if and only if $x$ and $y$ are joined by a space-like geodesic. This remark allows us to construct a $\Pho(\R^{2,n+1})$-structure on a fiber bundle over $\Sigma$ from the data of any $\rho$-equivariant space-like embedding of $\widetilde{\Sigma}$ into $\HH^{2,n}$, and hence, prove Theorem \ref{t:photontructures}.

The $\Ein^{1,2}$-structures associated to Hitchin representations into $\SO_0(2,3)$ from Theorem \ref{t:einsteinstructures} are constructed from the unique maximal space-like surface of Theorem~\ref{t:ExistenceUniquenessMaximalSurface} as follows. To each unit tangent vector $v$ of the maximal space-like $\rho$-equivariant embedding of $\widetilde{\Sigma}$ in $\H^{2,2}$, one can associate a point in $\Ein^{1,2} = \partial_\infty \H^{2,2}$ by ``following the geodesic determined by $v$ to infinity''. In this way, one obtains a $\rho$-equivariant map from $T^1 \widetilde{\Sigma}$ to $\Ein^{1,2}$. Using a maximum principle which involves the components of the solution to the self-duality equations, we will prove that this map is a local diffeomorphism. Note that this is specific to Hitchin representations and is not true for other maximal representations.

 Finally, to prove Theorem \ref{t:domination}, we introduce the length spectrum of the maximal $\rho$-equivariant embedding as an intermediate comparison. On the one hand, this length spectrum is larger than the length spectrum of the conformal metric of curvature $-1$ on the maximal surface, and, on the other hand, it is less than the length spectrum of the representation $\rho$. This should be compared to \cite{DeroinTholozan} where Deroin and the second author prove that for any representation $\rho$ into the isometry group of $\H^n$, there exists a Fuchsian representation $j$ such that $L_j\geq L_\rho$. Here, both inequalities are reversed because of the pseudo-Riemannian geometry on $\H^{2,n}$.
 
\begin{rmk}
In the recent paper \cite{DancigerGueritaudKasselPseudoHyperbolic}, Danciger, Gu\'eritaud and Kassel show that many Anosov representations can be seen as acting convex-cocompactly on a pseudo-Riemannian hyperbolic space $\H^{p,q}$. Our maximal representations into $\SO_0(2,n+1)$ are an occurrence of this phenomenon. This gives us hope that pseudo-Riemannian hyperbolic geometry can bring a better understanding of Anosov representations in more generality. 
\end{rmk}

\subsection*{Acknowledgments}
When we started this project, Olivier Guichard and Anna Wienhard very kindly shared their working notes on Einstein structures associated to Hitchin representations with us. In addition, Olivier Guichard carefully read a previous version of this paper and sent us numerous remarks. For this we are very grateful.

The authors gratefully acknowledge support from the NSF grants DMS-1107452, 1107263 and 1107367 “RNMS: GEometric structures And Representation varieties” (the GEAR Network). N. Tholozan's research is partially supported by the ANR project DynGeo (ANR-11-BS01-013). B. Collier's research is supported by the National Science Foundation under Award No. 1604263.  

\section{Maximal representations into $\SO_0(2,n+1)$}

For the rest of the paper, $\Sigma$ will be a closed surface of genus $g\geq2$. We denote by $\Gamma$ its fundamental group and by $\widetilde{\Sigma}$ its universal cover. 
Recall that the group $\Gamma$ is \emph{Gromov hyperbolic} and that its \emph{boundary at infinity}, denoted by $\partial_\infty \Gamma$, is homeomorphic to a circle.

\subsection{The Toledo invariant}

Let $\R^{2,n+1}$ denote the space $\R^{n+3}$ endowed with the quadratic form 
\[\mathbf{q}: (x_1, \ldots ,x_{n+3}) \mapsto x_1^2 + x_2^2 - x_3^2 - \ldots - x_{n+3}^2~.\]
The Lie group $\SO_0(2,n+1)$ is the identity component of the group of linear transformations of $\R^{n+3}$ preserving $\mathbf{q}$. Its subgroup $\SO(2)\times \SO(n+1)$ is a maximal compact subgroup.

To a representation $\rho:\Gamma\to\SO_0(2,n+1)$, one can associate a principal $\SO_0(2,n+1)$-bundle $P_\rho$ whose total space is the quotient of $\widetilde{\Sigma}\times \SO_0(2,n+1)$ by the action of $\Gamma$ by deck transformations:
\[\gamma\cdot (x,y) = (x\cdot\gamma^{-1}, \rho(\gamma)y)~.\]
Since the quotient of $\SO_0(2,n+1)$ by a maximal compact subgroup is contractible, this principal bundle admits a reduction of structure group to a principal $\SO(2)\times \SO(n+1)$-bundle $B_\rho$ which is unique up to gauge equivalence. Finally, the quotient of $B_\rho$ by the right action of $\SO(n+1)$ gives a principal $\SO(2)$-bundle $M_\rho$ on $\Sigma$.

\begin{defi}
The \emph{Toledo invariant} $\tau(\rho)$ of the representation $\rho$ is the Euler class of the $\SO(2)$-bundle $M_\rho$.
\end{defi}

The Toledo invariant is locally constant and invariant by conjugation. It thus defines a map
\[\tau : \xymatrix{\Rep(\Gamma,\SO_0(2,n+1))\ar[r]& \Z},\]
where $\Rep(\Gamma,\SO_0(2,n+1))$ denotes the set of conjugacy classes of representations of $\Gamma$ into $\SO_0(2,n+1)).$
It is proven in \cite{domic} that the Toledo invariant satisfies the \emph{Milnor--Wood inequality}:

\begin{prop}
For each representation $\rho:\Gamma\to\SO_0(2,n+1)$ the Toledo invariant satisfies
\[ |\tau(\rho)| \leq 2g-2~.\]

\end{prop}

This leads to the following definition:

\begin{defi}
A representation $\rho: \Gamma \to \SO_0(2,n+1)$ is \emph{maximal} if $\tau(\rho) = 2g-2$. 
\end{defi}
\begin{rmk}
	The Toledo number $\tau(\rho)\in\Z$ depends
on the identification $\Z\cong H^2(\Sigma,\Z)$ and hence on the orientation of the surface. Reversing the orientation changes the Toledo number to its opposite. For this reason, we focus on the case $\tau(\rho)\geq 0$ and, in particular, $\tau(\rho)=2g-2.$
\end{rmk}

\subsection{Maximal representations are Anosov}

The Toledo invariant and the notion of maximal representations can be defined more generally for representations of $\Gamma$ into Hermitian Lie groups. In \cite{burgeriozziwienhard}, Burger, Iozzi and Wienhard study these representations. They prove in particular that for any Hermitian Lie group $G$ \emph{of tube type}, there exist maximal representations of $\Gamma$ into $G$ that have Zariski dense image. This applies in particular to maximal representations into $\SO_0(2,n+1)$. 

In that same paper, they exhibit a very nice geometric property of maximal representations that was reinterpreted in \cite{BILW} as the \emph{Anosov property} introduced independently by Labourie in \cite{labouriehyperconvex}. Here we describe one of the main consequences of their work in our setting.

Let $\Ein^{1,n} \subset \ProjR{n+2}$ denote the space of isotropic lines in $\R^{2,n+1}$. The group $\SO_0(2,n+1)$ acts transitively on $\Ein^{1,n}$ and preserves the conformal class of a pseudo-Riemannian metric of signature $(1,n)$.
We will say that three isotropic lines $[e_1], [e_2]$ and $[e_3]$ in $\Ein^{1,n}$ are \emph{in a space-like configuration} if the quadratic form $\mathbf{q}$ restricted to the vector space spanned by $e_1$, $e_2$ and $e_3$ has signature $(2,1)$.

\begin{theo}[Burger--Iozzi--Labourie--Wienhard, \cite{BILW}] \label{t:AnosovCurve}
If $\rho: \Gamma \to \SO_0(2,n+1)$ is a maximal representation, then there is a unique $\rho$-equivariant continuous embedding
\[\xi: \partial_\infty \Gamma \to \Ein^{1,n}~.\]
Moreover, the image of $\xi$ is a \emph{space-like curve}, meaning that the images of any three distinct points in $\partial_\infty \Gamma$ are in a space-like configuration.
\end{theo}

Note that the result of Burger, Iozzi, Labourie and Wienhard does not concern directly the case $G=\SO_0(2,n+1)$, but $G=\SU(p,q)$ and $G=\Sp(2n,\R)$. However, as proven by Pozzetti and Hamlet \cite{pozzettihamlet}, there exists a tight homomorphism $\iota: \SO_0(2,n+1) \to \Sp(2m,\R)$ for some $m\in \N$. This property is sufficient to extend the result to the case of $\SO_0(2,n+1)$.

The Anosov property implies that maximal representations are \emph{proximal}. In particular, the limit curve $\xi$ can be reconstructed from the attracting and repelling eigenvectors of $\rho(\gamma)$ for $\gamma \in \Gamma$. More precisely, we have the following:

\begin{coro}[see for instance \cite{BochiPotrieSambarino}] \label{c:AnosovProximal}
For every $\gamma \in \Gamma\backslash\{\mathrm{id}\}$, let $\gamma_+$ and $\gamma_-$ denote the attracting and repelling fixed points of $\gamma$ in $\partial_\infty \Gamma$ and let $\lambda$ denote the spectral radius of $\gamma$. Then $\xi(\gamma_+)$ and $\xi(\gamma_-)$ are the eigen-directions of $\rho(\gamma)$ for the eigenvalues $\lambda$ and $\lambda^{-1}$ respectively. Moreover, the $2$-plane spanned $\xi(\gamma_+)$ and $\xi(\gamma_-)$ is non-degenerate with respect to $\mathbf{q}$, and the restriction of $\rho(\gamma)$ to $\xi(\gamma_-)^\perp \cap \xi(\gamma_+)^\perp$ has spectral radius strictly less than $\lambda$.
\end{coro}

For $n=0$, maximal representations into $\SO_0(2,1)$ correspond to Fuchsian representations \cite{GoldmanTopologicalComponents}. The isometric inclusion 
$$\begin{array}{lll}
~\R^{2,1} & \longrightarrow & \R^{2,n+1} \\
(x_1,x_2,x_3) & \longmapsto & (x_1,x_2,x_3,0,\cdots,0)
\end{array}$$
defines an inclusion $\iota_n: \SO_0(2,1)\hookrightarrow\SO_0(2,n+1)$ which preserves the Toledo invariant. 
In particular, given a Fuchsian representation $j:\Gamma\to\SO_0(2,1)$, the $\SO_0(2,n+1)$-representation $\iota_n\circ j$ is maximal.

If $\alpha:\Gamma\to\mathrm{O}(n)$ is an orthogonal representation, let $\det(\alpha):\Gamma\to\mathrm{O}(1)$ be the determinant representation. One can construct the representation
$$j\otimes \det(\alpha): \Gamma \to \mathrm{O}(2,1),$$
obtained by twisting $j$ by $\det(\alpha)$. More precisely, $j\otimes \det(\alpha)$ takes value in the index two subgroup of $\mathrm{O}(2,1)$ whose maximal compact subgroup is $\SO(2)\times \mathrm{O}(1)$.

\begin{prop}
The maximal representation
\[\rho=\left(j\otimes \det(\alpha)\right)\oplus\alpha:\Gamma\to\mathrm{O}(2,n+1)\]
takes value in $\SO_0(2,n+1)$.
\end{prop}
\begin{proof}
Because $j$ takes value in $\SO_0(2,1)$, one can deform $j(\gamma)$ to the identity in $\SO_0(2,1)$ for any $\gamma\in \Gamma$. In particular, $\rho(\gamma)$ can be deformed to an element in $\SO(2)\times \SO(n+1)\subset \SO_0(2,n+1)$.
\end{proof}

\begin{defi}\label{d:FuchsianLocus}
	A maximal representation $\rho:\Gamma\to\SO_0(2,n+1)$ lies in the {\em Fuchsian locus} if it preserves a three dimensional linear subspace of $\R^{2,n+1}$ in restriction to which $\mathbf{q}$ has signature $(2,1)$;
	equivalently, $\rho$ is in the Fuchsian locus if \[\rho=\big(j\otimes \det(\alpha)\big)\oplus \alpha\]
	for some Fuchsian representation $j:\Gamma\to\SO_0(2,1)$ and some $\alpha:\Gamma\to\mathrm{O}(n)$.
\end{defi}

\subsection{Harmonic metrics and Higgs bundles}\label{s:harmonicmetrics}
We now recall the non-abelian Hodge correspondence between
representations of $\Gamma$ into $\SO_0(2,n+1)$ and $\SO_0(2,n+1)$-Higgs bundles. This correspondence holds for any real reductive Lie group $G$, but we will restrict the discussion to our group of interest.

When the surface $\Sigma$ is endowed with a complex structure, we will denote the associated Riemann surface by $X.$ The canonical bundle of $X$ will be denoted by $\mathcal{K}$ and the trivial bundle will be denoted by $\Oo.$ We also denote the Riemannian symmetric space of $\SO_0(2,n+1)$ by $\mathfrak{X}$, namely
\[\mathfrak{X}=\SO_0(2,n+1)/(\SO(2)\times\SO(n+1)).\]

\begin{rmk}
The symmetric space $\mathfrak{X}$ can be viewed as the space of space-like $2$-planes in $\R^{2,n+1}$ since $\SO_0(2,n+1)$ acts transitively on the space of such $2$-planes, with stabilizer $\SO(2)\times\SO(n+1)$. Equivalently, we can see $\mathfrak{X}$ as the space of positive-definite scalar products on $\R^{2,n+1}$ of the form $\mathbf{q}(\cdot, \sigma \cdot)$, where $\sigma$ is the orthogonal reflection along a space-like $2$-plane.
\end{rmk}

Let us start by recalling the notion of a harmonic metric. 

\begin{defi}
Let $\rho:\Gamma\to\SO_0(2,n+1)$ be a representation and let $P_\rho$ be the associated flat $\SO_0(2,n+1)$-bundle. A {\em metric} on $P_\rho$ is a reduction of structure group to $\SO(2)\times \SO(n+1)$. Equivalently, a metric is a $\rho$-equivariant map 
\[\textbf{h}_\rho:\xymatrix{\widetilde \Sigma\ar[r]&\mathfrak{X}}.\]
\end{defi}

The differential $d\textbf{h}_\rho$ of a (smooth) metric $\textbf{h}_\rho$ is a section of $T^*\widetilde\Sigma \otimes \textbf{h}_\rho^* T\mathfrak{X}$. 
Given a metric $g$ on $\Sigma$, one can define the norm $\Vert d\textbf{h}_\rho \Vert$ of $d\textbf{h}_\rho$ which, by equivariance of $\textbf{h}_\rho$, is invariant under the action of $\Gamma$ on $\widetilde\Sigma$ by deck transformations. 
In particular, $\Vert d\textbf{h}_\rho \Vert$ descends to a function on $\Sigma$. The {\em energy} of $\textbf{h}_\rho$ is the $L^2$-norm of $d\textbf{h}_\rho$, namely:
\[\EE(\textbf{h}_\rho)=\int_\Sigma\Vert d\textbf{h}_\rho\Vert^2dv_g.\]
Note that the energy of $\textbf{h}_\rho$ depends only on the conformal class of the metric $g,$ and so, only on the Riemann surface structure $X$ associated to $g$.

\begin{defi}
A metric $\textbf{h}_\rho: \widetilde X\rightarrow \mathfrak{X}$ on $P_\rho$ is \textit{harmonic} if it is a critical point of the energy functional.  
\end{defi}
The complex structure on $X$ and the Levi-Civita connection on $\mathfrak{X}$ induce a holomorphic structure $\nabla^{0,1}$ on the bundle $\big( T^*X\otimes \textbf{h}_\rho^*T\mathfrak{X}\big)\otimes\C$. The following is classical (see \cite[p. 425]{wood}):
\begin{prop}\label{p-harmonicholomorphic}
	A metric $\textbf{h}_\rho:\widetilde X\rightarrow \mathfrak{X}$ is harmonic if and only if the $(1,0)$ part $\partial \textbf{h}_\rho$ of $d\textbf{h}_\rho$ is holomorphic, that is 
	\[\nabla^{0,1}\partial \textbf{h}_\rho=0.\]
\end{prop}


A representation $\rho:\Gamma\rightarrow \SO_0(2,n+1)$ is {\em completely reducible} if any $\rho(\Gamma)$-invariant subspace of $\R^{n+3}$ has a $\rho(\Gamma)$-invariant complement. For completely reducible representations, we have the following theorem.
\begin{theo}[Corlette \cite{corlette}]\label{t:Corlette}
A representation $\rho:\Gamma\rightarrow \SO_0(2,n+1)$ is completely reducible if and only if, for each Riemann surface structure $X$ on $\Sigma,$ there exists a harmonic metric 
$\textbf{h}_\rho:\widetilde X\rightarrow \mathfrak{X}.$ Moreover, a harmonic metric is unique up to the action of the centralizer of $\rho.$ 
 \end{theo}

 \begin{rmk}
 	In \cite{burgeriozziwienhard}, it is shown that all maximal representations are completely reducible and that the centralizer of a maximal representation is compact. Thus, for maximal representations harmonic metrics are unique. 
 \end{rmk}

For a completely reducible representation $\rho$, the energy of the harmonic metric $\textbf{h}_\rho$ defines a function on the Teichm\"uller space $\Teich(\Sigma)$ of $\Sigma$

\begin{equation}
		\label{eq:Energy on Teich} \EE_\rho:\xymatrix@R=0em{\Teich(\Sigma)\ar[r]&\R\\X\ar@{|->}[r]&\EE(\textbf{h}_\rho)}~.
	\end{equation}	
The critical points of the energy are determined by the following.
 \begin{prop}[Sacks-Uhlenbeck \cite{sacksuhlenbeck}, Schoen-Yau \cite{schoenyau}]\label{p:conformal Branched Minimal Imm}
 A harmonic metric $\textbf{h}_\rho$ is a critical point of $\EE_\rho$ if and only if it is weakly conformal,
 i.e. $\tr(\partial \textbf{h}_\rho \otimes \partial \textbf{h}_\rho)=0.$ This is equivalent to the image of $\textbf{h}_\rho$ being a branched minimal immersion.
\end{prop}

For Anosov representations, Labourie has shown that the energy function $\EE_\rho$ is smooth and proper, and so, has a critical point. As a corollary we have:

\begin{prop}[Labourie \cite{labourieenergy}]\label{p:Labourie Existence}
	For each maximal representation there exists a Riemann surface structure on $\Sigma$ for which the harmonic metric is weakly conformal.
\end{prop}

We now recall the notion of a Higgs bundle on a Riemann surface $X$.



\begin{defi}
	A $\mathrm{GL}(n,\C)$-Higgs bundle on $X$ is a pair $(\Ee,\Phi)$ where $\Ee$ is a rank $n$ holomorphic vector bundle and $\Phi\in H^0(\End(\Ee)\otimes \mathcal{K})$ is a holomorphic endomorphism of $\Ee$ twisted by $\mathcal{K}$. An $\SL(n,\C)$-Higgs bundle on $X$ consists of such a pair $(\Ee,\Phi)$ with $\tr(\Phi)=0$ and a fixed trivialization $\Lambda^n\Ee\cong\Oo.$ 
\end{defi} 
Higgs bundles were originally defined by Hitchin \cite{hitchinselfduality} for the group $\SL(2,\C)$ and generalized by Simpson \cite{simpsonVHS} for any complex semi-simple Lie group. 
More generally, Higgs bundles can be defined for real reductive Lie groups \cite{HiggsPairsSTABILITY}. For the group $\SO_0(2,n+1)$ the appropriate vector bundle definition is the following. 

\begin{defi}
 	An $\SO_0(2,n+1)$-Higgs bundle over a Riemann surface $X$ is a tuple $(\Uu,q_\Uu,\Vv,q_\Vv,\eta)$ where
 	\begin{itemize}
 		\item $\Uu$ and $\Vv$ are respectively rank 2 and rank $(n+1)$ holomorphic vector bundles on $X$ with trivial determinant and trivializations $\Lambda^2\Uu\cong \Oo,~\Lambda^{n+1}\Vv\cong\Oo$.
 		\item $q_\Uu$ and $q_\Vv$ are non-degenerate holomorphic sections of $\Sym^2(\Uu^*)$ and $\Sym^2(\Vv^*)$,
 		\item $\eta$ is a holomorphic section of $\Hom(\Uu,\Vv)\otimes \mathcal{K}.$ 
 	\end{itemize}   
 \end{defi}
The non-degenerate sections $q_\Uu$ and $q_\Vv$ define holomorphic isomorphisms 
\[\xymatrix{q_\Uu:\Uu\to\Uu^*&\text{and}&q_\Vv:\Vv\to\Vv^*}.\]
	Given an $\SO_0(2,n+1)$-Higgs bundle $(\Uu,q_\Uu,\Vv,q_\Vv,\eta)$, we get an $\SL(n+3,\C)$-Higgs bundle $(\Ee,\Phi)$ by setting $\Ee=\Uu\oplus\Vv$ and 
	\begin{equation}\label{eq:SL(n+3,C)HiggsBundle}
		\Phi=\mtrx{0&\eta^\dagger\\\eta&0}:\Uu\oplus\Vv\longrightarrow (\Uu\oplus\Vv)\otimes \mathcal{K}.
	\end{equation}
	where $\eta^\dagger=q_{\Uu}^{-1}\circ\eta^T\circ q_\Vv\in H^0(\Hom(\Vv,\Uu)\otimes \mathcal{K}).$ 
	Note that 
	\[\Phi^T\mtrx{q_\Uu&\\&-q_\Vv}+\mtrx{q_{\Uu}&\\&-q_\Vv}\Phi=0.\]
	
Appropriate notions of poly-stability exist for $\mathrm{G}$-Higgs bundles \cite{HiggsPairsSTABILITY}. However, for our considerations, the following definition will suffice.

\begin{defi}
Let $(\Ee,\Phi)$ be a $\mathrm{GL}(n,\C)$-Higgs bundle with $\deg(\Ee)=0$ or an $\SL(n,\C)$-Higgs bundle. Then $(\Ee,\Phi)$ is called 
\begin{itemize}
	\item stable if for all proper sub-bundles $\Ff\subset\Ee$ with $\Phi(\Ff)\subset\Ff\otimes \mathcal{K}$ we have $\deg(\Ff)<0.$
	\item polystable if it is direct sum of stable $\mathrm{GL}(n_j,\C)$-Higgs bundles $(\Ee_j,\Phi_j)$ with $\deg(\Ee_j)=0$ for all $j$. 
\end{itemize}
%
	An $\SO_0(2,n+1)$-Higgs bundle is poly-stable if and only if the $\SL(n+3,\C)$-Higgs bundle \eqref{eq:SL(n+3,C)HiggsBundle} is poly-stable. 
\end{defi}
\noindent\textbf{From Higgs bundles to representations.} Poly-stability is equivalent to existence of a Hermitian metric solving certain gauge theoretic equations which we refer to as the {\em self-duality equations}. 
This was proven by Hitchin \cite{hitchinselfduality} for $\SL(2,\C)$ and Simpson \cite{simpsonVHS} for semi-simple complex Lie groups, see \cite{HiggsPairsSTABILITY} for the statement for real reductive groups.

We say that a Hermitian metric $h$ on $\mathcal{E}$ is \emph{adapted} to the $\C$-bilinear symmetric form $q$ if $h(u,v)=q(u,\lambda(v))$ where $\lambda: \mathcal{E} \to \mathcal{E}$ is an anti-linear involution. In such a case, we say that $\lambda$ is the involution associated to the metric $h$.
\begin{theo}
\label{t:Hitchin-Simpson} 
	An $\SO_0(2,n+1)$-Higgs bundle $(\Uu,q_\Uu,\Vv,q_\Vv,\eta)$ is poly-stable if and only if there exist adapted Hermitian metrics $h_\Uu$ and $h_\Vv$ on $\Uu$ and $\Vv$ such that 
	\begin{equation}
 	\label{eq: SO(2,n+1) Higgs bundle Equations}
 	\begin{cases}
 	F_{h_\Uu}+\eta^\dagger\wedge(\eta^\dagger)^{*_h}+\eta^{*_h}\wedge\eta=0 \\
 	F_{h_\Vv}+\eta\wedge\eta^{*_h}+(\eta^\dagger)^{*_h}\wedge \eta^\dagger=0~.
 	\end{cases}
 \end{equation} 
Here $F_{h_\Uu}$ and $F_{h_\Vv}$ denote the curvature of the Chern connections of $h_\Uu$ and $h_\Vv$ and $\eta^{*_h}$ denotes the Hermitian adjoint of $\eta$, i.e. $h_\Vv(u,\eta(v))= h_\Uu(\eta^{*_h}(u),v).$
\end{theo}

If $(h_\Uu,h_\Vv)$ solves the self-duality equations \eqref{eq: SO(2,n+1) Higgs bundle Equations}, then the metric $h=h_\Uu\oplus h_\Vv$ on $\Ee=\Uu\oplus\Vv$ solves the $\SL(n+3,\C)$-self-duality equations
\[
	F_h+[\Phi,\Phi^{*h}]=0.\]
Given a solution $(h_\Uu,h_\Vv)$ to the self-duality equations, the connection
\begin{equation}\label{eq:flat conn of Higgs bundle}
	\nabla=\mtrx{\nabla_{h_\Uu}&\\&\nabla_{h_\Vv}}+\mtrx{0&\eta^\dagger\\\eta&0}+\mtrx{0&\eta^{*_h}\\(\eta^\dagger)^{*_h}&0}
\end{equation}
is a {\em flat} connection on $\Ee=\Uu\oplus\Vv$. Moreover, if $\lambda_\Uu$ and $\lambda_\Vv$ are the associated involutions, $\lambda_\Uu\oplus\lambda_\Vv$ is preserved by $\nabla$.

Denote the associated real bundle by $E_\nabla.$ The orthogonal structure $q_\Uu\oplus -q_\Vv$ restricts to a $\nabla$-parallel signature $(2,n+1)$ metric $g_\Uu\oplus g_\Vv$ on $E_\nabla.$ The holonomy of $\nabla$ gives a representation $\rho: \Gamma \to \SO_0(2,n+1)$ which is completely reducible.

\medskip

\noindent\textbf{From representations to Higgs bundles.} Let $(E_\rho,\nabla,g)$ be the flat rank $(n+3)$ vector bundle with signature $(2,n+1)$ metric $g$ and flat connection $\nabla$ associated to a representation $\rho:\Gamma\to\SO_0(2,n+1)$. A metric on $E_\rho$, 
\[\textbf{h}_\rho:\xymatrix{\widetilde \Sigma\ar[r]&\mathfrak{X}}\]
is equivalent to an orthogonal splitting $E_\rho=U\oplus V$ where $U$ is a rank $2$ orthogonal bundle with $g_U=g|_U$ positive definite and $V$ is a rank $(n+1)$-bundle with $-g_V=-g|_V$ positive definite. 
Moreover, the flat connection $\nabla$ decomposes as
\begin{equation}\label{eq:flatconnection decomp}
	\nabla=\mtrx{\nabla_U&\\&\nabla_V}+\mtrx{&\Psi^\dagger\\\Psi&}
\end{equation}
where $\nabla_U$ and $\nabla_V$ are connections on $U$ and $V$ such that $g_U$ and $g_V$ are covariantly constant, $\Psi$ is a one form valued in the bundle $\Hom(U,V)$ and $\Psi^\dagger=g_U^{-1}\Psi^T g_V$. 

Recall that a point $x\in\mathfrak{X}$ corresponds to an orthogonal splitting $\R^{2,n+1}=\R^{2,0}\oplus\R^{0,n+1}$. In this splitting, a tangent vector $v\in T_x\mathfrak{X}$ is given by an endomorphism 
\[\mtrx{0&a^\dagger\\a&0}:\R^{2,0}\oplus\R^{0,n+1}\to\R^{2,0}\oplus\R^{0,n+1}~.\] 
Since the differential $d\textbf{h}_\rho$ is an equivariant section of $T^*\widetilde\Sigma \otimes \textbf{h}_\rho^* T\mathfrak{X}$, it descends to a one form valued in the bundle $\Hom(U,V)\oplus \Hom(V,U)$ of the form $\alpha+\alpha^\dagger.$
In fact, the differential $d\textbf{h}_\rho$ of the metric is identified with $\Psi+\Psi^\dagger$ (see \cite[Lemma 9.13]{GuichardNUS} for more details).

If $X$ is a Riemann surface structure on $\Sigma$, then the Hermitian extension $h_U\oplus h_V$ of $g_U\oplus-g_V$ to the complexification of $E_\rho$ defines a Hermitian metric. The complex linear extensions of $\nabla_U,\nabla_V,\Psi$ and $\Psi^\dagger$ all decompose into $(1,0)$ and $(0,1)$ parts, and $\nabla_U^{0,1}$ and $\nabla_V^{0,1}$ define holomorphic structures. Writing $\nabla_{U,V}^{0,1}$ for the $(0,1)$-part of the connection on $\Hom(U,V)$ induced by the connections $\nabla_U$ and $\nabla_V$, Proposition~\ref{p-harmonicholomorphic} reads:

\begin{prop}
A metric $\textbf{h}_\rho:\widetilde X\to \mathfrak{X}$ is harmonic if and only if $\nabla_{U,V}^{0,1}\Psi^{1,0}=0,$ (or equivalently $\nabla_{V,U}^{0,1}(\Psi^\dagger)^{1,0}).$
\end{prop}

Given a harmonic metric $\textbf{h}_\rho,$ the Hermitian adjoints of $\Psi^{1,0}$ and $(\Psi^{\dagger})^{0,1}$ are given by $(\Psi^{1,0})^{*}=(\Psi^\dagger)^{0,1}$ and $(\Psi^{\dagger})^{1,0}=(\Psi^{0,1})^*$.
With respect to a harmonic metric, the flatness equations $F_\nabla=0$ decompose as
\begin{equation}\label{eq:flatness EQ harmonic metric}
\left\{\begin{array}{l}
F_{\nabla_U}+\Psi^{1,0}\wedge(\Psi^{1,0})^*+((\Psi^{\dagger})^{1,0})^{*}\wedge (\Psi^{\dagger})^{1,0} = 0 \\
F_{\nabla_V}+(\Psi^{\dagger})^{1,0}\wedge((\Psi^{\dagger})^{1,0})^{*}+(\Psi^{1,0})^{*}\wedge\Psi^{1,0}=0 \\
\nabla_{U,V}^{0,1}\Psi^{1,0}=0
\end{array}\right..
\end{equation}

	Note that setting $\Psi^{1,0}=\eta,$ the self-duality equations \eqref{eq: SO(2,n+1) Higgs bundle Equations} are the same as the decomposition of the flatness equations \eqref{eq:flatness EQ harmonic metric} with respect to a harmonic metric. 
	Thus, if $\Uu$ and $\Vv$ are the holomorphic bundles $(U\otimes \C,\nabla_U^{0,1})$ and $(V\otimes \C,\nabla_V^{0,1}),$ then $(\Uu,q_\Uu,\Vv,q_\Vv,\Psi^{1,0})$ is a poly-stable $\SO_0(2,n+1)$-Higgs bundle, where $q_\Uu$ is the $\C$-linear extension of $g_\U$ to $U\otimes \C$ (similarly for $q_\Vv$).



\begin{prop}\label{p:Minimal Imm Tr(Phi2)=0}
	Let $\rho:\Gamma\to\SO_0(2,n+1)$ be a completely reducible representation and $X$ be a Riemann surface structure on $\Sigma.$
	If $(\Uu,q_\Uu,\Vv,q_\Vv,\eta)$ is the Higgs bundle associated to $\rho$, then the harmonic metric $\textbf{h}_\rho$ is weakly conformal if and only if $\tr(\eta\otimes\eta^\dagger) =0.$
\end{prop}
\begin{proof}
	The derivative of the harmonic metric is identified with the $1$-form $\Psi+\Psi^\dagger$ from \eqref{eq:flatconnection decomp}. By Proposition \ref{p:conformal Branched Minimal Imm}, $\textbf{h}_\rho$ is weakly conformal if and only if 
	\[\tr\left(\mtrx{0&(\Psi^{\dagger})^{0,1}\\\Psi^{0,1}&0}^2\right)=0.\] 
	This is equivalent to $\tr(\eta\otimes\eta^\dagger)=0.$
\end{proof}
\begin{defi}
An $\SO_0(2,n+1)$-Higgs bundle $(\Uu,q_\Uu,\Vv,q_\Vv,\eta)$ will be  called \emph{conformal} if $\tr(\eta\otimes\eta^\dagger)=0.$
\end{defi}
\subsection{Maximal Higgs bundle parameterizations}\label{Higgsbundleparametrization}
We now describe the Higgs bundles which give rise to maximal $\SO_0(2,n+1)$-representations.

\begin{prop}
The isomorphism class of a $\SO_0(2,n+1)$-Higgs bundle $(\Uu,q_\Uu,\Vv,q_\Vv,\eta)$ is determined by the data $(\Ll,\Vv,q_\Vv,\beta,\gamma)$ where $\Ll$ is a holomorphic line bundle on $X$, $\beta\in H^0(\Ll\otimes\Vv\otimes \mathcal{K})$ and $\gamma\in H^0(\Ll^{-1}\otimes \Vv\otimes \mathcal{K}).$ Here $\Uu=\Ll\oplus \Ll^{-1}$, $q_\Uu = \left(\begin{array}{ll} 0 & 1 \\ 1 & 0 \end{array}\right)$ and $\eta = (\gamma,\beta) : \mathcal{L}\oplus \mathcal{L}^{-1}\to \Vv\otimes \mathcal{K}$. Moreover, if $(\Uu,q_\Uu,\Vv,q_\Vv,\eta)$ is poly-stable, then the Toledo invariant of the corresponding representation is the degree of $\Ll.$ 
\end{prop}
\begin{proof}
The group $\SO(2,\C)$ is isomorphic to the set of $2\times2$ matrices $A$ such that $\det(A)=1$ and $A^T\mtrx{0&1\\1&0}A=\mtrx{0&1\\1&0}$. Such matrices are given by $\mtrx{\lambda&0\\0&\lambda^{-1}}$ for $\lambda\in\C^*.$ Since the bundle $(\Uu,q_\Uu)$ is the associated bundle of a holomorphic principal $\SO(2,\C)$-bundle and the action of $\SO(2,\C)$ on $\C^2$ preserves two lines, up to isomorphism we have 
\[(\Uu,q_\Uu)=\left(\Ll\oplus\Ll^{-1},\mtrx{0&1\\1&0}:\Ll\oplus\Ll^{-1}\to(\Ll\oplus\Ll^{-1})^*\right).\] 
With respect to the splitting $\Uu=\Ll\oplus\Ll^{-1}$, the holomorphic section $\eta\in\Hom(\Uu,\Vv)\otimes \mathcal{K}$ decomposes as $\beta\oplus\gamma$ where $\beta\in \Hom(\Ll\otimes \Vv\otimes \mathcal{K})$ and $\gamma\in \Hom(\Ll^{-1}\otimes\Vv\otimes \mathcal{K})$. 
Using the standard isomorphism $\SO(2)\cong\mathrm{U}(1)$ given by $e^{i\theta}$ and $\C^*\cong\SO(2,\C)$ given by $\lambda\mapsto \mtrx{\lambda&\\&\lambda^{-1}}$, one can see that the degree of $\Ll$ is the degree of the $\SO(2)$-bundle whose complexification is $\Uu.$ Thus, the Toledo invariant of the associated representation is the degree of $\Ll$. 
\end{proof}
\begin{rmk}
The $\SL(n+3,\C)$-Higgs bundle $(\Ee,\Phi)$ associated to $(\Ll,\Vv,q_\Vv,\beta,\gamma)$
is given by $\Ee=\Ll\oplus\Ll^{-1}\oplus\Vv$ and
\begin{equation}\label{eq:betagammaHiggsfield}
	\Phi=\mtrx{0&0&\beta^\dagger\\0&0&\gamma\dagger\\\gamma&\beta&0}:\Ee\to\Ee\otimes \mathcal{K}.
\end{equation} 
\end{rmk}
The Milnor-Wood inequality can be seen directly for poly-stable Higgs bundles.
\begin{prop}\label{p:maximal Higgs bundle Param}
If $(\Ll,\Vv,q_\Vv,\beta,\gamma)$ is a poly-stable $\SO_0(2,n+1)$-Higgs bundle, then $\deg(\Ll)\leq 2g-2$. Furthermore, if $\deg(\Ll)=2g-2,$ then 
\begin{itemize}
	\item $\Vv$ admits a $q_\Vv$-orthogonal decomposition $\Vv=\Ii\oplus\Vv_0$ where $\Vv_0$ is a holomorphic rank $n$ bundle and $\Ii=\Lambda^n\Vv_0$ satisfies $\Ii^2=\Oo.$
	\item $\Ll\cong \mathcal{IK}$  
	\item $\gamma\cong\mtrx{1\\0}:\Ii\mathcal{K}\to \Ii \mathcal{K}\oplus \Vv_0\otimes \mathcal{K}$ and $\beta=\mtrx{q_2\\\beta_0}:\mathcal{K}^{-1}\Ii\to \Ii \mathcal{K}\oplus\Vv_0\otimes \mathcal{K}$ where $q_2\in H^0(\mathcal{K}^2)$ and $\beta_0\in H^0(\mathcal{K}\otimes\Ii\otimes\Vv_0).$
\end{itemize}
\end{prop}

\begin{proof}
The poly-stable $\SL(n+3,\C)$ Higgs bundle $(\Ee,\Phi)$ associated to $(\Ll,\Vv,q_\Vv,\beta,\gamma)$ 
has $\Ee=\Ll\oplus\Ll^{-1}\oplus\Vv$ and $\Phi$ is given by \eqref{eq:betagammaHiggsfield}. 
If $\deg(\Ll)>0$, then by poly-stability $\gamma\neq0.$ If the image of $\gamma$ is isotropic, then we have a sequence
\[\xymatrix{&0\ar[r]&\Ll \mathcal{K}^{-1}\ar[r]^\gamma&\ker(\gamma^\dagger)\ar[r]&\Vv\ar[r]^{\gamma^\dagger}&\Ll^{-1}\mathcal{K}\ar[r]&0}.\]
In particular, this implies that $\deg(\ker(\gamma^\dagger))=\deg(\Ll)-(2g-2)$. Since $\Ll\oplus \ker(\gamma^\dagger)$ is a $\Phi$-invariant sub-bundle, we have
$\deg(\Ll)\leq g-1.$ 
Thus, for $\deg(\Ll)>g-1$ the composition $\gamma\circ\gamma^\dagger$ is a non-zero element of $H^0((\Ll^{-1}\mathcal{K})^2)$, and we conclude that $\deg(\Ll)\leq 2g-2.$ 

If $\deg(\Ll)=2g-2,$ then $(\Ll^{-1}\mathcal{K})^2=\Oo$ and $\gamma$ is nowhere vanishing. Set $\Ii=\Ll \mathcal{K}^{-1}$, then $\Ll=\Ii \mathcal{K}$ and $\Ii$ defines an orthogonal line sub-bundle of $\Vv$. 
Taking the $q_\Vv$-orthogonal complement of $\Ii$ gives a holomorphic decomposition $\Vv=\Ii\oplus(\Ii)^\perp$. 
Since $\Lambda^{n+1}\Vv=\Oo,$ we conclude $\Vv=\Ii\oplus\Vv_0$ where $\Ii=\Lambda^n\Vv_0.$
Since the image of $\gamma$ is identified with $\Ii,$ we can take $\gamma\cong\mtrx{1\\0}:\mathcal{IK}\to \Ii \mathcal{K}\oplus \Vv_0\otimes \mathcal{K}$. 
Finally, the holomorphic section $\beta$ of $\Hom(\Ii \mathcal{K}^{-1},\Ii\oplus\Vv_0)\otimes \mathcal{K}$ decomposes as
\[\beta= q_2\oplus \beta_0\] 
where $q_2$ is a holomorphic quadratic differential and $\beta_0\in H^0(\Vv_0\otimes\Ii \mathcal{K})$
\end{proof}
\begin{rmk}
Higgs bundles with $\deg(\Ll)=2g-2$ will be called maximal Higgs bundles. They are determined by tuples $(\Vv_0,q_{\Vv_0},\beta_0,q_2)$ from Proposition \ref{p:maximal Higgs bundle Param}.
\end{rmk}

\begin{prop}
	If $\rho:\Rep^{max}(\Gamma,\SO_0(2,n+1))$ is a maximal representation, $X$ is a Riemann surface structure on $\Sigma$ and the Higgs bundle corresponding to $\rho$ is defined by the data $(\Vv_0,q_{\Vv_0},\beta_0,q_2),$ then the harmonic metric
	is a conformal immersion if and only if the holomorphic quadratic differential $q_2$ vanishes. 
\end{prop}
\begin{proof}
By Proposition \ref{p:Minimal Imm Tr(Phi2)=0}, the harmonic metric associated to a poly-stable Higgs bundle $(\Uu,q_\Uu,\Vv,q_\Vv,\eta)$ is weakly conformal if and only if $\tr(\eta\otimes\eta^\dagger)=0.$ For a maximal Higgs bundle determined by $(\Vv_0,q_{\Vv_0},\beta_0,q_2)$
\[\eta=\mtrx{1&0\\q_2&\beta_0}:\Ii \mathcal{K}\oplus \Ii \mathcal{K}^{-1}\to\Ii \mathcal{K}\oplus\Vv_0 \mathcal{K}\ \ \ \ \ \text{and}\ \ \ \ \ \eta^\dagger=\mtrx{ q_2&\beta_0^\dagger\\1&0}:\Ii\oplus\Vv_0\to\Ii \mathcal{K}^2\oplus \Ii.\]
A computation shows $\tr(\eta\otimes\eta^\dagger)=2q_2$, thus, by Proposition \ref{p:Minimal Imm Tr(Phi2)=0}, the harmonic map is weakly conformal if and only if $q_2=0.$  
Finally, $\eta+\eta^\dagger$ is nowhere vanishing, hence the harmonic metric has no branch point.   
\end{proof}
\begin{rmk}
For any Hermitian group $G$, the Higgs bundles associated to maximal representations always have a nowhere vanishing Higgs field (see for example \cite{MaxRepsHermSymmSpace}). Thus, by the argument above, every weakly conformal harmonic metric associated to a maximal representation is an immersion, i.e. it is unbranched. 
\end{rmk}

Given a maximal representation $\rho\in\Rep(\Gamma,\SO_0(2,n+1)),$ by Proposition \ref{p:Labourie Existence}, we can always find a Riemann surface structure in which the corresponding Higgs bundle is a maximal conformal Higgs bundle. A maximal conformal Higgs bundle is determined by $(\Vv_0,q_{\Vv_0},\beta_0)$:
\[(\Uu,q_\Uu,\Vv,q_\Vv,\eta)=\left(\mathcal{IK}\oplus \Ii\mathcal{K}^{-1},\ \mtrx{0&1\\1&0},\ \Ii\oplus\Vv_0,\ \mtrx{1&0\\0&q_{\Vv_0}},\ \mtrx{1&0\\0&\beta_0}\right).\]
The associated $\SL(n+3,\C)$-Higgs bundle will be represented schematically by 
	\[
	\xymatrix@R=-.2em{\Ii\mathcal{K}\ar[r]^1&\Ii\ar[r]^1&\Ii\mathcal{K}^{-1}\ar[ddl]^{\beta_0}\\&\oplus&\\&\Vv_0\ar[uul]^{\beta_0^\dagger}&}\]
where the arrows represent the Higgs field and we omit the tensor product by $\mathcal{K}$. Such a Higgs bundle is an example of a {\em cyclic} Higgs bundle. 
\begin{defi} \label{d:CyclicHiggsBundle}
An $\SL(n,\C)$-Higgs bundle $(\Ee,\Phi)$ is called \emph{cyclic of order $k$} if there is a holomorphic splitting $\Ee = \Ee_1\oplus \ldots \oplus \Ee_k$ such that $\Phi$ maps $\Ee_i$ into $\Ee_{i+1}\otimes \mathcal{K}$ (for $i < k$) and $\Ee_k$ to $\Ee_1\otimes \mathcal{K}$. 
\end{defi}

\begin{prop}[Simpson, \cite{KatzMiddleInvCyclicHiggs}]\label{p:CyclicOrthogonalSplitting}
If the Higgs bundle $(E,\Phi)$ is cyclic of order $k$, then the cyclic splitting of $\Phi$ is orthogonal with respect to the Hermitian metric $h$ which solves the self-duality equations $F_h+[\Phi,\Phi^{*h}]=0$.
\end{prop}
The symmetries of the solution metrics \eqref{eq: SO(2,n+1) Higgs bundle Equations} and Proposition \ref{p:CyclicOrthogonalSplitting} give a further simplification of the self-duality equations for maximal conformal $\SO_0(2,n+1)$-Higgs bundles.
\begin{prop}\label{p:MaximalConformalHiggsbundleEQ}
	For a poly-stable maximal conformal $\SO_0(2,n+1)$-Higgs bundle determined by $(\Vv_0,q_{\Vv_0},\beta_0)$, if $(h_\Uu,h_\Vv)$ solves the self-duality equations \eqref{eq: SO(2,n+1) Higgs bundle Equations}, then 
	\begin{itemize}
	 	\item $h_\Uu=\mtrx{h_{\Ii \mathcal{K}}&\\&h_{\Ii \mathcal{K}}^{-1}}$ where $h_{\Ii \mathcal{K}}$ is a metric on $\Ii \mathcal{K}$ and $h_{\Ii \mathcal{K}}^{-1}$ is the induced metric on $\Ii \mathcal{K}^{-1}$
	 	\item $h_{\Vv}=\mtrx{h_\Ii&\\&h_{\Vv_0}}$ where $h_\Ii$ is a flat metric on $\Ii$ and $h_{\Vv_0}$ is a metric on $\Vv_0$ adapted to $q_{\Vv_0}$.
	 \end{itemize} 
	 Furthermore, the self-duality equations \eqref{eq: SO(2,n+1) Higgs bundle Equations} simplify 
	 \begin{equation*}
	 	\label{eq:CycliHiggsBundleEQ}
	 	\left\{\begin{array}{l}
	 	F_{h_{\Ii \mathcal{K}}}+\beta_0^\dagger\wedge (\beta_0^\dagger)^{*_h}+1^{*_h}\wedge 1=0 \\
	 	F_{\Vv_0}+\beta_0\wedge\beta_0^{*_h}+(\beta_0^\dagger)^{*_h}\wedge\beta_0^\dagger=0
	 	\end{array}\right.
	 \end{equation*}
 \end{prop}
\begin{proof}
	Because $h_\Uu$ is adapted to $q_\Uu$, $h_\Uu=\mtrx{h_{\Ii \mathcal{K}}&\\&h_{\Ii \mathcal{K}}^{-1}}$ where $h_{\Ii \mathcal{K}}$ is a metric on $\Ii \mathcal{K}$ and $h_{\Ii \mathcal{K}}^{-1}$ is the induced metric on $\Ii \mathcal{K}^{-1}$.
	The splitting of $h_\Vv$ follows from Proposition \ref{p:CyclicOrthogonalSplitting}.
	The self-duality equations \eqref{eq: SO(2,n+1) Higgs bundle Equations} with $\eta=\mtrx{1&0\\0&\beta_0}$ and $h_\Uu=h_{\Ii \mathcal{K}}\oplus h_{\Ii \mathcal{K}}^{-1}$ and $h_\Vv=h_\Ii\oplus h_{\Vv_0}$ simplify to
\[	\left\{ \begin{array}{l}
	F_{h_{\Ii \mathcal{K}}}+\beta_0^\dagger\wedge (\beta_0^\dagger)^{*_h}+1^{*_h}\wedge 1=0\\
	F_{h_{\Ii \mathcal{K}}^{-1}}+1\wedge1^{*_h}+\beta_0^{*_h}\wedge \beta_0=0 \\
	F_{h_\Ii}+1\wedge1^{*_h}+1^{*_h}\wedge 1=0 \\
    F_{\Vv_0}+\beta_0\wedge\beta_0^{*_h}+(\beta_0^\dagger)^{*_h}\wedge\beta_0^\dagger=0
	\end{array}\right.\]
	Note that the first two equations are the same and the third equation implies the metric $h_{\Ii}$ is flat.
\end{proof}
For a poly-stable maximal conformal $\SO_0(2,n+1)$-Higgs bundle determined by $(\Vv_0,q_{\Vv_0},\beta_0)$, the flat connection \eqref{eq:flat conn of Higgs bundle} on $\Ee=\Ii\mathcal{K}\oplus \Ii\mathcal{K}^{-1}\oplus\Ii\oplus\Vv_0$ decomposes as 
\begin{equation}
	\label{eq flat conn hol decomp}\mtrx{\nabla_{h_{\Ii\mathcal{K}}}&0&1^{*_h}&\beta_0^\dagger\\0&\nabla_{h_{\Ii\mathcal{K}^{-1}}}&1&\beta_0^{*_h}\\1&1^{*_h}&\nabla_{h_{\Ii}}&0\\(\beta_0^\dagger)^{*_h}&\beta_0&0&\nabla_{h_{\Vv_0}}}~.
\end{equation}
The associated flat bundle $E_\rho\subset (\Ii \mathcal{K}\oplus \Ii \mathcal{K}^{-1}\oplus \Ii\oplus\Vv_0)$ is the fixed point locus of the associated anti-linear involution $\lambda: \mathcal{E} \to \mathcal{E}$, that is the involution defined by the equation $h(u,v)=q(u,\lambda(v))$.

In the splitting $\mathcal{E}= \mathcal{IK}\oplus \mathcal{IK}^{-1}\oplus \Ii \oplus \Vv_0$, the $\mathbb{C}$-bilinear form $q$ is given by
\begin{equation}
	\label{eq:q on Ee}q=\mtrx{&1&&\\1&&&\\&&-1&\\&&&-q_{\Vv_0}}~.
\end{equation}
By Proposition \ref{p:MaximalConformalHiggsbundleEQ}, the previous splitting is orthogonal with respect to the Hermitian metric solving the self-duality equations. In particular, one easily checks that the associated involution $\lambda: \Ii \mathcal{K}\oplus \Ii \mathcal{K}^{-1}\oplus \Ii\oplus\Vv_0 \to \Ii \mathcal{K}\oplus \Ii \mathcal{K}^{-1}\oplus \Ii\oplus\Vv_0$ is written
\[\lambda=\mtrx{&h_{\Ii \mathcal{K}}^{-1}&&\\h_{\Ii \mathcal{K}}&&&\\ &&-h_\Ii&\\&&&-\lambda_{\Vv_0}},\]
where $h_{\mathcal{IK}}(u)$ is the anti-linear map defined by $h_{\mathcal{IK}}(u).v=h_{\mathcal{IK}}(u,v)$.

Thus the flat bundle $E_\rho=U\oplus V$ of a maximal representation decomposes further. This decomposition will play an essential role in the rest of the paper. 

\begin{theo}\label{p:decompositionbundle}
The flat bundle associated to a poly-stable maximal conformal $\SO_0(2,n+1)$-Higgs bundle determined by  $(\Vv_0,q_{\Vv_0},\beta_0)$ decomposes orthogonally as
\[E_\rho=U\oplus\ell\oplus V_0\] where $U\subset\Uu$ is a positive definite rank two sub-bundle, 
	 $\ell\subset\Ii$ is a negative definite line sub-bundle and $V_0\subset \Vv_0$ is a negative definite rank $n$ bundle. Using the decomposition of \eqref{eq flat conn hol decomp}, in this splitting the flat connection is given by
	 \[\nabla=\mtrx{\nabla_{h_\Uu}&1+1^{*_h}&\beta_0^\dagger+\beta_0^{*_h}\\1+1^{*_h}&\nabla_{h_\Ii}&0\\\beta_0+(\beta_0^\dagger)^{*_h}&0&\nabla_{h_{\Vv_0}}}~.\]
       \end{theo}


The following consequence will be useful in the next subsection:
\begin{coro} \label{c:nablal}
Let $u$ be a section of $\ell$ such that $q(u) = -1$. Then $\nabla u$ gives a ($\R$-linear) isomorphism between $T X$ and $U$.
\end{coro}

\begin{proof}
Since the connection $\nabla_{h_\mathcal{I}}$ on $\mathcal{I}$ preserves $\ell$ and $q_{\vert \ell}$, $u$ is parallel for $\nabla_{h_\mathcal{I}}$. Thus $\nabla u = (1+1^{*_h})u$ which is a $1$-form with values in $U$. The notation ``$1$'' means that (by definition) the $(1,0)$-part of this $1$-form never vanishes, implying that $\nabla u : T X \to U$ is injective at every point. It is thus an isomorphism.
\end{proof}

\begin{rmk}
We will see in Section \ref{ss:extremalsurfaces} that $\nabla u : T X \to (U,q_{\vert U})$ is conformal. 
\end{rmk}

From now on we will only consider poly-stable maximal $\SO_0(2,n+1)$-Higgs bundles. For notational convenience, we will drop the subscript $0$ and write the decomposition of the flat bundle $E_\rho$ as $E_\rho=U\oplus \ell\oplus V.$

\subsection{Connected components of maximal representations}
Given a maximal $\SO_0(2,n+1)$-Higgs bundle 
\begin{equation}\label{eq: SO(2,n+1) Higgs}
	\xymatrix@R=0em{\Ii\mathcal{K}\ar[r]^1&\Ii\ar[r]^1&\Ii\mathcal{K}^{-1}\ar[ddl]^{\beta}\\&\oplus&\\&\Vv\ar[uul]^{\beta^\dagger}&},
\end{equation}
the Stiefel-Whitney classes $sw_1\in H^1(\Sigma,\Z/2)$ and $sw_2\in H^2(\Sigma,\Z/2)$ of $\Vv$ define characteristic classes which help distinguish the connected components of maximal Higgs bundles. 
Thus, the space of maximal representations decomposes as 
\[\Rep^{max}(\Gamma,\SO_0(2,n+1))=\bigsqcup_{\substack{sw_1\in H^1(\Sigma,\Z/2)\\sw_2\in H^2(\Sigma,\Z/2)}} \Rep^{max}_{sw_1,sw_2}(\Gamma,\SO_0(2,n+1))\]
where $\Rep^{max}_{sw_1,sw_2}(\Gamma,\SO_0(2,n+1))$ is the set of maximal representations such that the Stiefel-Whitney classes of the bundle $\Vv$ are $sw_1$ and $sw_2$. 

\begin{rmk}\label{rmk: Fuchsian higgs bundle}
	Note that when $n=0,$ maximal conformal $\SO_0(2,1)$-Higgs bundles are given by 
	\[\xymatrix{\mathcal{K}\ar[r]^1&\mathcal O\ar[r]^1&\mathcal{K}^{-1}}~.\] 
Since Fuchsian representations are exactly the maximal $\SO_0(2,1)$-representations, this is the maximal conformal $\SO_0(2,1)$-Higgs bundle associated to a Fuchsian representation.
\end{rmk}

When $n>2,$ these characteristic classes distinguish the connected components of maximal $\SO_0(2,n+1)$-Higgs bundles. In other words each of the sets $\Rep^{max}_{sw_1,sw_2}(\Gamma,\SO_0(2,n+1))$ is non-empty and connected \cite{MaxRepsHermSymmSpace}. Thus, for $n>2,$ the space $\Rep^{max}(\Gamma,\SO_0(2,n+1))$ has $2^{2g+1}$ connected components. This is proven in two steps. First one shows that for $n>2$ every component of moduli space of maximal $\SO_0(2,n+1)$-Higgs bundles contains points of the form \eqref{eq: SO(2,n+1) Higgs} with $\Vv$ a poly-stable orthogonal bundle and $\beta=0$. Then one uses the the fact that for $n>2,$ the components of the moduli space of poly-stable orthogonal bundles are labeled by the first and second Stiefel-Whitney class of the bundle \cite{ramanathan_1975,AndrePGLnR}. 
\begin{prop}\label{p:Fuchsian Locus n>2}
	For $n>2$ each connected component of maximal $\SO_0(2,n+1)$-representations contains a point in the Fuchsian locus from Definition \ref{d:FuchsianLocus}.
\end{prop}
\begin{proof}
Let $j:\Gamma\to\SO_0(2,1)$ be a Fuchsian representation and $\alpha:\Gamma\to\mathrm{O}(n)$ be an orthogonal representation. Consider the maximal $\SO_0(2,n+1)$-representation 
\[\rho=\big(j\otimes \det(\alpha)\big)\oplus \alpha\]
in the Fuchsian locus. By Remark \ref{rmk: Fuchsian higgs bundle}, the associated conformal Higgs bundle is given by 
\[\xymatrix@R=0em{\Ii\mathcal{K}\ar[r]^1&\Ii\ar[r]^1&\Ii\mathcal{K}^{-1}\\&\oplus&\\&\Vv&},\]
where $\Vv$ is the flat orthogonal bundle associated to the representation $\alpha$ and $\Ii=\det(\Vv)$ is the flat orthogonal bundle associated to the representation $\det(\alpha).$ 
By the above discussion, for $n>2$ every component of the moduli space of maximal $\SO_0(2,n+1)$-Higgs bundles contains Higgs bundles this form. Thus, each connected component of maximal $\SO_0(2,n+1)$-representations contains a point in the Fuchsian locus when $n>2.$
\end{proof}

The case of maximal $\SO_0(2,3)$-representations is slightly different. Namely, when the first Stiefel-Whitney class of $\Vv$ vanishes, the structure group of $\Vv$ reduces to $\SO(2).$ In this case, $\Vv$ is isomorphic to $\mathcal{N}\oplus\mathcal{N}^{-1}$ for some line bundle $\mathcal{N}$ with non-negative degree. 
Furthermore, the holomorphic section $\beta$ decomposes as $\beta=(\mu,\nu)\in H^0(\mathcal{N}^{-1}\mathcal{K}^2)\oplus H^0(\mathcal{N}\mathcal{K}^2).$ 
By stability, if $\deg(\mathcal{N})\geq0$, then $\mu\neq0$. Thus, we have a bound $0\leq \deg(\mathcal{N})\leq 4g-4.$ In terms of the diagram \eqref{eq: SO(2,n+1) Higgs}, these conformal Higgs bundles are given by 
\begin{equation}\label{eq:GothenHiggsdiagram}
	\xymatrix@R=-.2em{\mathcal{K}\ar[r]^1&\mathcal{O}\ar[r]^1&\mathcal{K}^{-1}\ar@/^/[ddddl]^\mu\ar@/^/[ddl]_\nu\\&\oplus&\\&\mathcal{N}\ar@/^/[uul]_\mu&\\&\oplus&\\&\mathcal{N}^{-1}\ar@/^/[uuuul]^\nu&},
\end{equation}

The \emph{Hitchin component} is the connected component of the representation variety $\Rep(\Gamma,\SO_0(2,3))$ containing the representations of the form $\iota_{irr}\circ j$ where $j: \Gamma \to \SO_0(2,1)$ is Fuchsian and $\iota_{\irr}: \SO_0(2,1) \to \SO_0(2,3)$ is the unique (up to conjugation) irreducible representation. This component is maximal and corresponds to the case $\deg(\mathcal{N})=4g-4$, which implies $\mathcal{N}=\mathcal{K}^2$ and $\mu$ is nowhere vanishing.


\begin{prop}\cite{MaxRepsHermSymmSpace}\label{p:Connected Components Maximal SO(2,3)}
The space of maximal $\SO_0(2,3)$-representations decomposes as
 \[\bigsqcup\limits_{sw_1\neq 0,\ sw_2}\Rep^{max}_{sw_1,sw_2}(\Gamma,\SO_0(2,3))\ \sqcup\bigsqcup_{0\leq d\leq 4g-4}\Rep_{d}^{max}(\Gamma,\SO_0(2,3)).\]
Here the Higgs bundles corresponding to representations in $\Rep^{max}_{sw_1,sw_2}(\Gamma,\SO_0(2,3))$ are given by \eqref{eq: SO(2,n+1) Higgs} with Stiefel-Whitney classes of $\Vv$ given by $sw_1$ and $sw_2$ and, for representations in $\Rep^{max}_{ d}(\Gamma,\SO_0(2,3)),$ the corresponding Higgs bundles have $\Vv=\mathcal{N}\oplus\mathcal{N}^{-1}$ with $\deg(\mathcal{N})=d.$
Moreover, each of the above spaces is connected. 
\end{prop} 

\begin{rmk}\label{r: Gothen v nonGothen components}
The components $\Rep_{d}^{max}(\Gamma,\SO_0(2,3))$ are the $\SO_0(2,3)$-versions of maximal $\Sp(4,\R)$-representations discovered by Gothen \cite{gothen}. Hence, we will call the $4g-4$ components 
$\bigsqcup\limits_{0<d\leq 4g-4}\Rep^{max}_{d}(\Gamma,\SO_0(2,3))$ \textit{Gothen components}. In particular, Hitchin representations are Gothen representations corresponding to $d=4g-4$. The remaining components 
\[\bigsqcup\limits_{sw_1\neq 0,\ sw_2}\Rep^{max}_{sw_1,sw_2}(\Gamma,\SO_0(2,3))\ \sqcup\ \Rep^{max}_{0}(\Gamma,\SO_0(2,3))\] will be called {\em reducible components}. The name is justified by Proposition \ref{p:Fuchsian Locus n=2}.
\end{rmk}

The Gothen components and the reducible components have important differences. 
In particular, the Gothen components are smooth, and all representations in Gothen components that are not Hitchin representations are Zariski dense \cite{bradlow,collierthesis}, while all reducible components contain representations in the Fuchsian locus. Thus we have:
\begin{prop}\label{p:Fuchsian Locus n=2}
The reducible components of maximal $\SO_0(2,3)$-representations are exactly the components containing representations in the Fuchsian locus from Definition \ref{d:FuchsianLocus}. 
\end{prop}

\section{Maximal space-like surfaces in $\H^{2,n}$}\label{s:maximalsurface}

In this section, we consider the action of a maximal representation $\rho: \Gamma \to \SO_0(2,n+1)$ on the pseudo-Riemannian symmetric space $\H^{2,n}$. We show that this action preserves a unique maximal space-like surface, the Gauss map of which gives a minimal surface in the Riemannian symmetric space $\mathfrak{X}$ of $\SO_0(2,n+1)$. As a corollary, we prove Labourie's conjecture for maximal $\SO_0(2,n+1)$ representations (see Theorem \ref{t:LabourieConjecture} of the introduction).

\subsection{The space $\H^{2,n}$} \label{ss:H2n}

In this section, we recall without proofs some classical facts about the pseudo-Riemannian symmetric spaces $\H^{2,n}$.

Recall that $\R^{2,n+1}$ denotes the space $\R^{n+3}$ endowed with the quadratic form
\[\mathbf{q}: (x_1, \ldots ,x_{n+3}) \mapsto x_1^2 + x_2^2 - x_3^2 - \ldots - x_{n+3}^2~.\]
Throughout the paper, $\scal{\cdot}{\cdot}$ denotes the symmetric bilinear pairing associated to $\mathbf q$ and the symbol $\perp$ refers to orthogonality with respect to $\mathbf q$.

\begin{defi}
The space $\H^{2,n} \subset \ProjR{n+3}$ is the set of lines in $\R^{2,n+1}$ in restriction to which the quadratic form $\mathbf{q}$ is negative. The space $\hat{\H}^{2,n}$ is the set of vectors $u$ in $\R^{2,n+1}$ such that $\mathbf{q}(u) = -1$.
\end{defi}

The natural projection from $\hat{\H}^{2,n}$ to $\H^{2,n}$ is a covering of degree $2$. The tangent space to $\hat{\H}^{2,n}$ is the hyperplane $x^\perp$, and  the restriction of $\mathbf{q}$ to that tangent space induces a pseudo-Riemannian metric on $\H^{2,n}$ of signature $(2,n)$ and sectional curvature $-1$. The group $\SO_0(2,n+1)$ acts transitively on $\H^{2,n}$ preserving this pseudo-Riemannian metric.

\begin{rmk}
The space $\H^{2,1}$ is a Lorentz manifold called the \emph{anti-de Sitter space} of dimension $3$. Some of the results presented in this section generalize known results for $\H^{2,1}$ (see \cite{BonsanteSchlenker}). Note however that the Lie group $\SO_0(2,2)$ is isomorphic to a two-to-one cover of $\PSL(2,\R) \times \PSL(2,\R)$, thus the case $n=1$ is quite special.
\end{rmk}

\noindent\textbf{Compactification.} The boundary of $\H^{2,n}$ in $\ProjR{n+2}$ is the space of isotropic lines in $\R^{2,n+1}$:

\begin{defi}
The Einstein space $\Ein^{1,n} \subset \ProjR{n+2}$ is the set of isotropic lines in $\R^{2,n+1}$. The space $\hat{\Ein}^{1,n}$ is the quotient of the space of isotropic vectors in $\R^{2,n+1}$ by the action of $\R_{>0}$ by homotheties.
\end{defi}

The space $\Ein^{1,n}$ has a natural conformal class of pseudo-Riemannian metrics with signature $(1,n)$ which is invariant by the action of $\SO_0(2,n+1)$. It is thus the local model for conformally flat Lorentz manifolds.

\medskip

\noindent\textbf{Geodesics.} The complete geodesics of $\H^{2,n}$ are the intersections of $\H^{2,n}$ with projective lines. These geodesics fall into three categories:
\begin{itemize}
\item \emph{space-like geodesics} are intersections of $\H^{2,n}$ with projective lines (corresponding to planes) of signature $(1,1)$,
\item \emph{light-like geodesics} are intersections of $\H^{2,n}$ with projective lines of (degenerate) signature $(0,1)$,
\item \emph{time-like geodesics} are intersections of $\H^{2,n}$ with projective lines of signature $(0,2)$.
\end{itemize}

Let $u$ and $v$ be two vectors in $\R^{2,n+1}$ such that $\mathbf{q}(u) = \mathbf{q}(v) = -1$ and $v \neq \pm u$. Then the projections $[u]$ and $[v]$ of $u$ and $v$ in $\H^{2,n}$ are joined by a unique geodesic, which is the intersection of $\H^{2,n}$ with the projective line associated to the span of $u$ and $v$. If this geodesic is space-like, then one can define the space-like distance $d_{\H^{2,n}}([u],[v])$ between $[u]$ and $[v]$ as the length of the geodesic segment joining them. Though this function is not an actual distance, it will be useful later on.
\begin{prop}[see \cite{GlorieuxMonclair}, Proposition 3.2] \label{p:Distancespace-likePoints}
The points $[u]$ and $[v]$ are joined by
\begin{itemize}
 	\item a space-like geodesic if and only if $\left|\langle u,v\rangle\right| >1$,
 	\item a time-like geodesic if and only if $\vert \langle u,v\rangle\vert<1.$
 \end{itemize} 
Moreover, when $\left|\langle u,v\rangle\right| >1$ we have
\[d_{\H^{2,n}}([u],[v])=\cosh^{-1}\left|\langle u,v\rangle\right|~.\]
\end{prop}

\noindent\textbf{Warped product structure.} It is sometimes useful to picture $\hat{\H}^{2,n}$ as a product of a $\H^2\times \S^n$ endowed with a "twisted" metric. To do so, consider an orthogonal splitting $\R^{2,n+1}= E\oplus F$, where $E$ is a space-like $2$-plane and $F$ is its orthogonal. We use the notation $\norm{\cdot}$ to denote both the square root of $\mathbf q$ on $E$ and the square root of $-\mathbf q$ on $F$, so that $\norm{\cdot}$ is a euclidean norm on both $E$ and $F$. 

\begin{prop} \label{p:WarpedProduct}
Let $\D$ be the open disc of radius $1$ in $E$, and $\S^n$ the sphere of radius $1$ in $F$. 
\begin{itemize}
\item[(a)] The map
\[\function{\Psi}{\D \times \S^n}{\hat{\H}^{2,n}}{(u,v)}{\frac{2}{1-\norm{u}^2} u + \frac{1+\norm{u}^2}{1-\norm{u}^2} v}\]
is a homeomorphism.
\item[(b)] We have
\[ \Psi^*\hat g_{\H^{2,n}} = \frac{4}{(1-\norm{u}^2)^2} g_\D \oplus - \left( \frac{1+ \norm{u}^2}{1-\norm{u}^2}\right)^2 g_{\S^n}~,\]
where $g_\D$ is the flat euclidean metric on $\D$ and $g_{\S^n}$ is the spherical metric on $\S^n$.
\item[(c)] The map 
\[\function{\partial_\infty \Psi}{\partial \D \times \S^n}{\hat{\Ein}^{1,n}}{(u,v)}{u+ v}\]
is a homeomorphism that extends $\Psi$ continuously.
\end{itemize}
\end{prop}

\begin{proof}
This presumably well-known result is essentially a straightforward computation. We sketch the proof here for completeness.
\begin{itemize}
\item[(a)] Let $w$ be a vector in $\hat{\H}^{2,n}$. Write $w= u'+v'$, with $u'\in E$ and $v'\in F$. Since the map $r\mapsto \frac{2r}{1-r^2}$ is a diffeomorphism from $[0,1)$ to $\R_+$, there is a unique $u\in \D$ such that $u' = \frac{2}{1-\norm{u}^2} u$. Take $v= \frac{1-\norm{u}^2}{1+\norm{u}^2}v'$. One easily verifies that
\[\mathbf q (w) = -1 \Leftrightarrow \norm{v} =1~. \\ \]

\item[(b)] Let $(\dot u, \dot v)$ be a tangent vector to $\D \times \S^n$ at $(u,v)$. We view $\dot{v}$ as a vector in $\R^{n+1}$ orthogonal to $v$. We then have
\[d \Psi(\dot u, \dot v) = \left(\frac{2}{1-\norm{u}^2}\dot u + \frac{4 \scal{\dot u}{u}}{(1-\norm{u}^2)^2} u\right) + \left(\frac{1+\norm{u}^2}{1-\norm{u}^2} \dot v + \frac{4 \scal{\dot u}{u}}{(1-\norm{u}^2)^2} v\right)~.\]
Hence,
\begin{eqnarray*}
\Psi^*\hat g_{\H^{2,n}}(\dot u, \dot v) & = & \mathbf q\left( d \Psi(\dot u,\dot v)\right) \\
&=& \norm{\frac{2}{1-\norm{u}^2}\dot u + \frac{4 \scal{\dot u}{u}}{(1-\norm{u}^2)^2} u}^2 - \norm{\frac{1+\norm{u}^2}{1-\norm{u}^2} \dot v + \frac{4 \scal{\dot u}{u}}{(1-\norm{u}^2)^2} v}^2 \\
&=&\frac{4\norm{\dot u}^2}{(1-\norm{u}^2)^2} - \left(\frac{1+\norm{u}^2}{1-\norm{u}^2}\right)^2 \norm{\dot v}^2~.\\
\end{eqnarray*}

\item[(c)] Let $w$ be a non-zero isotropic vector in $\R^{2,n+1}$. Write $w= u'+v'$, with $u'\in E$ and $v'\in F$. Then $\norm{u'}^2 - \norm{v'}^2= 0$ and there is a unique positive scalar $\lambda$ such that $\lambda w= u+v$ with $\norm{u}^2 = \norm{v}^2 =1$. This proves that $\partial_\infty \Psi: \partial \D \times \S^n\to \hat{\Ein}^{1,n}$ is a homeomorphism. Moreover, if $(u_n,v_n) \in \D\times \S^n$ converges to $(u,v)\in \partial \D \times \S^n$, then
\[\frac{1-\norm{u_n}^2}{2}\Psi(u_n,v_n) \tend{n\to +\infty}u+v~,\]
proving that $\partial_\infty \Psi$ extends $\Psi$ continuously.
\end{itemize}
\end{proof}

\begin{rmk}
These coordinates depend on the choice of the orthogonal splitting $\R^{2,n+1} = E\oplus F$ which can be arbitrary. We will make extensive use of that latitude in the choice of the splitting later on.
\end{rmk}

\subsection{Space-like surfaces} \label{ss:SpacelikeSurface}
We recall here some basic facts about space-like surfaces in $\H^{2,n}$.
An immersed surface $S\subset \H^{2,n}$ of class $\mathcal C^1$ is \emph{space-like} if the restriction of the pseudo-Riemannian metric $g_{\H^{2,n}}$ to $TS$ is positive definite. It is called \emph{complete} if the Riemannian metric induced by $g_{\H^{2,n}}$ is complete. We will prove the following:

\begin{lem} \label{l:GeometrySpacelikeSurface}
Let $S$ be a complete connected immersed space-like surface in $\H^{2,n}$. Then 
\begin{itemize}
\item[(a)] $S$ is embedded and diffeomorphic to an open disc,
\item[(b)] the boundary $\partial S$ of $S$ in $\ProjR{n+3}$ is homeomorphic to a circle and contained in $\Ein^{1,n}$.
\item[(c)] Any two distinct points $x\in S$ and $y\in S\cup \partial S$ are joined by a space-like geodesic.
\end{itemize}
\end{lem}

The key of Lemma \ref{l:GeometrySpacelikeSurface} is that an immersed space-like surface can be written as the graph of a Lipschitz map with respect to the warped product coordinates described in Proposition \ref{p:WarpedProduct}. Let $S$ be a complete connected immersed space-like surface in $\H^{2,n}$. Let $\hat{S}$ denote the inverse image of $S$ by the projection from $\hat{\H}^{2,n}$ to $\H^{2,n}$.

\begin{prop} \label{p:SLipschitzGraph}
The inverse image $\hat{S}$ of $S$ is embedded and has at most two connected components diffeomorphic to discs. Moreover, if we identify $\hat{\H}^{2,n}$ with $\D \times \S^n$ as in Proposition \ref{p:WarpedProduct}, then each of these connected components identifies with the graph of a Lipschitz map from $\D$ to $\S^n$.
\end{prop}

\begin{rmk}
We will see in the proof of Lemma \ref{l:GeometrySpacelikeSurface}  that $\hat{S}$ indeed has two connected components and deduce that $S$ itself is homeomorphic to a disc.
\end{rmk}

\begin{proof}
Denote by $g_{\H^2}$ the complete hyperbolic metric $\frac{4}{(1-\norm{u}^2)^2} g_\D$ on $\D$, and let $\pi: \hat S \to \D$ be the projection on the first factor.
By Proposition \ref{p:WarpedProduct}, the metric $\hat g_{\H^{2,n}}$ expressed in the warped product coordinates has negative vertical part and horizontal part equal to $g_{\H^2}$. We thus have
\[\pi^* g_{\H^2} \geq \hat g_{\H^{2,n}}~,\]
where $\hat g_{\H^{2,n}}$ is the metric induced on $\hat S$. Since the restriction of $\hat g_{\H^{2,n}}$ is assumed to be complete, $\pi^* g_{\H^2}$ is also a complete Riemannian metric on $\hat{S}$. It follows that $\pi: \hat{S} \to \H^2$ is a proper immersion, hence a covering. Since $\H^2$ is simply connected and $S$ is connected, $\hat{S}$ has at most $2$ connected components diffeomorphic to discs. 

Let $\hat{S}_0$ be one of the connected components of $\hat{S}$. Since the projection $\hat{S}_0$ to $\D$ is a diffeomorphism, $\hat{S}_0$ is the graph of a $\mathcal{C}^1$ map $f: \D \to \S^n$. In particular, $\hat{S}_0$ is embedded. For every $u \in \D$ and every $\dot{u}\in T_u\D$, we have
\[ \frac{4}{\left(1-\norm{u}^2\right)^2} \norm{\dot{u}}^2 - \left(\frac{1+\norm{u}^2}{1-\norm{u}^2}\right)^2\norm{d f_u(\dot{u})}^2 >0\]
since $\hat{S}_0$ is space-like. Therefore,
\begin{equation}\label{LipschitzControl}\norm{d f_u(\dot{u})} < \frac{2\norm{\dot{u}}}{1+ \norm{u}^2} \leq 2\norm{\dot{u}} \end{equation}
and $f$ is $2$-Lipschitz.
\end{proof}

Note that one can choose the identification of $\hat{\H}^{2,n}$ with $\D \times \S^n$ so that $\{0\} \times \S^n$ is the intersection of $\hat{\H}^{2,n}$ with any given negative definite linear subspace of $\R^{2,n+1}$ of dimension $n+1$. One thus obtains the following corollary:

\begin{coro} \label{c:IntersectionSpacelikeSurface}
Any negative definite subspace of $\R^{2,n+1}$ of dimension $n+1$ intersects $S$ exactly once.
\end{coro}

Let $\hat{S}_0$ be a connected component of $\hat{S}$. By Proposition \ref{p:SLipschitzGraph}, $\hat{S}_0$ is the graph of a Lipschitz map $f: \D \to \S^n$. The map $f$ thus extends to a continuous map $\partial f: \partial \D \to \S^n$ and the boundary of $\hat{S}_0$ is the graph of $\partial f$ (seen as a subset of $\hat{\Ein}^{1,n}$). In particular, it is a topological circle, and so is its projection to $\Ein^{1,n}$.

\begin{lem} \label{l:HyperplaneSeparatesS}
Let $x$ be a point in $S$. Then $S \cup \partial_\infty S$ does not intersect $x^\perp$.
\end{lem}

\begin{proof}
Let $\hat{x}$ be a lift of $x$ in $\hat{\H}^{2,n}$ and $\hat{S}_0$ the lift of $S$ containing $\hat{x}$. Since the space $\hat{\H}^{2,n}$ is homogeneous, we can choose an identification of $\hat{\H}^{2,n}$ with $\D \times \S^n$ so that $\hat{x}$ is identified to the point $(0,v_0)$ for some $v_0 \in \S^n$.

Let $f:\overline{\D} \to \S^n$ be such that $\hat{S}_0 \cup \partial_\infty \hat{S}_0$ is the graph of $f$. In particular, we have $f(0) = v_0$. For every $u$ in $\overline{\D}$, we have
\begin{eqnarray*}
d_{\S^n}(f(u), v_0) & \leq & \int_0^1 \norm{\dt f(tu)} d t \\
 & < & \int_0^1 \frac{2\norm{u}}{1+\norm{u}^2 t^2}dt \\
 & < & 2 \arctan(\norm{u}) ~\leq~ \frac{\pi}{2}~,
\end{eqnarray*}
where in the second line we used Equation \eqref{LipschitzControl}. 
Since points orthogonal to $f(u)$ are at a distance $\frac{\pi}{2}$ in $\S^n$, $v_0$ is not orthogonal to $f(u)$, and we conclude that the point with coordinates $(u, f(u))$ in $\hat{\H}^{2,n} \cup \hat{\Ein}^{1,n}$ is not in the orthogonal of $\hat{x}$. Since this is true for any $u\in \overline{\D}$ and since $\hat{S}_0 \cup \partial_\infty \hat{S}_0$ is the graph of $f$, this concludes the proof of the lemma.
\end{proof}

We can now conclude the proof of Lemma \ref{l:GeometrySpacelikeSurface}.
\begin{proof}[Proof of Lemma \ref{l:GeometrySpacelikeSurface}]
The projection from $\hat{S} \cup \partial_\infty \hat{S}$ to $S\cup \partial_\infty S$ is a covering of degree $2$. Let $x$ be a point in $\hat{S}$. Then the function from $\hat{S} \cup \partial_\infty \hat{S}$ to $\{-1,1\}$ associating to $y$ the sign of $\scal{x}{y}$ is a well-defined continuous function by Lemma~\ref{l:HyperplaneSeparatesS}. Since $\scal{x}{-y} = - \scal{x}{y}$, this function takes both possible values, and hence $\hat{S} \cup \partial_\infty \hat{S}$ has two connected components.

The covering of degree $2$ from $\hat{S} \cup \partial_\infty \hat{S}$ to $S\cup \partial_\infty S$ is thus a trivial covering. Since each connected component of $\hat{S} \cup \partial_\infty \hat{S}$ is homeomorphic to a closed disc, so is $S$. Properties (a) and (b) follow.

Finally, let $x$ be a point in $S$ and $y$ a point in $S\cup \partial_\infty S$. Let $\hat{x}$ and $\hat{y}$ be lifts of $x$ and $y$ belonging to the same component of $\hat{S} \cup \partial_\infty \hat S$. Identify $\hat{\H}^{2,n} \cup \hat{\Ein}^{1,n}$ with $\overline{\D} \times \S^n$ as in Proposition \ref{p:WarpedProduct} so that $\hat{x}$ is identified with the point $(0,v_0)$ and $y$ is identified with the point $(u, v)$. Following the computation in the proof of Lemma \ref{l:HyperplaneSeparatesS}, we get that $d_{\S^n}(v_0,v) < 2\arctan(\norm{u})$. Thus $\scal{v_0}{v} > \cos\left(2\arctan(\norm{u})\right) = \frac{1- \norm{u}^2}{1+\norm{u}^2}$ and
\begin{eqnarray*}
\scal{\hat x}{\hat y} & = & -\frac{1+\norm{u}^2}{1-\norm{u}^2} \scal{v_0}{v} \\
& < & -1~.
\end{eqnarray*}
Hence $x$ and $y$ are joined by a space-like geodesic by Proposition \ref{p:Distancespace-likePoints}.
\end{proof}

\subsection{The normal bundle, second fundamental form and Gauss maps}\label{subsection-Gaussmaps} In this section, we recall the definition of several classical differential geometric objects associated to space-like immersions into a pseudo-Riemannian manifold, with an emphasis on the case of $\H^{2,n}$.

Let $(M,g)$ be a pseudo-Riemannian manifold of signature $(p,q)$ with $p\geq 2$, $\rho$ be a representation of $\Gamma$ into $\Isom(M,g)$ and $u: \widetilde{\Sigma} \to M$ be a $\rho$-equivariant space-like immersion. 
The action of $\Gamma$ on $\widetilde{\Sigma}$ lifts to an action on the bundle $u^*TM$, and the  quotient by this action gives rise to a bundle on $\Sigma$ that we still denote $u^*TM$. This bundle is naturally endowed with the pull-back $u^*g$ of the pseudo-Riemannian metric of $M$.

The differential of $u$ embeds the tangent bundle $T \Sigma$ as a space-like sub-bundle of $(u^*TM,u^*g)$, which we denote by $T^u$. The \emph{normal bundle} of the immersion $u$ is by definition the orthogonal of $T^u$ in $(u^*TM, u^*g)$. We denote it by $N^u$. Finally, we denote the restrictions of $u^*g$ to $T^u$ and $N^u$ by $g_T$ and $g_N$ respectively, and we denote the pull-back of the Levi-Civita connection of $M$ by $\nabla$. 

Let $X$ be a vector field on $\Sigma$, $Y$ be a section of $T^u$ and  $\xi$ be a section of $N^u$. Since $d u$ identifies $T\Sigma$ with $T^u$ we can alternatively view $Y$ as a vector field on $\Sigma$. Then the decomposition of $\nabla$ along $T^u$ and $N^u$ takes the form
$$\left\{ \begin{array}{lll}
\nabla_X Y & = & \nabla^T_X Y + \mathrm{II}(X,Y) \\
\nabla_X \xi & = & -B(X,\xi) + \nabla^N_X \xi~,
\end{array} \right.$$
where $\nabla^T$ is the Levi-Civita connection of $(\Sigma,g_T)$, $\nabla^N$ is a connection on $N^u$ preserving $g_N$, $\mathrm{II}\in \Omega^1(\Sigma,\Hom(T^u,N^u))$ is the \emph{second fundamental form} and $B\in\Omega^1(S,\Hom(N^u,T^u))$ is called the \textit{shape operator}.

Since $\nabla$ is torsion-free, the second fundamental form is symmetric, namely, $\mathrm{II}\in \Omega^0(\Sym^2(TS)^*\otimes N^u)$. Moreover, the second fundamental form and the shape operator are dual to each other:
$$g_N\big(\mathrm{II}(X,Y),\xi\big) = g_T\big(B(X,\xi),Y\big)~. \\$$

Let us now turn to the case where $M = {\H}^{2,n}$. Let $\pi: \Gg(\H^{2,n}) \to \H^{2,n}$ be the fiber bundle whose fiber over $x\in \H^{2,n}$ is the set of oriented positive definite $2$-planes in $T_x\H^{2,n}$. Since $T_x\H^{2,n}$ is identified with $x^\bot$, a point in $\Gg(\H^{2,n})$ is the same thing as an orthogonal splitting $\R^{2,n+1}= T \oplus l \oplus N$, where $T$ is an oriented positive definite $2$-plane, $l$ is a negative definite line, and $N=(l\oplus T)^\bot$ is a negative definite subspace of dimension $n$. We call $\Gg(\H^{2,n})$ the \textit{main Grassmannian}.

The group $\SO_0(2,n+1)$ acts transitively on $\Gg(\H^{2,n})$ and the stabilizer of a point is conjugated to the subgroup  $\text{H}:=\SO(2)\times \text{S}(\mathrm{O}(1)\times \mathrm{O}(n))$. Hence, the main Grassmannian is identified with the reductive homogeneous space $\SO_0(2,n+1)/\text{H}.$

The injections $\iota_1: \text{H}\hookrightarrow \text{S}(\text{O}(2,1)\times\text{O}(n))$ and $\iota_2: \text{H} \hookrightarrow \SO(2)\times \SO(n+1)$ define projections 
\[\xymatrix{p_1: \Gg(\H^{2,n}) \to \Gr_{(2,1)}\big(\R^{2,n+1} \big)&\text{and }&p_2: \Gg(\H^{2,n}) \to \Gr_{(2,0)}\big(\R^{2,n+1}\big)~,}\]
where $\Gr_{(2,1)}\big(\R^{2,n+1} \big)\cong \SO_0(2,n+1)/ \text{S}(\text{O}(2,1)\times\text{O}(n))$ and $\Gr_{(2,0)}\big(\R^{2,n+1}\big)= \SO_0(2,n+1)/\SO(2)\times\SO(n+1)$ are respectively the  Grassmannian of signature $(2,1)$ and $(2,0)$ linear subspaces of $\R^{2,n+1}$.\\

Let $\rho$ be a representation of $\Gamma$ into $\SO_0(2,n+1)$ and $u:\widetilde{\Sigma} \to \H^{2,n}$ be a $\rho$-equivariant space-like immersion. At every point $x$ in $\widetilde \Sigma$, the tangent and normal spaces of $u$  define a $\rho$-equivariant splitting of $\R^{2,n+1}$ as $T^u_x \oplus u(x) \oplus N^u_x$.

\begin{defi}\label{d-Gaussmaps}
The \textit{main Gauss map} of a $\rho$-equivariant space-like immersion $u:\widetilde{\Sigma} \hookrightarrow \H^{2,n}$ is the $\rho$-equivariant map
\[\begin{array}{llll}\Gg: & \widetilde{\Sigma}  & \longrightarrow & \Gg(\H^{2,n}) \\
& x & \longmapsto & T^u_x\oplus u(x)\oplus N^u_x \end{array}.\]
The \textit{first} and \textit{second Gauss map} are respectively defined to be $G_1= p_1\circ \mathcal{G}$ and $G_2=p_2 \circ \mathcal{G}$.
\end{defi}

Since the main Gauss map is $\rho$-equivariant, it induces a reduction of structure group of the flat principal $\SO_0(2,n+1)$-bundle associated to $\rho$ to an $\textrm{H}$-bundle. This reduction is given by the splitting of the flat $\R^{2,n+1}$-bundle associated to $\rho$ as
\[T^u \oplus l^u \oplus N^u,\]
where $l^u_x=u(x)$. In particular, the Toledo invariant of $\rho$ is given by the Euler class of the $\SO(2)$-bundle $T^u$. Since $d u$ identifies $T\Sigma$ with $T^u$, we obtain the following:
\begin{prop} \label{p:SpacelikeImmersion=>Maximal}
Let $\rho$ be a representation of $\Gamma$ into $\SO_0(2,n+1)$. If there exists a $\rho$-equivariant space-like immersion of $\widetilde{\Sigma}$ into $\H^{2,n}$, then
\[\vert \tau(\rho)\vert =  (2g-2)~.\\ \]
\end{prop}

\begin{rmk}
We will prove the converse of this proposition in the next subsection.
\end{rmk}

Finally, let us describe the Killing metric on $\Gg(\H^{2,n})$. The Killing form on $\mathfrak{so}(2,n+1)\subset \mathfrak{sl}(n+3,\R)$ is given by
\[\langle M,N\rangle = (n+3)\tr(MN).\]
Its restriction to the Lie algebra of $\text{H}$ is non-degenerate, and so it induces a pseudo-Riemannian metric $g_\Gg$ of signature $(2n+2,n)$ on $\Gg(\H^{2,n})$.

Given a point $p\in \Gg(\H^{2,n})$ corresponding to a splitting $\R^{2,n+1}=T\oplus l \oplus N$, we have an identification
\[ T_p\Gg(\H^{2,n})= \Hom(l,T) \oplus \Hom(l,N) \oplus \Hom(T,N).\]
If $\varphi=(\varphi_1,\varphi_2,\varphi_3)\in T_p\Gg(\H^{2,n})$ is a tangent vector, its norm with respect to the Killing metric is given by
\[\Vert\varphi\Vert^2_{\Gg} = (n+1)\left( \tr(\varphi_1\varphi_1^\dagger)-\tr(\varphi_2\varphi_2^\dagger)+\tr(\varphi_3\varphi_3^\dagger) \right),\]
where $\varphi_i^\dagger$ is obtained from the dual of $\varphi_i$ using the identifications $T\cong T^*,~l\cong l^*$ and $N\cong N^*$ given by the \emph{positive definite} quadratic form $\mathbf{q}_T\oplus (-\mathbf{q}_{l\oplus N})$ on $\R^{n+3}$.

Note that when $\mathcal{G}: \widetilde\Sigma \longrightarrow \mathcal{G}(\H^{2,n})$ is the main Gauss map of $u: \widetilde\Sigma \to \H^{2,n}$, we have \[d\mathcal{G}=(du,0,\mathrm{II})~,\] 
where $du: T\widetilde\Sigma \to T^u$ is an isomorphism and $\mathrm{II}: T\widetilde\Sigma \to \Hom(T^u,N^u)$ is the second fundamental form of $u$ (this is actually an equivalent definition of $\mathrm{II}$).

The same construction applies to the homogeneous spaces $\Gr_{(2,1)}\big(\R^{2,n+1} \big)$ and $\Gr_{(2,0)}\big(\R^{2,n+1} \big)$ where the Killing form induces a metric of signature $(2n,n)$ and $(2n+2,0)$ respectively. The induced metric on $\Gr_{(2,0)}\big(\R^{2,n+1} \big)$ is the Riemannian metric of the symmetric space of $\SO_0(2,n+1)$.

\subsection{Extremal surfaces}\label{ss:extremalsurfaces}
In this section, we show that given a maximal representation $\rho : \Gamma \to \SO_0(2,n+1)$, there exists a $\rho$-equivariant \emph{maximal space-like embedding} from $\widetilde{\Sigma}$ into $\H^{2,n}$, proving in particular the converse of Proposition \ref{p:SpacelikeImmersion=>Maximal}.

Let us first recall some classical facts about extremal immersions into pseudo-Riemannian manifolds. We refer to \cite[Chapter 1.3]{anciaux} for more details.

Let $(M,g)$ be a pseudo-Riemannian manifold of signature $(p,q)$ with $p\geq 2$, $\rho$ be a representation of $\Gamma$ into $\Isom(M)$ and $u:\widetilde{\Sigma} \to M$ be a $\rho$-equivariant space-like immersion.
Recall that the second fundamental form $\mathrm{II}$ of $u$ is a symmetric $2$-form on $\Sigma$ with values in the normal bundle $N^u$. The \textit{mean curvature vector field} of the immersion $u$ is given by 
$$H(u):=\text{tr}_{g_T} (\mathrm{II}^u)\in \Omega^0(N^u)$$
(i.e. $H(u)_x = \mathrm{II}^u_x(e_1,e_1) + \mathrm{II}^u_x(e_2,e_2)$, where $(e_1,e_2)$ is an orthonormal basis of $(T_x\Sigma, g_T)$).
The following is classical:
\begin{prop}
The mean curvature vector field $H(u)$ vanishes identically if and only if the space-like immersion $u$ is a critical point of the \emph{area functional} which associates to $u$ the total area of the metric $g_T$.
\end{prop}

We will call such an immersion an \textit{extremal immersion} and its image an \emph{extremal surface}. When $(N^u,g_N)$ is positive definite (for example when $(M,g)$ is Riemannian), an extremal immersion locally minimizes the area and will be called a \textit{minimal immersion}. On the other hand, when $(N^u,g_N)$ is negative definite (i.e. when $M$ has signature $(2,q)$), an extremal immersion locally maximizes the area and will be called a \textit{maximal immersion}.

\begin{rmk}
An immersion $u: \widetilde{\Sigma} \to (M,g)$ is extremal if and only if $u$ is harmonic with respect to the metric $g_T$ on $\Sigma$. This can be reformulated as follows, if $X$ is a Riemann surface structure on $\Sigma$ and $u:\widetilde{X} \to \H^{2,n}$ is a space-like immersion which is conformal, then $u$ is an extremal immersion if and only if it is harmonic. (Here, ``harmonic'' refers to the general notion of harmonic maps between (pseudo-)Riemannian manifolds.)
\end{rmk}

We can now state our existence theorem for maximal representations: 

\begin{prop}\label{p:ExistenceMaximalSurface}
Let $\rho:\Gamma \longrightarrow \SO_0(2,n+1)$ be a maximal representation. Then there exists a $\rho$-equivariant maximal space-like embedding $u: \widetilde\Sigma \longrightarrow \H^{2,n}$.
\end{prop}

\begin{proof}
Recall that the energy function  $\EE_\rho : \Teich(\Sigma)\to \R$ of a representation $\rho$ is given by \eqref{eq:Energy on Teich}. 
Let $X\in\Teich(\Sigma)$ be a critical point of the energy function, such an $X$ exists by Proposition \ref{p:Labourie Existence}. 
By Theorem \ref{p:decompositionbundle}, the flat $\R^{2,n+1}$-bundle $E_\rho$ with holonomy $\rho$ decomposes orthogonally as
$$E_\rho=U \oplus \ell \oplus V~,$$
where $\ell$ is a negative definite line sub-bundle, $U$ is positive definite of rank 2 and $V$ is negative definite of rank $n$.

The pullback $\widetilde\ell$ of $\ell\subset E_\rho$ to $\widetilde X$ defines a $\rho$-equivariant map $u: \widetilde X \to \H^{2,n}$. By Corollary \ref{c:nablal}, 
$\nabla u$ is an isomorphism from $T\Sigma$ to $U$. Therefore $u$ is a space-like immersion, and the orthogonal splitting $E_\rho = U\oplus \ell \oplus V$ is the splitting $T^u \oplus l^u \oplus N^u$ given by the main Gauss map of $u$.

We now prove that $u$ is harmonic and conformal.

Over a local chart $U\subset \widetilde X$, the map $u$ can be lifted to a map into $\hat{\H}^{2,n}\subset \R^{2,n+1}$ (equivalently, $\widetilde\ell$ is locally spanned by a section $u$ of norm $-1$). The Levi-Civita connection of $\hat{\H}^{2,n}$ is the connection induced by the flat connection $\nabla$ on $\R^{2,n+1}$. Since $\hat{\H}^{2,n}\subset\R^{2,n+1}$ is umbilical, $u$ satisfies the harmonic equation of Proposition \ref{p-harmonicholomorphic} if and only if $\nabla_{\overline{\partial}_z}\nabla_{\partial_z} u$ is parallel to $u$ (here, $z$ is a complex coordinate on the local chart $U$).

Let 
$$(\mathcal{E},\Phi)=
	\xymatrix@R=0em{\Ii\mathcal{K}\ar[r]^1&\Ii\ar[r]^1&\Ii\mathcal{K}^{-1}\ar[ddl]^{\beta}\\&\oplus&\\&\Vv\ar[uul]^{\beta^\dagger}&},$$ 
be the Higgs bundle associated to $\rho$ as in Section \ref{Higgsbundleparametrization} and let $h$ be the Hermitian metric on $\mathcal{E}$ solving the self-duality equations \eqref{eq: SO(2,n+1) Higgs bundle Equations}. Decomposing the flat connection as 
$$\nabla = \nabla_h + \Phi + \Phi^*,$$
where $\nabla_h$ is the Chern connection of $(\mathcal{E},h)$, one gets
\begin{eqnarray*}
\nabla_{\overline{\partial}_z}\nabla_{\partial_z} u & = & \left[\big((\nabla_h^{0,1}+\Phi^*)(\overline{\partial}_z)\big)\circ (\nabla_h^{1,0} + \Phi)(\partial_z)\right] (u) \\
& = & \left[\big((\nabla_h^{0,1}+\Phi^*)(\overline{\partial}_z)\big)\circ \Phi(\partial_z)\right] (u) \\
& = & \left[\Phi^*\big(\overline{\partial}_z\big)\circ\Phi(\partial_z)\right](u).
\end{eqnarray*}
On the second line, we used the fact that the section $u$ is $\nabla_h$-parallel, while for the third line, we used the holomorphicity of $\Phi$.

In particular, $\Phi(\partial_z) u$ is a section of $\Ll^{-1}$.  Since the splitting $\mathcal{E}=\mathcal{L}\oplus \mathcal{I}\oplus \mathcal{L}^{-1} \oplus \mathcal{V}$ is orthogonal with respect to the metric $h$, $\Phi^*\big(\overline{\partial}_z\big)$ sends $\mathcal{L}^{-1}$ to $\mathcal{I}$. Hence, $\nabla_{\overline{\partial}_z}\nabla_{\partial_z} u$ is a section of $\ell$ and is thus is parallel to $u$. It follows that $u$ is harmonic.

Locally, the differential $du$ corresponds to $\nabla u=(\Phi+\Phi^*)u=(1+1^*)u$ which is nowhere vanishing. This  implies that $u$ is an immersion.

The Hopf differential of $u$ is locally given by
$$u^* q_{\H}^{2,0} = q_{\H}\big(\nabla_{\partial_z}u,\nabla_{\partial_z}u\big)dz^2,$$
where $q_\H$ is the $\C$-linear extension of the metric $g_\H$ on $\H^{2,n}$. But $\nabla_{\partial_z}u$ is a section of $\mathcal{L}^{-1}$ which is isotropic with respect to the $\C$-bilinear symmetric form $q$ on $\mathcal{E}$ given in \eqref{eq:q on Ee}. In particular, the Hopf differential of $u$ is zero. Thus $u^*g_{\H}$ is a conformal metric on $X$. Since the map $u$ is harmonic and conformal, it is a maximal immersion.  Finally, since $\rho(\Gamma)$ acts cocompactly on $u(\widetilde X)$, the pull-back metric is complete, and so $u$ is an embedding by Lemma \ref{l:GeometrySpacelikeSurface}.
\end{proof}
%

\begin{rmk}\label{r:geometricinterpretationofhiggsbundles}
The component of the Higgs field defined by $\beta\in \Omega^{1,0}\big(X,\Hom(\mathcal{L}^{-1},\mathcal{V})\big)$ is identified with the $(1,0)$-part of the second fundamental form $\mathrm{II}\in \Omega^1\big(X, \Hom(T^u, N^u)\big) $ of the maximal immersion $u$.
\end{rmk}

Finally, we show that the different Gauss maps (see Definition \ref{d-Gaussmaps}) are extremal immersions.

\begin{prop}\label{p-gaussmaps}
The main, first and second Gauss maps of the $\rho$-equivariant maximal embedding $u: \widetilde X \longrightarrow \H^{2,n}$ constructed above are extremal immersions.
\end{prop}

\begin{proof}
Since the calculations for each of the Gauss maps are similar, we will only prove the result for the main Gauss map.
It is proved by Ishihara in \cite{ishihara} that the three Gauss maps of a maximal immersion are harmonic. Thus, to prove the result we will show the Gauss maps are also conformal.

Consider a maximal embedding $u:\widetilde{X}\longrightarrow \H^{2,n}$, and let $\mathcal{G}: \widetilde X \longrightarrow \Gg(\H^{2,n})$ be its associated main Gauss map. As explained in Section \ref{subsection-Gaussmaps}, we have an identification $d\mathcal{G}=(du,0,\mathrm{II})$.

If $\partial\mathcal{G}$ denotes the $\mathbb{C}$-linear part of $d\mathcal{G}$ and $q_{\Gg}$ the $\C$-linear extension of $g_{\Gg}$, then $\partial\mathcal{G}=\partial u + \beta$ where $\beta$ is the $(1,0)$-part of the second fundamental form, and so $\beta$ is identified with the part of the Higgs field sending $\mathcal{L}^{-1}$ to $\mathcal{V}$ (see Remark \ref{r:geometricinterpretationofhiggsbundles}).

The Hopf differential of $\mathcal{G}$ is thus given by
\begin{eqnarray*}
\text{Hopf}(\mathcal{G}) & = &  q_{\Gg}(\partial\mathcal{G},\partial\mathcal{G}) \\
 & = & 2(n+1)\text{Hopf}(u) + 2(n+1)\text{tr}\big(\beta \beta^\dagger\big) \\
 & = & (n+1)\text{tr}\big(\beta \beta^\dagger\big) \\
 & = & 0.
 \end{eqnarray*}
For the last equation, we used the fact that $\beta^\dagger$ sends $\mathcal{V}$ to $\mathcal{L}$ (see Section \ref{Higgsbundleparametrization}).
Finally, a similar computation shows that $\mathcal{G}^*g_{\Gg}= (n+1)\Vert \Phi\Vert^2$. Since $\Phi$ is nowhere vanishing, we conclude that $\mathcal{G}^*g_{\Gg}$ is a space-like immersion.
\end{proof}

\begin{rmk}
The second Gauss map of $u$ is the $\rho$-equivariant harmonic map from $\widetilde{X}$ to the Riemannian symmetric space of $\SO_0(2,n+1)$.
\end{rmk}

\subsection{Uniqueness of the maximal surface}

Let $\rho\in\Rep^{max}(\Gamma,\SO_0(2,n+1))$ be a maximal representation. In this subsection, we prove the following theorem:

\begin{theo} \label{t:UniquenessMaximalSurfacePrecise}
Let $S_1$ and $S_2$ be two connected $\rho$-invariant maximal space-like surfaces in $\H^{2,n}$ on which $\rho(\Gamma)$ acts co-compactly. Then $S_1 = S_2$.
\end{theo}

As a corollary, we prove Labourie's conjecture for maximal representations into Hermitian Lie groups of rank $2$.

\begin{coro}\label{c:UniquenessCriticalPoint}
Let $\rho$ be a maximal representation from $\Gamma$ into a Hermitian Lie group of rank $2$. Then the energy function $\EE_\rho:~\Teich(\Sigma)\to \R$ defined in Section \ref{s:harmonicmetrics} has a unique critical point $X$. Moreover, the corresponding minimal immersion $f: \widetilde X \rightarrow \mathfrak{X}$ is an embedding.
\end{coro}

Note that when $n=1$, Theorem \ref{t:UniquenessMaximalSurfacePrecise} was obtained by Barbot, B\'eguin, Zeghib \cite{BBZ} and its corollary was obtained by Schoen \cite{schoen} (see Remark \ref{r:ComparisonBonsanteSchlenker} for details).

Before proving Theorem \ref{t:UniquenessMaximalSurfacePrecise}, we need to collect a few more facts about $\rho$-invariant maximal surfaces. Let $S$ be a connected $\rho$-invariant space-like surface in $\H^{2,n}$ on which $\rho(\Gamma)$ acts co-compactly. Note that $S$ is complete since it is a covering of the compact Riemannian manifold $\rho(\Gamma)\backslash S$. Thus the content of Section~\ref{ss:SpacelikeSurface} applies to $S$. In particular, $S$ is compactified by a circle in $\Ein^{1,n}$. We first prove that this circle coincides with the image of the $\rho$-equivariant boundary map $\xi: \partial_\infty \Gamma \to \Ein^{1,n}$ from Theorem \ref{t:AnosovCurve}.

\begin{lem} \label{l:GammaOrbitsS}
For every $\gamma \in \Gamma\backslash\{\mathrm{id}\}$, there exists a point $x\in S$ such that
\[\rho(\gamma)^n \cdot x \tend{n\to +\infty} \xi(\gamma_+)~.\]
\end{lem}

\begin{proof}
Fix $\gamma \in \Gamma\backslash\{\mathrm{id}\}$. Let $e_+$ and $e_- $ be isotropic vectors spanning  $\xi(\gamma_+)$ and $\xi(\gamma_-)$ respectively and normalized so that $\scal{e_+}{e_-} = 1$. By Corollary \ref{c:AnosovProximal}, there is a $\lambda >1$ such that $\rho(\gamma)\cdot e_+ = \lambda e_+$, $\rho(\gamma)\cdot e_- = \frac{1}{\lambda} e_-$ and the spectral radius of the restriction of $\rho(\gamma)$ to $V = (e_-\oplus e_+)^\perp$ is less than $\lambda$.

Let $x$ be a point in $S$ and $\hat{x}$ be a lift of $x$ in $\hat{\H}^{2,n}$. Up to scaling $e_-$ and $e_+$, we can write
\[\hat{x} = \alpha(e_- + e_+) + v~,\]
for some $\alpha \in \R$ and some $v\in V$. We thus have
\[\rho(\gamma)^n \cdot \hat{x} = \lambda^n \alpha e_+ + \lambda^{-n} \alpha e_- + \rho(\gamma)^n v~.\]
Since $\rho(\gamma)_{|V}$ has spectral radius strictly less than $\lambda$, we deduce that $\rho(\gamma)^n \cdot x$ converges (in $\ProjR{n+2}$) to $[e_+] = \xi(\gamma_+)$ unless $\alpha = 0$.

Assume by contradiction that $\rho(\gamma)^n \cdot x$ does not converge to $\xi(\gamma_+)$ for any $x \in S$. In this case, $S$ is included in the projectivization of $V$. This is impossible because the intersection of $\H^{2,n}$ with the projectivization of $V$ is a sub-manifold of signature $(1,n-1)$, and hence, cannot contain a space-like surface.
\end{proof}

\begin{coro} \label{p:BoundaryMaximalSurface}
The boundary of $S$ in $\Ein^{1,n}$ is the image of $\xi$.
\end{coro}

\begin{proof}
By Lemma \ref{l:GammaOrbitsS}, $\partial_\infty S$ contains $\xi(\gamma_+)$ for every $\gamma \in \Gamma$. Since the set $\{\gamma_+, \gamma \in \Gamma\}$ is dense in $\partial_\infty \Gamma$, we deduce that $\partial_\infty S$ contains the image of $\xi$. Since the image of $\xi$ is also a topological circle, we conclude that $\partial_\infty S$ is exactly the image of $\xi$.
\end{proof}

We will now prove that, if $S$ is maximal, then it is contained in the \emph{convex hull} of its boundary.

\begin{defi}
Let $\partial_\infty \hat{S}_0$ be one connected component of $\partial_\infty \hat{S}$. The convex hull of $\partial_\infty \hat{S}_0$ is the set of vectors $u\in \hat{\H}^{2,n}$ such that any linear form on $\R^{2,n+1}$ which is positive on $\partial_\infty \hat{S}_0$ is positive on $u$. The convex hull of $\partial_\infty S$, denoted $\Conv(\partial_\infty S)$, is the projection to $\H^{2,n}$ of the convex hull of either connected component of $\partial_\infty \hat{S}$.
\end{defi}

\begin{prop} \label{p:MinimalSurfaceinConvexHull}
Assume that $S$ is a maximal surface. Then $S$ is included in the convex hull of $\partial_\infty S$.
\end{prop}

\begin{proof}
Let $\hat{S}_0$ be a connected component of $\hat{S}$, and let $\phi$ be a linear form on $\R^{2,n+1}$ which is positive on $\partial_\infty \hat{S}_0$. Let $u_0$ be a point in $\hat{S}_0$ and $\dot{u}_0$ be a tangent vector to $\hat{S}_0$ at $u_0$. Let $\mathrm{II}$ denote the second fundamental form of $\hat{S}_0$ in $\hat{\H}^{2,n}$ and let $\left(u(t)\right)_{-\epsilon < t < \epsilon}$ be a geodesic on $\hat{S}_0$ such that $u(0)= u_0$ and $u'(0) = \dot{u}_0$. Then $\mathrm{II}(\dot{u}_0,\dot{u}_0)$ is the orthogonal projection of $u''(0)$ to $T_{u_0}\hat{\H}^{2,n}$. Since $\hat{\H}^{2,n}\subset \R^{2,n+1}$ is umbilical, we have
\begin{equation}\label{eq:SecondDerivativeSurfaceGeodesic}
\mathrm{II}(\dot{u}_0,\dot{u}_0) = u''(0) - q(\dot{u}_0) u_0~,
\end{equation}
and thus
\[\Hess_{u_0}\phi_{|\hat{S}_0}(\dot{u}_0, \dot{u}_0) = \frac{d^2}{d t^2}_{\vert t=0} \phi(u(t)) = q\mathbf (\dot{u}_0)\phi(u_0) + \phi(\mathrm{II}(\dot{u}_0,\dot{u}_0))~.\]
Since $\hat{S}_0$ is a maximal surface, the trace of $\mathrm{II}$ with respect to the metric induced by $\mathbf{q}$ on $\hat{S}_0$ vanishes. We deduce that $\phi$ satisfies the equation
\[\Delta \phi_{|\hat{S}_0} = \phi_{|\hat{S}_0}~,\]
where $\Delta$ is the Laplace operator of the metric induced by $\mathbf{q}$ on $\hat{S}_0$.\footnote{Here we use the convention that $\Delta f = \frac{1}{2} \Tr \,\Hess(f)$.}

Now, by assumption, $\phi_{|\hat{S}_0}$ is positive in a neighborhood of $\partial_\infty \hat{S}_0$. The classical maximum principle (Lemma \ref{l:StrongMaxPrinciple}) then implies that $\phi$ is positive on $\hat{S}_0$.
Therefore, $\hat{S}_0$ is included in $\Conv\big(\partial_\infty\hat{S}_0 \big)$ and $S$ is included in $\Conv(\partial_\infty S)$.
\end{proof}

We now turn to the proof of Theorem \ref{t:UniquenessMaximalSurfacePrecise}. Let $S_1$ and $S_2$ be two maximal $\rho$-invariant space-like surfaces in $\H^{2,n}$ on which $\rho$ acts co-compactly. Assume by contradiction that $S_1$ and $S_2$ are distinct. By Corollary \ref{p:BoundaryMaximalSurface}, $S_1$ and $S_2$ have the same boundary in $\Ein^{1,n}$. Let us start by lifting $S_1$ and $S_2$ to $\hat{\H}^{2,n}$ so that the two lifts have the same boundary. To simplify notations, we still denote those lifts by $S_1$ and $S_2$. Recall that $\scal{\cdot}{\cdot}$ denotes the symmetric bilinear form associated to the quadratic form $\mathbf{q}$ on $\R^{2,n+1}$.

\begin{lem} \label{l:InfMoreThan0}
For all $(u,v)\in S_1\times S_2$, 
\[\scal{u}{v} < 0~.\]
\end{lem}

\begin{proof}
By Lemma \ref{l:HyperplaneSeparatesS}, for any $u \in S_1$, the linear form $\scal{u}{\cdot}$ is negative on $\partial_\infty S_1$. Moreover, since $\partial_\infty S_2 = \partial_\infty S_1$, Proposition \ref{p:MinimalSurfaceinConvexHull} implies that $S_2$ is included in $\Conv\left(\partial_\infty S_1\right)$. Therefore, the linear form $\scal{u}{\cdot}$ is negative on $S_2$.
\end{proof}

\begin{lem} \label{l:InfLessThan1}
If $S_1 \neq S_2$, then there exists $(u, v)\in S_1 \times S_2$ such that 
\[\scal{u}{v} > -1~.\]
\end{lem}

\begin{proof}
Assume that $S_1$ is not included in $S_2$. Let $x$ be a point in $S_1$ which is not in $S_2$. As in Proposition \ref{p:WarpedProduct}, choose an identification of $\hat{\H}^{2,n}$ with $\D \times \S^n$ in which $x$ is identified to $(0,v_1)$ for some $v_1 \in \S^n$. Since $S_2$ is the graph of some function $f:\D \to \S^n$, there exists $v_2 \in \S^n$ such that $y = (0,v_2) \in S_2$. Since $x\not\in S_2$, we have $v_2 \neq v_1$ and therefore
\[\scal{x}{y} = - \scal{v_1}{v_2} > -1~.\]
\end{proof}

\begin{lem}
The function \[\function{B}{S_1 \times S_2}{\R_{<0}}{(u,v)}{\scal{u}{v}}\] achieves its maximum.
\end{lem}

\begin{proof}
Let $(u_n,v_n) \in \left(S_1 \times S_2 \right)^\N$ be a maximizing sequence for $B$. Since $\rho(\Gamma)$ preserves $B$ and acts co-compactly on $S_1$, we can assume that $\left(u_n\right)_{n\in\N}$ is bounded in $S_1$. Up to extracting a sub-sequence, we can assume that $u_n$ converges to $u \in S_1$. By Lemma \ref{l:InfLessThan1}, we know that $B(u_n,v_n) > -1$ for $n$ sufficiently large. 
Assume by contradiction that $(v_n)_{n\in\mathbb{N}}$ is unbounded in $S_2$. Up to extracting a sub-sequence, there exists $\epsilon_n\tend{n\to +\infty} 0$ such that $\epsilon_n v_n$ converges to a vector $v \in \partial_\infty S_2$. Since $B(u_n,v_n)$ is bounded, we have \[B(u,v) = \lim_{n\to +\infty} \epsilon_n B(u_n,v_n) = 0~.\]
The vector $v$ is thus in $u^\perp$. Since $\partial_\infty S_1 = \partial_\infty S_2$, this contradicts Lemma \ref{l:HyperplaneSeparatesS}.
\end{proof}

We now have all the tools needed to apply the minimum principle to $B$ and prove Theorem \ref{t:UniquenessMaximalSurfacePrecise}.

\begin{proof}[Proof of Theorem \ref{t:UniquenessMaximalSurfacePrecise}]
Let $(u_0,v_0) \in S_1\times S_2$ be a point where $B$ achieves its maximum. By Lemmas \ref{l:InfMoreThan0} and \ref{l:InfLessThan1}, we have
\[-1 < B(u_0,v_0) < 0~.\]
For $\dot{u}_0 \in T_{u_0}S_1$ and $\dot{v}_0 \in T_{v_0}S_2$, let $(u(t))_{t\in(-\epsilon, \epsilon)}$ and $(v(t))_{t\in(-\epsilon, \epsilon)}$ be geodesic paths on $S_1$ and $S_2$ respectively, satisfying $u(0) = u_0$, $u'(0) = \dot{u}_0$ and $v(0) = v_0$, $v'(0) = \dot{v}_0$.

Since $B(u(t),v_0)$ is maximal at $t=0$, we have $\scal{\dot{u}_0}{v_0} = 0$. Since $\mathbf{q}(u(t)) = -1$ for all $t$, we also have $\scal{\dot{u}_0}{u_0} = 0$. Similarly, we have 
$\scal{\dot{v}_0}{u_0} = \scal{\dot{v}_0}{v_0} = 0$. We thus obtain that $T_{u_0} S_1$ and $T_{v_0} S_2$ are both orthogonal to $u_0$ and $v_0$. 

Let $\mathrm{II}_1: T_{u_0}S_1 \times T_{u_0} S_1 \to u_0^\perp$ and $\mathrm{II}_2: T_{v_0}S_2 \times T_{v_0} S_2 \to v_0^\perp$ denote respectively the second fundamental forms of $S_1$ and $S_2$ in $\hat{\H}^{2,n}$. Recall that we have
\[u''(0) = \mathrm{II}_1(\dot{u}_0,\dot{u}_0) + q(\dot{u}_0) u_0\]
and
\[v''(0) = \mathrm{II}_2(\dot{v}_0,\dot{v}_0) + q(\dot{v}_0) v_0\]
(see Equation \eqref{eq:SecondDerivativeSurfaceGeodesic}). The second derivative of $B(u(t),v(t))$ at $t=0$ is given by
\begin{equation} \label{eq:SecondVariationScalarProduct}
\ddt_{|t=0} B(u(t),v(t)) = \xymatrix@=.2em{2 \scal{\dot{u}_0}{\dot{v}_0} + \scal{\mathrm{II}_1(\dot{u}_0,\dot{u}_0)}{v_0} + \scal{\mathrm{II}_2(\dot{v}_0,\dot{v}_0)}{u_0}\\ +\ \mathbf{q}(\dot{u}_0) \scal{u_0}{v_0} + \mathbf{q}(\dot{v}_0) \scal{u_0}{v_0}~.}
\end{equation}
 Our goal is to find $\dot{u}_0$ and $\dot{v}_0$ such that this second derivative is positive.

Since $S_1$ is a maximal surface in $\hat{\H}^{2,n}$, the quadratic form $\beta_1:w \mapsto \scal{\mathrm{II}_1(w,w)}{v_0}$ on $\left(T_{u_0}(S_1),\mathbf{q}\right)$ has two opposite eigenvalues $\lambda$ and $-\lambda$. Similarly, the quadratic form $w \mapsto \scal{\mathrm{II}_2(w,w)}{u_0}$ on $\left(T_{v_0}(S_2),\mathbf{q}\right)$ has two opposite eigenvalues $\mu$ and $-\mu$.
Up to switching $S_1$ and $S_2$, we may assume that $\lambda \geq \mu\geq 0$. We now choose $\dot{u}_0$ and $\dot{v}_0$ such that 
\[\xymatrix{\beta_1(\dot{u}_0) = \lambda~, & \mathbf{q}(\dot{u}_0) = 1&\text{and}&\dot{v}_0 =\dfrac{p(\dot{u}_0)}{\sqrt{\mathbf{q}(p(\dot{u}_0))}}}\] where $p:\{u_0,v_0\}^\perp \to T_{v_0}S_2$ denotes the orthogonal projection. 

Since $\mathbf{q}(u_0) = \mathbf{q}(v_0) = -1$ and $|\scal{u_0}{v_0}| < 1$, the restriction of $\mathbf{q}$ to the plane $P_0\subset\R^{2,n+1}$ spanned by $\{u_0,v_0\}$ is negative definite. The restriction of $\mathbf{q}$ to $P_0^\perp$ thus has signature $(2,n-2)$. Since $T_{v_0}S_2$ is a space-like plane in $P_0^\perp$, we can write $\dot{u}_0 = p(\dot{u}_0) + w$ where $\mathbf{q}(w) \leq 0$. We thus have
\[\mathbf{q}(p(\dot{u}_0)) = \mathbf{q}(\dot{u}_0) - \mathbf{q}(w) \geq \mathbf{q}(\dot{u}_0) = 1~,\]
and therefore
\[\scal{\dot{u}_0}{\dot{v}_0} = \sqrt{\mathbf{q}(p(\dot{u}_0))} \geq 1~.\]

Let us now return to Equation \eqref{eq:SecondVariationScalarProduct}. With our choices of $\dot{u}_0$ and $\dot{v}_0$, we have $\beta_1(\dot{u}_0) = \lambda$ and $\mathrm{II}_2(\dot{v}_0) \geq -\mu \geq -\lambda$. Since $\scal{u_0}{v_0} = B(u_0,v_0) >-1$, we have 
\begin{eqnarray*}
\ddt_{|t=0} B(u(t),v(t)) & = & 2 \scal{\dot{u}_0}{\dot{v}_0} + 2 \scal{u_0}{v_0} + \beta_1(\dot{u}_0) + \mathrm{II}_2(\dot{v}_0) \\
&\geq & 2 \scal{u_0}{v_0} + 2 \\
& > & 0~.
\end{eqnarray*}
 This contradicts the maximality of $B$ at $(u_0,v_0)$.
\end{proof}

Finally, let us deduce the proof of Labourie's conjecture.

\begin{proof}[Proof of Corollary \ref{c:UniquenessCriticalPoint}]
Let $\rho$ be a maximal representation from $\Gamma$ into a Hermitian Lie group of rank $2.$ By \cite{burgeriozziwienhard}, the Zariski closure of the image of $\rho(\Gamma)$ is of tube type; thus, we can assume that $\Gamma$ takes values in $\SO_0(2,n+1)$ for some $n$ (see Remark \ref{rmk:LieGroupsRank2}).
Let $X_1$ and $X_2$ be two critical points of $\EE_\rho$. Proposition \ref{p:ExistenceMaximalSurface} constructs two $\rho$-equivariant maximal space-like immersions $u_1: \widetilde{X}_1 \to \H^{2,n}$ and $u_2: \widetilde{X}_2 \to \H^{2,n}$. By Theorem \ref{t:UniquenessMaximalSurfacePrecise}, these two immersions have the same image $S$. Moreover, since $S$ is homeomorphic to a disc (see Lemma \ref{l:GeometrySpacelikeSurface}), both $u_1$ and $u_2$ are diffeomorphisms onto $S$. The map $u_2\circ u_1^{-1}$ induces a biholomorphism from $X_1$ to $X_2$ that is homotopic to the identity. Hence $X_1 = X_2$ in $\Teich(\Sigma)$.

Finally, by Proposition \ref{p-gaussmaps}, the minimal $\rho$-equivariant immersion $f_1: \widetilde{X} \to \mathfrak{X}=\Gr_{(2,0)}\left (\R^{2,n+1}\right)$ is the second Gauss map of the map $u_1$. In other words, $f_1$ maps a point $x$ to the space-like $2$-plane $T^{u_1}_x$. Assume my contradiction that $f_1(x) = f_1(y)= P$ for $x\neq y \in \widetilde{X}_1$.
Then $u_1(x)$ and $u_1(y)$ both belong to $P^\perp$. This contradicts Corollary \ref{c:IntersectionSpacelikeSurface} according to which every negative definite linear subspace of $\R^{2,n+1}$ of dimension $n+1$ intersects $u_1(\widetilde{X}_1)$ exactly once. Therefore, the second Gauss map $f_1$ is injective, which concludes the proof of Corollary \ref{c:UniquenessCriticalPoint}. 
\end{proof}

\begin{rmk}[Comparison with the work of Schoen and Bonsante--Schlenker] \label{r:ComparisonBonsanteSchlenker}

In the case of $\SO_0(2,2)$, Corollary \ref{c:UniquenessCriticalPoint} was proven directly by Schoen \cite{schoen}. This case is quite special because $\SO_0(2,2)$ is a degree $2$ cover of $\PSL(2,\R) \times \PSL(2,\R)$ and $\SO_0(2,2)/\mathrm{S}(\mathrm{O}(2) \times \mathrm{O}(2))$ identifies with $\H^2\times \H^2$.

Krasnov and Schlenker \cite{krasnovschlenker} and Bonsante and Schlenker \cite{BonsanteSchlenker} later clarified the link between maximal surfaces in $\H^{2,1}$ and minimal surfaces in $\H^2 \times \H^2$. In \cite{BonsanteSchlenker}, they gave an intrinsic proof of the uniqueness of a maximal surface in $\H^{2,1}$ in  a more general setting. In their proof, they maximize the time-like distance between a point in $S_1$ and a point in $S_2$ and derive a contradiction from a maximum principle. This approach seems to require an estimate on the curvature of the maximal surface. Our strategy above is inspired by their proof, except that we apply the maximum principle to the scalar product instead of the time-like distance, which does not require any curvature estimate. This relieves us from extra technical difficulties.
\end{rmk}

\subsection{Length spectrum of maximal representations}

In this section, we exploit the pseudo-Riemannian geometry of $\H^{2,n}$ and the existence of a $\rho$-equivariant maximal space-like embedding of $\widetilde{\Sigma}$ to obtain a comparison of the length spectrum of $\rho$ with that of a Fuchsian representation.

In our setting, we define the length spectrum of a representation $\rho$ as follows.

\begin{defi}
Let $\rho$ be a representation of $\Gamma$ into $\SO_0(2,n+1)$. The \emph{length spectrum} of $\rho$ is the function $L_\rho: \Gamma \to \R_+$ that associates to an element $\gamma \in \Gamma$ the logarithm of the spectral radius of $\rho(\gamma)$ (seen as a square matrix of size $n+3$).
\end{defi}

\begin{rmk}
 This definition coincides with Definition \ref{d:LengthSpectrumIntro} since, for $A \in \SO_0(2,n+1)$, $A$ and $A^{-1}$ have the same spectral radius.
\end{rmk}

\begin{theo} \label{t:LengthSpectrumComparison}
If $\rho: \Gamma \to \SO_0(2,n+1)$ is a maximal representation, then either $\rho$ is in the Fuchsian locus (see Definition \ref{d:FuchsianLocus}), or there exists a Fuchsian representation $j: \Gamma \to \SO_0(2,1)$ and $\lambda >1$ such that
\[L_\rho \geq \lambda L_j~.\]
\end{theo}

\begin{rmk}
The representation $\rho$ is in the Fuchsian locus if and only if it stabilizes a totally geodesic space-like copy of $\H^2$ in $\H^{2,n}$. The induced action of $\rho$ on $\H^2$ gives a Fuchsian representation $j$ such that $L_j= L_\rho$.
\end{rmk}

\begin{rmk}
Let $m_{irr}$ denote the irreducible representation of $\SO_0(2,1)$ into $\PSL(n,\R)$. For a Hitchin representation $\rho: \Gamma \to \PSL(n,\R)$, one could hope to find a Fuchsian representation $j:\Gamma \to \SO_0(2,1)$ such that 
\[L_\rho \geq L_{m_{irr} \circ j} = \frac{n-1}{2}L_j~.\]
However, this statement fails to be true for $n\geq 4$ (see \cite[Section 3.3]{LeeZhang}). In particular, it is not true for Hitchin representations into $\SO_0(2,3)$. Nonetheless, Theorem \ref{t:LengthSpectrumComparison} gives a weaker result.
\end{rmk}

In order to prove Theorem \ref{t:LengthSpectrumComparison}, let us fix a maximal representation $\rho: \Gamma \to \SO_0(2,n+1)$ and let $u: \widetilde{\Sigma} \to \H^{2,n}$ be a $\rho$-equivariant maximal space-like embedding. Recall that $g_T$ and $g_N$ respectively denote the metrics on the tangent bundle $T^u$ and the normal bundle $N^u$ induced by $g_{\H^{2,n}}$. By Poincar\'e's uniformization theorem, the metric $g_T$ is conformal to a unique metric $g_P$ of constant curvature $-1$.

\begin{lem} \label{l:ComparisonInducedMetric}
Either $\rho$ is in the Fuchsian locus, or there exists $\lambda >1$ such that $g_T \geq \lambda g_P$.
\end{lem}

\begin{proof}
Let $\kappa(g_T)$ denote the Gauss curvature of $g_T$.  Recall that $\kappa(g_T)$ can be computed from the second fundamental form by the formula :
\[\kappa(g_T)_x = -1 - \sum_{i=1}^n \det_{g_T} g_N\left(\mathrm{II}^u_x(\cdot,\cdot),e_i\right)\]
where $(e_i)_{1\leq i \leq n}$ is an orthonormal basis of the normal bundle $N^u$ at $x$. In this formula, $\mathrm{II}^u_x(\cdot,\cdot)$ is seen as a symmetric $2$-form on $T_x\Sigma$ with values in $N_x^u$, so $g_N\left(\mathrm{II}^u_x(\cdot,\cdot),e_i\right)$ is a symmetric $2$-form with values in $\R$. Note that the minus sign in front of the sum comes from the fact that $g_N$ is negative definite.

Since $u$ is a maximal immersion, the $2$-form $g_N\left(\mathrm{II}^u_x(\cdot,\cdot),e_i\right)$ has trace $0$ with respect to $g_T$. Thus $\det_{g_T} g_N\left(\mathrm{II}^u_x(\cdot,\cdot),e_i\right)<0$, with equality if and only if $\mathrm{II}^u_x = 0$. Hence, we have $\kappa(g_T) \geq -1$. The \emph{Ahlfors--Schwarz--Pick lemma} (Theorem \ref{t:CurvatureIneq=>MetricIneq}) then implies that $g_T \geq g_P$. Moreover, if equality holds at one point, then $g_T = g_P$, and, in particular, $\kappa(g_T) = -1$ everywhere. This implies
that $u(\Sigma)$ is totally geodesic.
\end{proof}

Let $g$ be a Riemannian metric on $\Sigma$ and denote by $d_g$ the associated distance on $\widetilde{\Sigma}$. Define the length spectrum of $g$ as the map
\[\function{L_g}{\Gamma}{\mathbb{R}_+}{\gamma}{\lim\limits_{n\to +\infty} \frac{1}{n}\ d_g(x,\gamma^n \cdot x)}~,\]
where $x$ is any point in $\widetilde{\Sigma}$. Note that this definition makes sense without any hypothesis on the curvature of $g$. If $g$ is negatively curved, then $L_g(\gamma)$ is the length of the unique closed geodesic freely homotopic to $\gamma$.

From now on, we assume that $\rho$ does not preserve a copy of $\H^2$. It follows from Lemma \ref{l:ComparisonInducedMetric} that $L_{g_u} \geq \lambda L_{g_P}$ for some $\lambda >1$. Let $j$ be the Fuchsian representation uniformizing $g_P$, i.e. such that there exists a $j$-equivariant isometry from $(\widetilde{\Sigma}, g_P)$ to $\H^2$. We then have
\[\lambda L_{j} = \lambda L_{g_P} \leq L_{g_T}~.\]
In order to prove Theorem \ref{t:LengthSpectrumComparison}, it is thus enough to show the following:

\begin{lem} \label{l:ComparisonInducedLengthSpectrum}
We have
\[L_\rho \geq L_{g_T}~.\]
\end{lem}

In order to prove this lemma, we need another characterization of $L_\rho$. Recall that, if $x$ and $y$ are joined by a space-like geodesic, then $d_{\H^{2,n}}(x,y)$ denotes the length of the space-like geodesic segment between $x$ and $y$ (see Section \ref{ss:H2n}). We set $d_{\H^{2,n}}(x,y)=0$ otherwise.

\begin{prop} \label{p:PseudoRiemannianLength}
For any $\gamma \in \Gamma$ and any $x \in u(\widetilde{\Sigma})$, we have
\[L_\rho(\gamma) = \lim_{n\to + \infty} \frac{1}{n} d_{\H^{2,n}}(x, \rho(\gamma)^n \cdot x)~.\]
\end{prop}

\begin{proof}
By Corollary \ref{c:AnosovProximal}, one can find two isotropic vectors $e_+$ and $e_- \in \mathbb{R}^{2,n+1}$ with $\scal{e_+}{e_-} = 1$ such that $\rho(\gamma)\cdot e_+ = e^{L_\rho(\gamma)} e_+$ and $\rho(\gamma)\cdot e_- = e^{- L_\rho(\gamma)} e_-$. Moreover, if $V$ denotes the orthogonal of the vector space spanned by $e_-$ and $e_+$, then the spectral radius of the restriction of $\rho(\gamma)$ to $V$ is strictly less than $e^{L_\rho(\gamma)}$.

Let $x$ be a point in $u(\widetilde\Sigma)$ and $\hat x$ be a lift of $x$ to $\widehat \H^{2,n}$. We can write
\[\hat x = \alpha_-e_- + \alpha_+ e_+ + v~,\]
with $v\in V$. By Proposition \ref{p:BoundaryMaximalSurface}, we have
\[\rho(\gamma)^n \cdot x \tend{n\to +\infty} [e_+]\]
and
\[\rho(\gamma)^n \cdot x \tend{n\to -\infty} [e_-]\]
Hence $\alpha_+$ and $\alpha_-$ are non-zero.

We have
\begin{eqnarray*}
\frac{1}{n} d_{\H^{2,n}}(x, \rho(\gamma)^n\cdot x) &=&  \frac{1}{n} \cosh^{-1} \left\vert\scal{\hat x}{\rho(\gamma)^n \cdot \hat x} \right\vert\\
&=&\frac{1}{n} \cosh^{-1} \vert \langle \alpha_- e_- + \alpha_+ e_+ + v~ ,\\
&&  \alpha_- e^{-nL_\rho(\gamma)} e_- + \alpha_+ e^{n L_\rho(\gamma)} e_+ + \rho(\gamma)^n \cdot v  \rangle  \vert \\
&=&\frac{1}{n} \cosh^{-1} \left \vert 2 \alpha_- \alpha_+ \cosh(nL_\rho(\gamma)) + \scal{v}{\rho(\gamma)^n \cdot v}\right\vert~.
\end{eqnarray*}
Since the spectral radius of $\rho(\gamma)$ restricted to $V$ is strictly less than $L_\rho(\gamma)$, the term $\scal{v}{\rho(\gamma)^n \cdot v}$ is negligible and we obtain
\[\frac{1}{n} d_{\H^{2,n}}(x, \rho(\gamma)^n\cdot x) \tend{n\to +\infty} L_\rho(\gamma).\]

\end{proof}

In order to conclude the proof of Lemma \ref{l:ComparisonInducedLengthSpectrum}, it suffices to prove the following:

\begin{prop} \label{p:ComparisonAmbientMetric}
If $x$ and $y \in u(\widetilde{\Sigma})$ are joined by a space-like geodesic segment, then we have
$d_u(x,y) \leq d_{\H^{2,n}}(x,y)$, where $d_u$ is the induced distance on $u(\widetilde \Sigma).$
\end{prop}

\begin{proof}
Recall that, according to Proposition \ref{p:WarpedProduct}, the space $\hat{\H}^{2,n}$ is isometric to a warped product
\[\H^2 \times \S^n\]
with the metric
\[g = g_{\H^2} \oplus - w g_{\S^n}~,\]
for some positive function $w$ on $\H^2$. In this warped product structure, the horizontal slices $\H^2 \times \{x_2\}$ are totally geodesic.

Let $x$ and $y$ be two points in $u(\widetilde{\Sigma})$ and let $\hat{x}$ and $\hat{y}$ be lifts of $x$ and $y$ to $\hat{\H}^{2,n}$ belonging to the same lift $\hat{S}$ of $u(\widetilde{\Sigma})$. Let us choose a warped product structure on $\hat{\H}^{2,n}$ such that $x$ and $y$ belong to the same horizontal slice. 

Let $\pi$ denote the restriction to $\hat{S}$ of the projection on the $\H^2$ factor with respect to this warped product structure. We then have
\[d_{\H^{2,n}}(x,y) = d_{\H^2}(\pi(x), \pi(y))~.\]
By Proposition \ref{p:SLipschitzGraph}, $\pi$ is a diffeomorphism. Moreover, given the warped product structure of the metric $g_{\H^{2,n}}$, on $u(\widetilde{\Sigma})$ we have
\[ \label{eq:ProjectionMetric}
\pi^*g_{\H^2} \geq g_{\H^{2,n}}~.
\]

Let $c: [0,1] \to \H^2$ denote the geodesic segment between $\pi(x)$ and $\pi(y)$. We have
\begin{eqnarray*}
d_u(x,y) &\leq & \int_0^1 \sqrt{g_{\H^{2,n}}\left(\dt \pi^{-1}\circ c(t) \right)} dt\\
& \leq & \int_0^1 \sqrt{g_{\mathbb{H}^2}\left(\dt c(t) \right) }d t = d_{\mathbb{H}^2}(\pi(x), \pi(y)) = d_{\H^{2,n}}(x,y)~.
\end{eqnarray*}
\end{proof}

We can now conclude that for any $\gamma \in \Gamma$,
\begin{eqnarray*}
L_{g_u}(\gamma) &= & \lim_{n\to +\infty} \frac{1}{n}d_u(x,\gamma^n\cdot x) \\
&\leq & \lim_{n\to +\infty} \frac{1}{n}d_{\H^{2,n}}(x,\gamma^n\cdot x) = L_\rho(\gamma)~,
\end{eqnarray*}
which proves Lemma \ref{l:ComparisonInducedLengthSpectrum} and thus Theorem \ref{t:LengthSpectrumComparison}.

\section{Geometric structures associated to maximal representations} \label{s:Geometrization}
In this section, we realize maximal representations into $\SO_0(2,n+1)$ as holonomies of geometric structures. More precisely, we prove the following two theorems:

\begin{theo}\label{t: fibered photons and maximal reps}
The holonomy gives a surjective map from the space of fibered photon structures on Stiefel bundles over $\Sigma$ onto the set of maximal representations into $\SO_0(2,n+1)$.
\end{theo}
\begin{theo} \label{t:EinsteinHitchin}
For any Hitchin representation $\rho\in \Hit(\Gamma,\SO_0(2,3))$, there exists a maximal fibered conformally flat Lorentz structure on the unit tangent bundle $\pi: T^1\Sigma \longrightarrow \Sigma$ whose holonomy is $ \rho\circ\pi_*$.
\end{theo}
The notions of fibered photon structure, Stiefel bundles, and maximal fibered conformally flat Lorentz structures are described in the next subsections.

\subsection{$(G,X)$-structures}\label{(G,X)-structures}

Here we recall the basic theory of $(G,X)$-structures. For more details, the reader is referred to \cite{goldmangeometricstructures}.

In this subsection, let $G$ be a semi-simple Lie group, $X=G/H$ be a $G$-homogeneous space and $M$ be a manifold such that $\dim(M)=\dim(X)$.

\begin{defi}\label{(G,X)-structures}
A $(G,X)$-structure on $M$ is a maximal atlas of charts taking values in $X$ whose transition functions are locally the restriction of elements in $G$.

Two $(G,X)$-structures on $M$ are \textit{equivalent} if there exists a diffeomorphism $f: M \longrightarrow M$ which is isotopic to the identity and whose expression in local charts is given by elements in $G$.
\end{defi}

Given a $(G,X)$-structure on $M$, one can associate a \textit{developing pair} $(\dev,\rho)$, where
$$\rho: \pi_1 (M) \longrightarrow G$$
is a representation called the \textit{holonomy} of the structure and
$$\dev: \widetilde{M} \longrightarrow X$$
is a $\rho$-equivariant local diffeomorphism called the \textit{developing map}. 

The developing pair is not uniquely defined. Given two developing pairs $(\dev_1,\rho_1)$ and $(\dev_2,\rho_2)$, if there exists an element $g\in G$ so that
$$\left\{\begin{array}{l}
\dev_1=g\circ\dev_2 \\
\rho_2(\gamma)=g \cdot \rho_1(\gamma) \cdot g^{-1},~\forall \gamma \in \pi_1(M)
\end{array}\right.,$$
then $(\dev_1,\rho_1)$ and $(\dev_2,\rho_2)$ correspond to equivalent $(G,X)$-structures.
It is well-known (see for example \cite{goldmangeometricstructures}) that a developing pair fully determines the $(G,X)$-structure on $M$.

In particular, if $\mathcal{D}_{(G,X)}(M)$ is the space of equivalence classes of $(G,X)$-structures on $M$, then we get a well-defined map
$$\textbf{hol}: \mathcal{D}_{(G,X)}(M) \longrightarrow \Rep(\pi_1(M),G),$$
where $\Rep(\pi_1(M),G):=\Hom\big(\pi_1(M),G\big)/G$ is the representation variety.

The well-known \emph{Ehresmann--Thurston principle} sates that this map induces a local homeomorphism from the set of equivalence classes of $(G,X)$-structures on $M$ to the representation variety.

\begin{theo}\cite[Chapter 3]{Thurston3manifolds}
Let $\rho_0$ be the holonomy of a $(G,X)$-structure on a closed manifold $M$. Then any representation $\rho:\pi_1(M) \to G$ sufficiently close to $\rho$ is the holonomy of a $(G,X)$-structure on $M$ close to the initial one, which is unique up to equivalence.
\end{theo}

An \emph{$X$-bundle over $M$} is a fiber bundle $p: \mathcal{X}\to M$ obtained by gluing together sets of the form $U_i\times X\cong p^{-1}(U_i)$, where $\{U_i\}_{i\in I}$ is a covering of $M$ and, for $U_i\cap U_j\neq \emptyset$, the transition functions have the form
$$\begin{array}{llll}
\Psi: & (U_i\cap U_j)\times X & \longrightarrow & (U_i\cap U_j)\times X~, \\
 & (m,x) & \longmapsto & \big(m,g(m)x\big)
 \end{array}$$
where $g$ is a smooth map from $U_i\cap U_j$ to $G$.

Given a principal $G$-bundle $P \to M$, the quotient $P/H$ is an $X$-bundle. Conversely, given an $X$-bundle over $M$, there exists an open covering $\mathcal{U}= \big\{ U_i\big\}_{i\in I}$ of $M$ such that the transition functions define a family of maps $g_{ij}: U_i\cap U_j \to G$ for any pair $(i,j)$ with $U_i\cap U_j \neq \emptyset$. These maps satisfy the cocycle condition $g_{ij}g_{jk}g_{ki}=1$ on triple intersections $U_i\cap U_j \cap U_k\neq \emptyset$. By gluing together sets of the form $U_i\times G$ with the same cocycle, we obtain a principal $G$-bundle that we call \emph{the underlying principal $G$-bundle}.

\begin{defi}
Two principal $G$-bundles $P_1$ and $P_2$ over $M$ are isomorphic if there exists a $G$-equivariant diffeomorphism from $P_1$ to $P_2$ covering the identity on $M$.
Two $X$-bundles are \emph{isomorphic} if the underlying principal bundles are equivalent.
\end{defi}

Given $\rho\in\Rep(\pi_1(M),G)$, one can associate an $X$-bundle $\mathcal{X}_\rho$ defined by
$$\mathcal{X}_\rho:= P_\rho/H~,$$
where $P_\rho$ is the flat principal $G$-bundle with holonomy $\rho$. Equivalently, $\mathcal{X}_\rho = \big(\widetilde M \times X \big)/\pi_1(M)$, where the action of $\gamma\in\pi_1(M)$ on $(m,x)\in \widetilde M \times X$ is given by $\gamma.(m,x)=(\gamma. m,\rho(\gamma)x)$.

The bundle $\mathcal{X}_\rho$ is equipped with a flat structure, that is an integrable distribution whose dimension is the same as the dimension of $M$ and which is transverse to the fibers of $p: \mathcal{X}_\rho \longrightarrow M$. It follows that for each $x\in \mathcal{X}_\rho$, we have a splitting
$$T_x \mathcal{X}_\rho = T^v_x\mathcal{X}_\rho \oplus T^h_x\mathcal{X}_\rho.$$
Here $T^v_x\mathcal{X}_\rho= \ker(dp_x)$ is the vertical tangent space and $T^h_x\mathcal{X}_\rho$ is the horizontal tangent space given by the distribution. Note also that the projection $p: \mathcal{X}_\rho \longrightarrow M$ identifies $T_x^h\mathcal{X}_\rho$ with $T_{p(x)}M$.

In this language, a developing map with holonomy $\rho$ corresponds to a section $s$ of $\mathcal{X}_\rho$ which is transverse to the horizontal distribution.

\subsection{Fibered photon structures}\label{isotropicplanes}

\begin{defi}
A \textit{photon} in $\R^{2,n+1}$ is an isotropic 2-plane. We denote the set of photons in $\R^{2,n+1}$ by $\Pho(\R^{2,n+1})$.
\end{defi}

\begin{rmk}
Equivalently, a photon is a projective line inside the set of isotropic lines $\Ein^{1,n}\subset \ProjR{n+1}$. Indeed, such a projective line is necessarily the projectivization of an isotropic plane in $\R^{2,n+1}$.
\end{rmk}

The group $\mathrm{O}(2,n+1)$ acts transitively on $\Pho\big(\R^{2,n+1}\big)$ and the stabilizer of a photon is a parabolic subgroup denoted $\mathrm{P}_{2}$. We thus have an identification of
$\Pho\big(\R^{2,n+1}\big)$ with the homogeneous space $\mathrm{O}(2,n+1)/\mathrm{P}_{2}$.

Given $k,n\in\N$, the \emph{Stiefel manifold} $\mathcal{S}_k(\R^n)$ is the space of orthonormal $k$-frames of $\R^n$, that is, the set of $k$-tuples $(v_1,\dots,v_k)$ of orthonormal vectors in $\R^n$.

\begin{lem}
For $n>0$, the space $\Pho(\R^{2,n+1})$ is diffeomorphic to the Stiefel manifold  $\mathcal{S}_2(\R^{n+1})$. In particular, $\Pho\left(\R^{2,2}\right)\cong \S^1\sqcup \S^1,~\Pho\left(\R^{2,3}\right)\cong \R\P^3$ and, for $n>2$, $\Pho(\R^{2,n+1})$ is simply connected.
\end{lem}

\begin{proof}
Consider an orthogonal splitting $\R^{2,n+1}=E\oplus F$, where $E$ is a positive definite $2$-plane and $F=E^{\bot}$. Denote by $g_E$ (respectively $g_F$) the scalar product induced on $E$ (respectively $F$). For each photon $V\in \Pho(\R^{2,n+1})$, the restriction of the orthogonal projection $p_E: \R^{2,n+1} \to E$ defines an isomorphism between $V$ and $E$. In particular, each photon is the graph of a linear map $\phi: E \to F$. Hence, we an injective map
$$\Psi: \Pho(E\oplus F) \longrightarrow \Hom(E,F)~.$$
The image of $\Psi$ consists of those linear maps $\varphi: E \to F$ such that $g_E(x,y)=-g_F(\varphi(x),\varphi(y))$ for any $x,y\in E$.

Fixing an orthonormal basis $(e_1,e_2)$ of $E$, such a map $\varphi: E \to F$ is completely determined by the pair of orthonormal vectors $(\varphi(e_1),\varphi(e_2))$.
\end{proof}

Let $P$ be a principal $\mathrm{O}(2,n)$-bundle over $\Sigma$. If $M=P/\mathrm{P_2}$ is the quotient $P$ by the action of the parabolic subgroup $\mathrm{P}_2<\mathrm{O}(2,n)$, then $M$ is a fiber bundle over $\Sigma$ with fibers isomorphic to $\Pho\big(\R^{2,n}\big)$. We call such a fiber bundle $M$ a $\Pho\big(\R^{2,n}\big)$-bundle.

Let $M$ be a $\Pho\big(\R^{2,n}\big)$-bundle over $\Sigma$ and
let $\widetilde{M}$ denote the pull-back of $M$ to the universal cover $\widetilde\Sigma$.
Given a representation $\rho$ of $\pi_1(\Sigma)$ into $\mathrm{O}(2,n+1)$, we identify the representations $\rho$ and $\rho\circ \pi_*: \pi_1(M) \to \mathrm{O}(2,n+1)$, where $\pi:M \to \Sigma$ denotes the fibration.

\begin{defi} \label{d:FiberedPhotonStructure}
A map $f: \widetilde{M}\to \Pho\big(\R^{2,n+1}\big)$ is called \emph{fibered} if it maps each fiber of $\widetilde{M}$ isomorphically onto $\Pho(\ell^\perp)$ for some $\ell \in \H^{2,n}$.

An $\big(\mathrm{O}(2,n+1),\Pho(\R^{2,n+1})\big)$-structure on $M$ with holonomy $\rho$ is a \emph{fibered photon structure} if its developing map $\dev: \widetilde{M} \to \Pho\big(\R^{2,n+1}\big)$ is fibered.
\end{defi}

\begin{rmk}
The covering $\widetilde{M}$ of $M$ is not always simply connected. However, the developing map of a photon structure on $M$ with holonomy $\rho$ must factor through $\widetilde{M}$, so the definition makes sense.
\end{rmk}

Fibered photon structures are related to space-like surfaces in the following way. Let $\rho: \pi_1(\Sigma) \to \textrm O(2,n+1)$ be a representation. Given a $\Pho\big(\R^{2,n}\big)$-bundle $M$ over $\Sigma$ and a $\rho$-equivariant fibered map $f:\widetilde{M}\to \Pho\big (\R^{2,n+1}\big)$, one can construct the $\rho$-equivariant map $f_\Sigma: \widetilde{\Sigma}\to \H^{2,n}$ such that
\[f(\widetilde{M}_x) = \Pho\big(f_\Sigma(x)^\perp\big)\]
for all $x\in \widetilde{\Sigma}$. 
Conversely, a $\rho$-equivariant map $u: \widetilde{\Sigma} \to \hat{\H}^{2,n}$ defines a reduction of structure group of the flat $\mathrm O(2,n+1)$-bundle with monodromy $\rho$ to a principal $\mathrm O(2,n)$-bundle, and the associated $\Pho(\R^{2,n})$-bundle $M$ over $\Sigma$ comes with a natural $\rho$-equivariant fibered map $f:\widetilde{M} \to \Pho(\R^{2,n+}1)$ such that $f_\Sigma = u$. There is thus a bijection between $\rho$-equivariant maps from $\widetilde{\Sigma}$ to $\H^{2,n}$ and fibered maps on $\Pho(2,n)$-bundles over $\Sigma$. We have the following:

\begin{lem} \label{l:FiberedPhoton=Spacelike}
Let $M$ be a $\Pho(\R^{2,n})$-bundle over $\Sigma$ and $f:\widetilde{M}\to \Pho(\R^{2,n+1})$ be a $\rho$-equivariant fibered map. Then $f$ is a local diffeomorphism if and only if $f_\Sigma:\widetilde{\Sigma}\to \H^{2,n}$ is a space-like immersion. 
\end{lem}

\begin{proof}
Let $\pi:\widetilde{M}\to \widetilde{\Sigma}$ denote the fibration. 
By definition, $f$ induces a bijection between $\pi^{^-1}(x)$ and $\Pho(f_\Sigma(x)^\perp)$ for all $x\in\widetilde \Sigma$. Hence, for $x_1\neq x_2 \in \widetilde{\Sigma}$, $f(\pi^{-1}(x_1))$ and $f(\pi^{-1}(x_2))$ are disjoint if and only if $f_\Sigma(x_1)^\perp\cap f_\Sigma(x_2)^\perp$ does not contain any photons. A computation of the signature shows that this happens exactly when $f_\Sigma(x_1)$ and $f_\Sigma(x_2)$ are joined by a space-like geodesic. In conclusion, $f$ is injective if and only if $f_\Sigma$ maps distinct points to points that are joined by a space-like geodesic. The rest of the proof is the infinitesimal analogue of this argument.

Fix an orthogonal splitting $\R^{2,n+1} = E \oplus F$, with $E$ a space-like $2$-plane. For every $y\in \widetilde{M}$, let $\phi(y) : E \to F$ be the linear map whose graph is the photon $f(y)$. Let $x$ be a point in $\widetilde{\Sigma}$, $y$ be a point in $\pi^{-1}(x)$, $v$ be a tangent vector to $\widetilde{M}$ at $y$ and $u = d \pi(v)$. By definition of $f_\Sigma$, we have
\[\scal{e + \phi(y) e}{f_\Sigma(x)} = 0\]
for all $e\in E$. 
Taking a derivative in the direction $v$ gives
\begin{equation} \label{eq:DerivativeFiberedMap}
\scal{d \phi(v) e}{f_\Sigma(x)} + \scal{e + \phi(y) e}{d f_\Sigma(u)} = 0
\end{equation}
for all $e\in E$.

Assume first that $f_\Sigma$ is a space-like immersion. If $v$ is such that $d f(v) = 0$, then we have $d \phi(v) = 0$ and Equation \eqref{eq:DerivativeFiberedMap} implies that $f(y)$ is contained in $d f_\Sigma(u)^\perp$. If $u$ were non-zero, then $d f_\Sigma(u)$ would be non-zero and space-like. The subspace $\Span(f_\Sigma(x),d f_\Sigma(u))^\perp$ would thus be of signature $(1,n)$, contradicting the fact that $d f_\Sigma(u)^\perp$ contains the photon $f(y)$. Therefore, $u$ must vanish, meaning that $v$ is tangent to the fiber of $\pi$. But the restriction of a fibered map to each fiber is an immersion. Thus $d f(v) = 0$ implies $v=0$, proving that $f$ is an immersion.

Conversely, assume that $f_\Sigma$ is not a space-like immersion. Choose $x$ in $\widetilde{\Sigma}$ and $u\in T_x(\widetilde\Sigma)\backslash \{0\}$ such that $\norm{d f_\Sigma(u)}\leq 0$. Then $\Span(f_\Sigma(x), d f_\Sigma (u))^\perp$ contains a photon $\phi$. Since $f$ induces a bijection between $\pi^{-1}(x)$ and $\Pho(f_\Sigma(x)^\perp)$, there is a point $y\in \pi^{-1}(x)$ such that $f(y) = \phi$. Let us choose $v\in T_y\widetilde{M}$ such that $d \pi(v) = u$. By construction, we have $f(y) \in d f_\Sigma(u)^\perp$. Equation \eqref{eq:DerivativeFiberedMap} thus implies that $\scal{d \phi(v) e}{f_\Sigma(x)} = 0$ for all $e\in E$, meaning that $d f(v)$ is tangent to $\Pho(f_\Sigma(x)^\perp)$. There thus exists $w\in T_y \pi^{-1}(x)$ such that $d f (w) = d f (v)$, and therefore $d f(v - w) =0$. However, $v - w \neq 0$ since $d \pi(v -w) = u\neq 0$. Hence, $f$ is not an immersion.
\end{proof}

As a corollary, we obtain that holonomies of fibered photon structures are exactly maximal representations.

\begin{coro}
Let $\rho: \Gamma \to \SO_0(2,n+1)$ be a representation. Then the following are equivalent:
\begin{itemize}
\item[$(i)$] $\rho$ is maximal,
\item[$(ii)$] there exists a $\rho$-equivariant space-like immersion from $\widetilde{\Sigma}$ to $\H^{2,n}$,
\item[$(iii)$] there exists a $\Pho(\R^{2,n})$-bundle over $\Sigma$ with a fibered photon structure whose holonomy is $\rho$.
\end{itemize}
\end{coro}

\begin{proof}
The equivalence between $(ii)$ and $(iii)$ follows directly from Lemma \ref{l:FiberedPhoton=Spacelike}, the implication $(i)\Rightarrow (ii)$ follows from Proposition \ref{p:ExistenceMaximalSurface} and the implication $(ii)\Rightarrow (i)$ is the content of Proposition \ref{p:SpacelikeImmersion=>Maximal}.
\end{proof}

\begin{rmk}
Note that our construction of a fibered photon structure with holonomy a maximal representation $\rho$ uses the existence of a $\rho$-equivariant space-like immersion of $\widetilde{\Sigma}$ into $\H^{2,n}$. One could hope for a geometric proof of this fact. Morally, the Anosov property for maximal representations tells us that $\rho$ preserves a space-like circle in $\partial_\infty \H^{2,n}$ which ``obviously'' bounds a $\rho$-invariant space-like disc in $\H^{2,n}$. For $n =1$, such a disc is given for instance by (a smoothening of) the upper boundary of the convex hull of the limit set. In higher dimension, however, we do not know any direct construction of this disc and have no choice but to use the maximal surface given by the cyclic Higgs bundle.
\end{rmk}

Let us now describe how the topology of a photon bundle with fibered photon structure depends on the associated equivariant space-like immersion.

Consider $\rho: \Gamma \to \SO_0(2,n+1)$ a maximal representation, and $u: \widetilde\Sigma \to \H^{2,n}$ a $\rho$-equivariant space-like immersion. Let $M$ be the photon bundle and $\dev: \widetilde{M}\to \Pho(\R^{2,n+1})$ the developing map of the fibered photon structure associated to $u$.

Let $E_\rho$ be the flat $\R^{2,n+1}$-bundle over $\Sigma$ with holonomy $\rho$. Recall that the main Gauss map of $u$ (see Definition \ref{d-Gaussmaps}) defines an orthogonal splitting $E_\rho= T^u \oplus \ell^u \oplus N^u$, where $\ell^u_x$ is the line $u(x)$ and $d u$ defines an isomorphism from $T\Sigma$ to $T^u$.
\footnote{Note that by Corollary \ref{c:nablal}, if $u$ is the maximal space-like immersion, then this splitting coincides with the one obtained in Theorem \ref{p:decompositionbundle}.} By construction, the fibered photon structure corresponding to $u$ is on the $\Pho(\R^{2,n})$-bundle $\Pho(T^u\oplus N^u)$. We have the following:

\begin{lem}
The isomorphism class of the $\Pho(\R^{2,n})$-bundle $\Pho(T^u\oplus N^u)$ is characterized by the topological type of $N^u$.
\end{lem}

\begin{proof}
The principal $\mathrm{O}(2,n)$-bundle associated to the $\Pho(\R^{2,n})$-bundle $\pi: \Pho(T^u \oplus N^u) \to \Sigma$ corresponds to the reduction of structure group given by $T^u \oplus N^u\subset E_\rho$. The orthogonal splitting $T^u\oplus N^u$ gives a further reduction of structure group to $\mathrm{O}(2)\times\mathrm{O}(n)$. Since $T^u$ is isomorphic to $T\Sigma$ and $\Sigma$ is orientable, there is a further structure group reduction to $\SO(2)\times \mathrm{O}(n)$. The topological type of this $\mathrm{O}(2,n)$-bundle is thus given by the absolute value of the degree of $T^u$ and the topological type of $N^u$. Again, since $T^u$ is isomorphic to $T\Sigma$, we have $|\deg(T^u)|=2g-2$.
\end{proof}

\noindent\textbf{The case $\SO_0(2,3)$.} For the case of $\SO_0(2,3)$, one can say more about the topology of $\Pho(T^u \oplus N^u)$. Recall from Remark \ref{r: Gothen v nonGothen components} that the space of maximal $\SO_0(2,3)$-representations decomposes as 
\[\bigsqcup\limits_{sw_1\neq 0,\ sw_2}\Rep^{max}_{sw_1,sw_2}(\Gamma,\SO_0(2,3))\ \sqcup\bigsqcup_{0\leq d\leq 4g-4}\Rep_{d}^{max}(\Gamma,\SO_0(2,3)),\]
and that the components $\bigsqcup\limits_{0< d\leq 4g-4}\Rep_{d}^{max}(\Gamma,\SO_0(2,3))$ are called Gothen components while the rest of the components are called reducible components. 

Let $\rho:\Gamma \to \SO_0(2,3)$ be a maximal representation and $E_\rho=T^u\oplus l^u \oplus N^u$ the splitting associated to the main Gauss map of the unique $\rho$-equivariant maximal space-like embedding $u: \widetilde\Sigma \to \H^{2,n}$. We have the following:

\begin{lem} \label{l:ComponentsPho(2,2)}
The total space of the $\Pho(\R^{2,2})$-bundle $\Pho(T^u\oplus N^u)\to\Sigma$ is connected if and only if the first Stiefel-Whitney class of $N^u$ is non-zero.
\end{lem}
\begin{proof}
The splitting $T^u\oplus N^u$ gives a reduction of the principal $\mathrm{O}(2,2)$-bundle underlying $\Pho(T^u\oplus N^u)$ to a principal $\SO(2) \times \mathrm{O}(2)$-bundle. The stabilizer of a photon in $\R^{2,2}$ under the action of $\SO(2) \times\mathrm{O}(2)$ is conjugated to the diagonal embedding of $\SO(2)$, which is a connected subgroup. Therefore, $\Pho(T^u\oplus N^u)$ is connected if and only if the principal $\SO(2) \times \mathrm{O}(2)$-bundle is connected. This happens exactly when the first Stiefel-Whitney class of $N^u$ is non-zero.
\end{proof}

Recall that $\Pho(\R^{2,2})$ is disconnected and that the developing map of the fibered photon structure on $\Pho(T^u\oplus N^u)$ is injective. In particular, the image of $\dev$ has two connected components. 
The topology of $\Pho(T^u\oplus N^u)$ is given by the following:

\begin{lem}\label{l: O(U,V) bundle description}
Let $\rho$ be a maximal $\SO_0(2,3)$-representation and $\Pho(T^u\oplus N^u)$ be the $\Pho(\R^{2,2})$-bundle associated to the $\rho$-equivariant maximal space-like embedding $u: \widetilde\Sigma \to \H^{2,n}$. 
	\begin{itemize}
		\item If $\rho$ is in the Gothen component $\Rep^{max}_d(\Gamma,\SO_0(2,3))$, or in the reducible component $\Rep^{max}_0(\Gamma,\SO_0(2,3))$, then $\Pho(T^u\oplus N^u)$ is the disjoint union of two circle bundles with degrees $2g-2+d$ and $2g-2-d$ 
		that we denote $\Pho^+(T^u\oplus N^u)$ and $\Pho^-(T^u\oplus N^u)$ respectively.
		\item If $\rho$ is in the reducible component $\Rep^{max}_{sw_1,sw_2}(\Gamma,\SO_0(2,3))$, then $\Pho(T^u\oplus N^u)$ is connected.
	\end{itemize}
\end{lem}
\begin{proof}
In the first case, the first Stiefel--Whitney class of $N^u$ vanishes and we can thus choose an orientation of $N^u$ such that $\deg(N^u)=d\geq 0$. 
The two connected components of $\Pho(T^u\oplus N^u)$ are then given by the graphs of linear isometries $\varphi: T^u \to N^u$ that preserve and reverse the orientation respectively. We respectively call them $\Pho^+(T^u\oplus N^u)$ and $\Pho^-(T^u\oplus N^u)$.

The complex structure $J_{T^u} : T^u \longrightarrow T^u$ given by the rotation of angle $\pi/2$ defines a canonical identification between $T^u$ and $\text{Ker}(J_{T^u}-i\text{Id})\cong \mathcal{K}^{-1} \subset T^u\otimes \mathbb{C}$.
In a same way, the complex structure $J_{N^u}: N^u \longrightarrow N^u$ identifies $N^u$ with a holomorphic line sub-bundle $\mathcal{N}\subset N^u\otimes \C$, and $\mathcal{N}$ has degree $d$.
Under these identifications, $\Pho^+(T^u\oplus N^u)$ corresponds to unit vectors in $\text{Hom}(\mathcal{K}^{-1},\mathcal{N})= \mathcal{KN}$. Therefore, the degree of $\Pho^+(T^u\oplus N^u)$ is $2g-2+d$. In the same way, one gets that the degree of $\Pho^-(T^u\oplus N^u)$ is $2g-2-d$.

In the second case, the first Stiefel-Whitney class of $N^u$ is non-zero, hence $\Pho(T^u\oplus N^u)$ is connected by Lemma \ref{l:ComponentsPho(2,2)}.
\end{proof}

\begin{rmk}
Note that, a priori, the topology of the photon structure associated to a $\rho$-equivariant space-like immersion depends on this immersion. We do not know whether the space of equivariant space-like immersions is connected, and could imagine that it has several connected components which give rise to fibered photon structures which are not isotopic within the space of fibered photon structures.

 However, even if this is the case, in Section \ref{ss:equivalence of Structures} we will show that two such photon structures are always isomorphic as photon structures. Indeed, it will be proven that all of these photon structures are isomorphic to those constructed by Guichard--Wienhard. In particular, two $\rho$-equivariant space-like immersions give rise to isomorphic photon bundles. 
\end{rmk}

\subsection{Einstein structures for $\SO_0(2,3)$-Hitchin representations} \label{ss:EinsteinHitchin}

Here we prove Theorem \ref{t:EinsteinHitchin}, namely that one can associate to any $\SO_0(2,3)$-Hitchin representation a maximal fibered conformally flat Lorentz structure on the unit tangent bundle of $\Sigma$.  
More generally, we construct these structures for special $\SO_0(2,3)$ representations which give rise to cyclic Higgs bundles.  


\begin{defi}
A \textit{conformally flat Lorentz structure} (CFL structure) on a three dimensional manifold $M$ is a $(G,X)$-structure with $G=\SO_0(2,3)$ and $X=\Ein^{1,2}$.
\end{defi}

A \textit{space-like circle} in $\Ein^{1,2}$ is the intersection of a 3-dimensional linear subspace of $\R^{2,3}$ of signature $(2,1)$ with $\Ein^{1,2}$. Note that a space-like circle is thus a copy of $\Ein^{1,0}$ in $\Ein^{1,2}$. The set of space-like circles in $\Ein^{1,2}$ is the pseudo-Riemannian symmetric space 
$$\Gr_{(2,1)}(\R^{2,3}):= \SO(2,3)/\text{S}(\mathrm{O}(2,1)\times \mathrm{O}(2)).$$

\begin{defi}
A CFL structure on a circle bundle $\pi: M\longrightarrow \Sigma$ is called \textit{fibered} if the developing map sends each fiber onto a space-like circle in $\Ein^{1,2}$ and the holonomy is trivial along the fiber.

Two fibered CFL structures on $M$ are \emph{equivalent} if there exists a diffeomorphism $f: M \to M$ which preserves the fibers, is isotopic to the identity  and defines an equivalence of $\big(\SO_0(2,3), \Ein^{1,2} \big)$-structures (see Definition \ref{(G,X)-structures}).
\end{defi}

In particular, the holonomy of a fibered CFL structure can thus be written as $\rho\circ\pi_*$ where $\rho: \Gamma \to \SO_0(2,3)$. Also, in a similar way to fibered photon structures, one can associate to a fibered CFL structure on $M$ a $\rho$-equivariant map
$$\Psi: \widetilde{\Sigma} \longrightarrow \Gr_{(2,1)}(\R^{2,3}).$$
The map $\Psi$ sends a point $x\in\widetilde\Sigma$ to the element in $\Gr_{2,1}(\R^{2,3})$ corresponding to the space-like circle $\dev(\pi^{-1}(x))$. 
\begin{defi}
A fibered CFL structure will be called \textit{maximal} if $\Psi$ is an extremal space-like immersion.
\end{defi}

Note that, up to the action of an element $g\in\SO_0(2,3)$, the surface $\Psi(\widetilde\Sigma)$ only depends on the equivalence class of the fibered CFL structure.

Consider a representation $\rho\in \Rep\big(\Gamma,\SO_0(2,3)\big)$ such that there exists a Riemann surface structure $X\in \mathcal{T}(\Sigma)$ satisfying the property that the associated $\SO_0(2,3)$-Higgs bundle $(\mathcal{E},\Phi)$ is cyclic (see Definition \ref{d:CyclicHiggsBundle}) and has the form
\[\xymatrix@R=-.2em{\mathcal{L}\ar[r]^\beta&\mathcal{O}\ar[r]^\beta&\mathcal{L}^{-1}\ar@/^/[ddddl]^1\ar@/^/[ddl]_\gamma\\&\oplus&\\&\mathcal{KL}\ar@/^/[uul]_1&\\&\oplus&\\&\mathcal{K}^{-1}\mathcal{L}^{-1}\ar@/^/[uuuul]^\gamma&}~,\]
where $\mathcal L$ is a holomorphic line bundle of degree $0\leq d\leq 2g-2$ and $\beta\in H^0(X,\mathcal{L}^{-1}\mathcal{K})$ is non-zero. In this case, the splitting $\mathcal{E}=\mathcal{K}\mathcal{L}\oplus \mathcal L\oplus\mathcal{O}\oplus \mathcal{L}^{-1}\oplus \mathcal{K}^{-1}\mathcal{L}^{-1}$ is orthogonal with respect to the Hermitian metric $h$ solving the self-duality equations.
Note that the form of these Higgs bundles is similar to the cyclic Higgs bundles in the Gothen components given in \eqref{eq:GothenHiggsdiagram}. The intersection of the Gothen components with Higgs bundles of the above form occurs when $\mathcal{L}=\mathcal{K}$. By Remark \ref{r: Gothen v nonGothen components}, this intersection corresponds exactly to Higgs bundles in the Hitchin component. 

The associated anti-linear involution $\lambda: \mathcal{E} \longrightarrow \mathcal{E}$ fixing the flat $\R^{2,3}$-bundle $E_\rho$ fixes $\mathcal{O}$, $\mathcal L\oplus \mathcal{L}^{-1}$ and $\mathcal{K}\mathcal L\oplus \mathcal{K}^{-1}\mathcal{L}^{-1}$, and one gets a splitting
$$E_\rho= U\oplus \ell \oplus V,$$
where $\ell=\text{Fix}(\lambda_{\vert \mathcal{O}})$ is trivial, $U=\text{Fix}(\lambda_{\vert \mathcal L\oplus \mathcal{L}^{-1}})$ and $V=\text{Fix}(\lambda_{\vert \mathcal{K}\mathcal L\oplus \mathcal{K}^{-1}\mathcal{L}^{-1}})$.

Let $M$ be set of points $u$ in the total space of $U$ such that $\norm{u}^2 =1$,  where the norm is taken with respect to the signature $(2,3)$ metric on $\pi^*E_\rho$. Since $U$ has degree $d$, $M$ is a circle bundle of Euler class $d$ over $\Sigma$. Let $\pi: M\longrightarrow \Sigma$ denote the fibration.

The bundle $\pi^*U$ over $M$ admits a tautological section $s_2$ with $\Vert s_2 \Vert^2=1$. If $s_1$ is the section of the trivial line sub-bundle $\pi^*\ell$ normalized such that $\Vert s_1 \Vert^2=-1$, then the non-zero section $s=s_1+s_2$ of $\pi^*E_\rho$ has zero norm. The section $s$ thus defines a section $\sigma$ of the flat homogeneous bundle $\pi^*\Ein(E)$ where 
$$\Ein(E_\rho):=\big(P_\rho\times \Ein^{1,2}\big)/\SO_0(2,3)$$
and $P_\rho$ is the flat $\SO_0(2,3)$-bundle with holonomy $\rho$.
More concretely, the fiber of $\pi^*\Ein(E)$ over $x\in M$ is the set of isotropic vectors in $(\pi^* E_\rho)_x$.

\begin{prop}
The section $\sigma\in\Omega^0\big(M,\pi^*\Ein(E_\rho)\big)$ introduced above defines a maximal fibered CFL structure on $M$.
\end{prop}

\begin{proof}

In the splitting $\mathcal{E}=\mathcal{KL}\oplus \mathcal L\oplus\mathcal{O}\oplus \mathcal{L}^{-1}\oplus \mathcal{K}^{-1}\mathcal{L}^{-1}$, the Higgs field $\Phi$ and its dual $\Phi^*$ with respect to $h$ have the following expression:
$$\Phi=\left(\begin{array}{lllll}
0 & 0 & 0 & \gamma & 0 \\
1 & 0 & 0 & 0 & \gamma \\
0 & \beta & 0 & 0 & 0 \\
0 & 0 & \beta & 0 & 0 \\
0 & 0 & 0 & 1 & 0
\end{array}\right)~~~~~~\text{and}~~~~~~\Phi^*=\left(\begin{array}{lllll}
0 & 1^* & 0 & 0 & 0 \\
0 & 0 & \beta^* & 0 & 0 \\
0 & 0 & 0 & \beta^* & 0 \\
\gamma^* & 0 & 0 & 0 & 1^* \\
0 & \gamma^* & 0 & 0 & 0
\end{array}\right), $$
where $\beta^*\in \Omega^{0,1}\big(X,\Hom(\mathcal{O},\mathcal{L})\big)\cong \Omega^{0,1}\big(X,\Hom(\mathcal{L}^{-1},\mathcal{O})\big)$ is the form dual to $\beta$ using the Hermitian metric on $\mathcal L$ and $\mathcal{O}$ (and similarly for $1^*$ and $\gamma^*$).

Consider a local chart $(z,\theta)$ on $\widetilde{\Sigma}\times S^1$, where $z$ is holomorphic. In this chart, the sections $s_1$ of $\pi^*\ell$ and $s_2$ of $\pi^*U$ defined above are given by
$$s_1=\left(\begin{array}{l} 0 \\ 0 \\ 1 \\ 0 \\ 0 \end{array}\right)~~~~\text{and}~~~~s_2 = \frac{1}{\sqrt{2}}\left(\begin{array}{l} 0 \\ \mu^{-1}e^{i\theta} \\ 0 \\ \mu e^{-i\theta} \\ 0 \end{array}\right),$$
where $\mu$ is the norm of the local section $\left(\begin{array}{l} 0 \\ e^{i\theta} \\ 0 \\ 0 \\ 0 \end{array}\right)$ with respect to $\pi^*h$. In particular, if $l$ is the local section of $\pi^* \Ll$ corresponding to $e^{i\theta}$, then the restriction of $\pi^*h$ to $\pi^*(\Ll\oplus \Ll^{-1})$ is locally given by 
$$\pi^*h_{\vert \pi^*(L\oplus L^{-1})} = \mu^2 l^{-1}\otimes\overline{l}^{-1} + \mu^{-2}l \otimes \overline l.$$

Writing the flat connection $\nabla= A +\Phi + \Phi^*$ (where $A=d + \partial \log h$ is the Chern connection of $(\pi^*\mathcal{E},\pi^*h)$), one obtains
$$\nabla s_1 = \left(\begin{array}{l}0 \\ \beta^*(s_1) \\ 0 \\ \beta(s_1) \\ 0\end{array}\right).$$
The calculations for $s_2$ are more tedious. We get

$$A_{\partial_\theta}s_2=\frac{1}{\sqrt{2}} \left(\begin{array}{l}0 \\ i\mu^{-1}e^{i\theta} \\ 0 \\ -i\mu e^{-i\theta} \\ 0\end{array}\right),~ A_{\partial_z}s_2=\frac{1}{\sqrt{2}} \left(\begin{array}{l}0 \\ \mu^{-2}\partial_ze^{i\theta} \mu \\ 0 \\ -\partial_z\mu e^{-i\theta}  \\ 0\end{array}\right),~A_{\overline{\partial}_z}s_2=\frac{1}{\sqrt{2}} \left(\begin{array}{l}0 \\ -\mu^{-2}\overline{\partial}_z\mu e^{i\theta} \\ 0 \\ \overline{\partial}_z\mu e^{-i\theta} \\ 0\end{array}\right),$$

\noindent and

$$\Phi(\partial_z)(s_2) =\frac{1}{\sqrt{2}} \left(\begin{array}{l}  \gamma(\partial_z)(\mu e^{-i\theta}) \\ 0 \\  \beta(\partial_z)(\mu^{-1}e^{i\theta}) \\ 0 \\  1(\partial_z)(\mu e^{-i\theta}) \end{array}\right),~ \Phi^*(\overline\partial_z)(s_2) =\frac{1}{\sqrt{2}} \left(\begin{array}{l} 1^*(\overline\partial_z)(\mu e^{i\theta}) \\ 0 \\  \beta^*(\overline\partial_z)(\mu^{-1}e^{-i\theta}) \\ 0 \\  \gamma^*(\overline\partial_z)(\mu e^{i\theta})\end{array}\right).$$

So finally, using $s=s_1+s_2$, we get
$$\nabla_{\partial_z}s=\frac{1}{\sqrt 2}\left(\begin{array}{l}  \gamma(\partial_z)(\mu e^{-i\theta}) \\ \mu^{-2}\partial_ze^{i\theta} \mu \\  \beta(\partial_z)(\mu^{-1}e^{i\theta}) \\ -\partial_z\mu e^{-i\theta} + \sqrt{2}\beta(\partial_z)(s_1) \\  1(\partial_z)(\mu e^{-i\theta})\end{array}\right),$$

$$\nabla_{\overline{\partial}_z}s=\frac{1}{\sqrt{2}} \left(\begin{array}{l} 1^*(\overline\partial_z)(\mu e^{i\theta}) \\ -\mu^{-2}\overline{\partial}_z\mu e^{i\theta} + \sqrt{2}\beta^*(\overline{\partial}_z)(s_1) \\  \beta^*(\overline\partial_z)(\mu^{-1}e^{-i\theta}) \\ \overline{\partial}_z\mu e^{-i\theta} \\  \gamma^*(\overline\partial_z)(\mu e^{i\theta})\end{array}\right) $$
and
$$\nabla_{\partial_\theta}s =\frac{1}{\sqrt{2}} \left(\begin{array}{l}0 \\ i\mu^{-1}e^{i\theta} \\ 0 \\ -i\mu e^{-i\theta} \\ 0\end{array}\right).$$

The section $\sigma\in \Omega^0\big(M,\pi^*\Ein(E)\big)$ is transverse to the flat structure if and only if the sections $\{s,\nabla_{\partial_z}s,\nabla_{\overline{\partial}_z}s,\nabla_{\partial_\theta}s  \}$ generate a 4-dimensional space at each point. In particular the non-vanishing of the determinant $\big\vert s_1~s_2~ \nabla_{\partial_z}s~\nabla_{\overline{\partial}_z}s~\nabla_{\partial_\theta}s\big\vert$ is a sufficient condition.

The three vectors $\{s_1, s_2, \nabla_{\partial_\theta} s\}$ span the bundle $\pi^*(\Ll\oplus \mathcal{O} \oplus \Ll^{-1})$ at each point. In particular, the determinant $\big\vert s_1~s_2~ \nabla_{\partial_z}s~\nabla_{\overline{\partial}_z}s~\nabla_{\partial_\theta}s\big\vert$ vanishes exactly when the first and last component of $\{\nabla_{\partial_z}s,~\nabla_{\overline{\partial}_z}s\}$ are proportional, that is when $\Vert \gamma \Vert^2 = \Vert 1\Vert^2$.

Because the section $\gamma\in H^0(X,\mathcal{K}^2\mathcal{L}^2)$ is holomorphic, we have
$$ \Delta \log\Vert \gamma\Vert^2 = -2F_{\mathcal{KL}},$$
where $F_{\mathcal{KL}}$ is the curvature of the bundle $\mathcal{KL}$ with respect to the Hermitian metric $h$.
By the Higgs bundle equation, $F_{\mathcal{KL}} = \Vert 1 \Vert^2-\Vert \gamma\Vert^2$ and we obtain
$$\Delta \log \Vert \gamma\Vert^2=2\Vert \gamma\Vert^2-2\Vert 1 \Vert^2.$$
The maximum principle applies: at a maximum of $\Vert \gamma \Vert^2$, one has $\Vert \gamma\Vert^2<\Vert 1\Vert^2$ and so $\Vert 1 \Vert^2\neq \Vert \gamma \Vert^2$ on $\Sigma$. In particular, $\sigma\in\Omega^0\big(M,\pi^*\Ein(E)\big)$ defines a CFL structure on $M$.

Note also that the associated developing map sends the fiber of $M$ over $x$ to the space-like circle corresponding to the signature $(2,1)$ linear subspace $\ell_x\oplus U_x$, so the CFL structure is fibered. Finally, the corresponding equivariant map $\Psi: \widetilde\Sigma \longrightarrow \Gr_{2,1}(\R^{2,3})$ is the first Gauss map of the maximal surface $u: \widetilde\Sigma\longrightarrow \H^{2,2}$. By Proposition \ref{p-gaussmaps}, $\Psi$ is extremal.
\end{proof}
\begin{rmk}
	For $\Ll=\mathcal{K},$ the above construction gives maximal fibered CFL structures on $T^1\Sigma$ whose holonomy factors through a Hitchin representation. But note also that, for any $d\in \mathbb{Z}$ with $\vert d\vert < 2g-2$, our construction gives examples of maximal fibered CFL structures on a degree $d$ circle bundle over $\Sigma$ whose holonomy factor through representations in the connected component of $\Rep(\Gamma,\SO_0(2,3))$ of Toledo invariant $d$. Unfortunately, for $|d|<2g-2$, these representations do not form an open domain of the representation variety and we do not know how to characterize the representations arising this way. One can show that these representations do not come from representations into $\SO(2,2),$ so these CFL structures do not come from AdS structures on the circle bundle. It would be interesting to understand whether these representations are Anosov and whether the Einstein structures constructed above are deformation of anti-de Sitter structures.
\end{rmk}

\section{Relation with the Guichard-Wienhard construction}

In this section, we show that both the fibered photon structure of Theorem \ref{t: fibered photons and maximal reps} and the maximal CFL structures of Theorem \ref{t:EinsteinHitchin} agree with the geometric structures constructed by Guichard and Wienhard in \cite{wienhardanosov}. As a corollary, we describe the topology of the geometric structures of Guichard-Wienhard.

\subsection{Geometric structures ``\`a la Guichard-Wienhard''} Here we explain the construction of geometric structures in \cite{wienhardanosov} in the case of Anosov representations of a surface group in $\SO_0(2,n+1)$.
Let $\mathrm{P}_1$ and $\mathrm{P}_2$ be respectively the stabilizer of an isotropic line and of an isotropic 2-plane in $\R^{2,n+1}$. 
In particular, $\SO_0(2,n+1)/\mathrm{P}_1\cong \Ein^{1,n}$ is the Einstein Universe and $\SO_0(2,n+1)/\mathrm{P}_2\cong \Pho(\R^{2,n+1})$ is the set of photons in $\R^{2,n+1}$.

Given a representation $\rho\in \Rep(\Gamma,\SO_0(2,n+1))$  which is $\mathrm{P}_i$-Anosov ($i=1,2$), there exists a continuous $\rho$-equivariant map 
\[\xi_i: \partial_\infty\Gamma \longrightarrow \SO_0(2,n+1)/\mathrm{P}_i.\] 
The following was established in \cite{labouriehyperconvex} for Hitchin representations and in \cite{BILW} for maximal representations into $\Sp(2n,\R)$; using the existence of a \emph{tight} embedding $\iota: \SO_0(2,n+1) \hookrightarrow \Sp(2m,\R)$ for some $m\in \N$ (see \cite{pozzettihamlet}), we have:
\begin{prop}\label{p:AnosovnessofMaximalAndHitchin}
	If $\rho\in \Rep(\Gamma,\SO_0(2,n+1))$ is a maximal representation, then it is $\mathrm{P}_1$-Anosov. If $\rho\in \Rep(\Gamma,\SO_0(2,3))$ is a Hitchin representation, then $\rho$ is both $\mathrm{P}_1$-Anosov and $\mathrm{P}_2$-Anosov. 
\end{prop}
If $\rho$ is $\mathrm{P}_1$-Anosov, define the subset $K_\rho^2\subset \Pho(\R^{2,n+1})$ by
\[	K_\rho^2:=\left\{ V\in\Pho(\R^{2,n+1})\ |\ \xi_1(x)\subset V\  \text{for\ some}\ x\in \partial_\infty\Gamma \right\}.\]
If $\rho$ is $\mathrm{P}_2$-Anosov define the subsets $K_\rho^1\subset \Ein^{1,n}$ by
\[K_\rho^1:=\left\{ \ell\in\Ein^{1,n}\ |\ \ell\subset \xi_2(x)\  \text{for\ some}\ x\in \partial_\infty\Gamma \right\}.\]

Note that $K_\rho^2$ is homeomorphic to $\partial_\infty \Gamma\times \S^{n-1}$. If $\rho\in \Rep(\Gamma,\SO_0(2,3))$ a $\mathrm{P}_2$-Anosov representation, $K_\rho^1$ is homeomorphic to $\partial_\infty \Gamma\times\S^1.$ 
The following is proved in \cite{wienhardanosov}:
\begin{theo}\label{t: GW properly discontinuous}
	If $\rho\in \Rep\big(\Gamma, \SO_0(2,n+1)\big)$ is $\mathrm{P}_1$-Anosov, then $\rho(\Gamma)$ acts properly discontinuously and co-compactly on the set 
	\[\Omega_\rho^2= \Pho(\R^{2,n+1})\setminus K_\rho^2.\] 
	Also, if $\rho\in \Rep\big(\Gamma, \SO_0(2,n+1)\big)$ is $\mathrm{P}_2$-Anosov, then $\rho(\Gamma)$ acts properly discontinuously and co-compactly on the set 
	\[\Omega_\rho^1= \Ein^{1,n}\setminus K_\rho^1.\]
	Moreover, the topology of the quotient $\rho(\Gamma)\backslash\Omega_\rho^i$ remains constant as the representation $\rho$ is varied continuously (Theorem 9.2 of \cite{wienhardanosov}).
	\end{theo}

\subsection{Equivalence of the photon structures}\label{ss:equivalence of Structures}
Here we prove that the fibered photon structures constructed in Theorem \ref{t: fibered photons and maximal reps} are equivalent to those of Guichard-Wienhard.

\begin{theo}\label{t:Equialent Photon Structures}
Let $\rho$ be a maximal representation from $\Gamma$ to $\SO_0(2,n+1)$. Let $\Pho(T^u\oplus N^u)$ be the associated $\Pho(\R^{2,n})$-bundle over $\Sigma$ constructed in Section \ref{isotropicplanes} and let $\dev_\rho$ be the developing map of the fibered photon structure on $\Pho(T^u\oplus N^u)$ from Theorem \ref{t: fibered photons and maximal reps}. Then $\dev_\rho$ takes values in $\Omega_\rho^2$ and induces a diffeomorphism from $\Pho(T^u\oplus N^u)$ to $\rho(\Gamma) \backslash \Omega_\rho^2$.
\end{theo}
\begin{proof}
Let $\rho\in\Rep^{max}(\Gamma,\SO_0(2,n+1))$ be a maximal representation, let $u:\widetilde\Sigma\to\mathbb{H}^{2,n}$ be the $\rho$-equivariant maximal surface and let $\xi:\partial_\infty\Gamma\to \Ein^{1,n}\cong\partial \mathbb{H}^{2,n}$ be the $\rho$-equivariant continuous map given by the Anosov property of $\rho.$ Recall from Corollary \ref{p:BoundaryMaximalSurface} that the boundary of $u(\widetilde\Sigma)$ corresponds to $\xi(\partial_\infty\Gamma).$ We will show that the developing map of the fibered photon structure of Theorem \ref{t: fibered photons and maximal reps} maps bijectively onto the Guichard-Wienhard domain $\Omega_\rho^2.$

In the construction of the fibered photon structure of Theorem \ref{t: fibered photons and maximal reps}, the developing map sends the fiber of the $\Pho(\R^{2,n})$-bundle over a point $x\in \widetilde\Sigma$ bijectively onto the set of photons contained in the orthogonal of $u(x)$ in $\R^{2,n+1}$. 
By Lemma \ref{l:HyperplaneSeparatesS}, the boundary of $u(\widetilde\Sigma)$ does not intersect $u(x)^\perp$ for any $x\in\widetilde\Sigma$. 
In particular, the developing map of the space-like fibered photon structure associated to $\rho$ is contained in the domain $\Omega_\rho^2$. 

For the other inclusion, suppose $V\in\Pho(\R^{2,n+1})$ is a photon and denote its orthogonal by $V^\perp.$ The restriction of the quadratic form $\mathbf{q}$ to $V^\perp$ is non-positive, and vanishes exactly on the subspace $V$. Thus, the subspace $V^\perp$ can be approximated by a sequence $W_k$ of rank $(n+1)$ negative definite subspaces. 
By Corollary \ref{c:IntersectionSpacelikeSurface}, each plane $W_{k}$ intersects the surface $u(\widetilde\Sigma)$ in exactly one point. Thus, $V^\perp$ intersects either $u(\widetilde \Sigma)$ or its boundary. This gives rise to a dichotomy:
\begin{itemize}
	\item If $V^\perp$ intersects $u(\widetilde\Sigma)$ at a point $x,$ then $V$ is contained in $\Pho(x^\perp)$ and is in the image of developing map of the fibered photon structure. 
	\item If $V^\perp$ intersects the boundary of $u(\widetilde\Sigma)$ at a point $\xi(x),$ then $V$ contains the isotropic line $\xi(x)$, and so $V$ belongs to $K^2_\rho$.
\end{itemize}
Therefore, the developing map of the fibered photon structure from Theorem \ref{t: fibered photons and maximal reps} maps surjectively onto $\Omega_\rho^2$.
\end{proof}
The following corollary is immediate:

\begin{coro}
If $\rho: \Gamma \longrightarrow \SO_0(2,n+1)$ is a maximal representation, then the quotient $\rho(\Gamma)\backslash\Omega^2_\rho$ of the Guichard-Wienhard domain of discontinuity is diffeomorphic to a $\Pho(\R^{2,n})$-bundle over $\Sigma$ and the topology of the bundle characterizes the connected component of $\rho$. 
\end{coro}
By Lemma \ref{l: O(U,V) bundle description}, for $\SO_0(2,3)$ we can say a little more.
\begin{coro}
If $\rho: \Gamma \longrightarrow \SO_0(2,3)$ is a maximal representation, then the quotient $\rho(\Gamma)\backslash\Omega^2_\rho$ of the Guichard-Wienhard domain of discontinuity 
	\begin{itemize}
		\item is homeomorphic to a connected $\mathrm{O}(2)$-bundle over $\Sigma$ with Stiefel-Whitney classes $(sw_1,sw_2)$ if $\rho\in\Rep^{max}_{sw_1,sw_2}(\Gamma,\SO_0(2,3))$
		\item is homeomorphic to the disjoint union of two circle bundles of degree $2g-2+d$ and $2g-2-d$ if $\rho\in\Rep^{max}_{d}(\Gamma,\SO_0(2,3))$.
	\end{itemize}
\end{coro}

\begin{rmk}
In \cite{guichardwienhardsl4}, Guichard and Wienhard explicitly describe the quotient of two connected domains of $\R\P^3$ by a Hitchin representation $\rho: \Gamma \to \text{PSp}(4,\R).$ They show that the quotient gives two circle bundles over $\Sigma$, one of degree $6g-6$, the other of degree $2g-2$. In particular, these bundles are equipped with a $(\text{PSp}(4,\R),\ProjR3)$-structure. When the degree is $2g-2,$ this explicit description was also obtained using cyclic Higgs bundles by Baraglia in \cite[Section 3.5]{BaragliaThesis}.

By considering the action of $\Sp(4,\R)$ on $\Lambda^2\R^4$, one obtains the low dimensional isomorphism $\text{PSp}(4,\R)\cong \SO_0(2,3)$. Under this isomorphism, lines in $\R^4$ correspond to isotropic $2$-planes in $\R^{2,3}$. It follows that a $(\text{PSp}(4,\R),\ProjR3)$-structure is the same thing as a photon structure (see \cite[Section 5]{charette} for more details).

As a result, the two $(\text{PSp}(4,\R),\ProjR3)$-structures on the circle bundles constructed by Guichard and Wienhard for Hitchin representations correspond to the fibered photon structures on $\Pho^+(U\oplus V)$ and $\Pho^-(U\oplus V)$ of Lemma \ref{l: O(U,V) bundle description}.
\end{rmk}

\subsection{Equivalence of Einstein structures}
For a Hitchin representation $\rho:\Gamma\to\SO_0(2,3)$ there is a Guichard-Wienhard domain $\Omega^1_\rho$ in $\Ein^{1,2}$ by Proposition \ref{p:AnosovnessofMaximalAndHitchin}.

Guichard--Wienhard's theorem (Theorem \ref{t: GW properly discontinuous}) implies that the action of $\rho(\Gamma)$ on $\Omega^1_\rho$ is properly discontinuous and co-compact. Actually, one can be a bit more precise. Mimicking their construction of projective structures associated to Hitchin representations into $\SL(4,\R)$ (see \cite{guichardwienhardsl4}), one can give\footnote{This construction was done in some working notes that Guichard and Wienhard kindly shared with us.} a $\rho$-equivariant parametrization of $\Omega^1_\rho$ by the set $\partial_\infty \Gamma^{(3)}$ of oriented triples of distinct points in $\partial_\infty \Gamma$. It follows that $\rho(\Gamma) \backslash \Omega^1_\rho$ is homeomorphic to $T^1\Sigma$. However, the circle bundle structure is not appearing in this construction.

Here, we prove that the conformally flat $3$-manifold associated to $\rho$ by Theorem \ref{t:EinsteinHitchin} is isomorphic (as a conformally flat $3$-manifold) to $\rho(\Gamma) \backslash \Omega^1_\rho$.

\begin{theo} \label{t:EquivalenceEinsteinGW}
Let $\rho: \Gamma \to \SO_0(2,3)$ be a Hitchin representation. Then the developing map $\dev_\rho$ constructed in Section \ref{ss:EinsteinHitchin} is a global homeomorphism from $T^1\widetilde{\Sigma}$ to $\Omega^1_\rho$.
\end{theo}

The proof is less straightforward than that of Theorem \ref{t:Equialent Photon Structures}. We first prove the following lemma, which settles the case when $\rho$ is Fuchsian, and then argue by continuity, using the Ehresmann--Thurston principle.

\begin{lem} \label{l:GWStructureFuchsianCase}
Suppose that $\rho = m_{irr} \circ j$, where $j: \Gamma \to \PSL(2,\R)$ is a Fuchsian representation and $m_{irr}: \PSL(2,\R) \to \SO_0(2,3)$ is the irreducible representation. The developing map $\dev_\rho$ constructed in Section \ref{ss:EinsteinHitchin} is a diffeomorphism onto $\Omega^1_\rho$.
\end{lem}

Lemma \ref{l:GWStructureFuchsianCase} shows in particular that for $\rho_0 = m_{irr} \circ j$, the manifold $\rho_0(\Gamma) \backslash \Omega^1_{\rho_0}$ is homeomorphic to $T^1 \Sigma$. Now, when $\rho$ varies continuously, the topology of $\rho(\Gamma) \backslash \Omega^1_\rho$ does not vary, and its Einstein structure varies continuously by \cite[Theorem 9.2]{wienhardanosov}. Therefore, the developing map $\dev_\rho$ constructed in the proof of Theorem \ref{t:EinsteinHitchin} and the identification of $T^1 \Sigma$ with $\rho(\Gamma) \backslash \Omega^1_\rho$ discussed above
each give Einstein structures on $T^1\Sigma$ with the same holonomy $\rho$ and depend continuously on $\rho$. Since the two Einstein structures coincide at $\rho_0 = m_{irr}\circ j$, they coincide on the whole connected component of $\rho_0$ according to the Ehresmann--Thurston principle. This concludes the proof of Theorem \ref{t:EquivalenceEinsteinGW}.

\begin{proof}[Proof of Lemma \ref{l:GWStructureFuchsianCase}]

A special property of the Fuchsian case is that the developing map extends as a $\PSL(2,\R)$-equivariant map from $T^1 \H^2$ to $\Ein^{1,2}$.

Let us recall that the irreducible representation of $\SL(2,\R)$ in dimension $n+1$ is given by the action of $\SL(2,\R)$ on the space $\R_n[X,Y]$ of homogeneous polynomials of degree $n$ in two variables $X$ and $Y$. This action preserves the bilinear form $Q_n$ given by the $n^{th}$-symmetric product of the volume form on $\R^2.$ In the coordinate  coordinate system 
\[(X^n, X^{n-1}Y,  \ldots , XY^{n-1} , Y^n),\] 
a computation shows that bilinear form $Q_n$ is given by 
\[\left(\begin{matrix}
&&&& a_{n,0} \\
&&&-a_{n,1}&\\
&&\iddots&&\\
&(-1)^{n-1} a_{n,n-1}\\
(-1)^n a_{n,n}
\end{matrix}\right)\]
where $a_{n,k} = \frac{k! (n-k)!}{n!}$.

This bilinear form is anti-symmetric for $n$ odd and symmetric of signature $(2k, 2k+1)$ for $n=4k$. In particular, for $n=2$, the quadratic form $-2 Q_2$ is the discriminant of quadratic polynomials, and this representation gives the isomorphism $\PSL(2,\R) \simeq \SO_0(2,1)$. The hyperbolic plane $\H^2$ thus identifies with the projectivization of the set of quadratic polynomials with negative discriminant (that is, scalar products on $\R^2$) while $\partial_\infty \H^2$ identifies with the projectivization of the set of quadratic polynomials with vanishing discriminant (that is, squares of linear forms).

Let $j: \Gamma \to \PSL(2,\R)$ be a Fuchsian representation. We identify $j$ with its composition with the isomorphism $\PSL(2,\R) \simeq \SO_0(2,1)$. Now, $\R^{2,3}$ identifies with $\left(\R_4[X,Y], - Q_4\right)$, and the irreducible representation described above is the representation $m_{irr}$. 

In this setting, the boundary map $\xi_1: \partial_\infty \Gamma \to \Ein^{1,2}$ given by the Anosov property of $\rho_0$ is identified with the $\PSL(2,\R)$-equivariant map
\[\function{\xi_1}{\partial_\infty \H^2}{\Ein^{1,2}}{[L^2]}{[L^4]~.}\]
(Here, $[L^2]$ denotes the projective class of the square of a linear form on $\R^2$.)

Moreover, given a point $[L^2]$ in $\partial_\infty{\H^2}$, the photon $\xi_2([L^2])$ is the tangent to $\xi_1$ at $L^4$. It is thus the projectivization of the space of polynomials of the form $L^3 L'$, where $L'$ is a linear form. We conclude that the domain $\Omega^1_{\rho_0}$ of Guichard and Wienhard is the complement in $\Ein^{1,2}$ of the set of polynomials having a triple root.

On the other side, the  $\rho_0$-invariant maximal surface in $\H^{2,2}$ is the image of the $\PSL(2,\R)$-equivariant map 
\[\function	{f}{\H^2}{\H^{2,2}}{[P]}{[P^2]~,}\]
referred to as the \emph{Veronese surface} in \cite{IshiharaVeronese}.
(Here, $[P]$ denotes the projective class of a positive definite quadratic form on $\R^2$.) 

Let $P$ be a positive definite quadratic form on $\R^2$. The tangent space to this maximal surface at the point $f([P])$ is the projective space of polynomials of the form $PQ$, with $Q\in \R_2[X,Y]$. Since none of these polynomials has a triple root, the intersection of this tangent space with $\Ein^{1,2}$ is contained in the domain $\Omega^1_{\rho_0}$. By construction of the developing map $\dev_{\rho_0}$ it follows that $\dev_{\rho_0}$ takes values in $\Omega^1_{\rho_0}$.

Let $[P]$ and $[Q]$ be two distinct points in $\H^2$. Then the intersection between the tangent spaces to $f(\H^2)$ at $f([P])$ and $f([Q])$ is the point $[PQ]$, which never belongs to $\Ein^{1,2}$. Indeed, up to applying an element of $\PSL(2,\R)$, one can assume that $[P] = [X^2+Y^2]$ and $[Q] = [aX^2 + bY^2]$. One easily compute that
\[Q_4(PQ) = \frac{1}{6}(a+b)^2 + 2ab~,\]
which never vanishes when $a$ and $b$ are of the same sign. By construction, it follows that $\dev_{\rho_0}$ is injective.

Let us finally prove that $\dev_{\rho_0}$ maps surjectively onto $\Omega^1_{\rho_0}$. Let $P$ be a non-zero polynomial of degree $4$ such $Q_4(P) = 0$. Suppose that $[P]$ is not in the image of $\dev_{\rho_0}$. Then $P$ is not divisible by a positive definite quadratic form and $P$ thus splits as a product of $4$ linear forms. If all these linear forms are co-linear, then $[P]$ belongs to the image of $\xi_1$ and thus not to $\Omega^1_{\rho_0}$. Otherwise, one can assume (up to applying an element of $\PSL(2,\R)$) that $P$ has the form
\[XY(aX+bY)(cX+dY)~.\]
One then computes that 
\[Q_4(P) = \frac{1}{6}\left((ad)^2 + (bc)^2 -adbc\right)~.\]
Since the polynomial $A^2 +B^2 - AB$ is positive definite, the fact that $Q_4(P)$ vanishes implies that both $ad$ and $bc$ vanish, from which we easily deduce that $P$ is divisible by $X^3$ or $Y^3$. Therefore, $P$ belongs to the complement of $\Omega^1_{\rho_0}$. 

By contraposition, we deduce that, if $[P]$ belongs to $\Omega^1_{\rho_0}$, then $P$ is divisible by a positive definite quadratic form. Therefore, the developing map $\dev_\rho$ maps surjectively onto $\Omega^1_{\rho_0}$. This concludes the proof of Lemma \ref{l:GWStructureFuchsianCase}.
\end{proof}

\appendix

\section{The Ahlfors--Schwarz--Pick lemma}

The key argument in our length spectrum comparison result (Theorem \ref{t:LengthSpectrumComparison}) is a version of the so-called Ahlfors--Schwarz--Pick lemma. Roughly, this lemma states that an inequality between the curvatures of two conformal metrics implies a comparison between the metrics. A common reference for this fact is Wolpert's paper \cite{Wolpert82}. However, Wolpert adds the unnecessary assumption that both metric be negatively curved. Moreover, he does not discuss the equality case. It thus seemed useful to include here a proof of the following:

\begin{theo} \label{t:CurvatureIneq=>MetricIneq}
Let $\Sigma$ be a closed Riemann surface and let $g$ and $h$ be two conformal metrics of class $\mathcal{C}^2$ on $\Sigma$ with respective Gauss curvature $\kappa(g)$ and $\kappa(h)$. Assume that $\kappa(g)$ is negative. If $\kappa(h) \geq \kappa(g)$, then either $h = g$ everywhere, of there exists $\lambda >1$ such that $h \geq \lambda g$.
\end{theo}

The proof will follow from applying a maximum principle to the ratio of the two metrics, which satisfies an elliptic equation involving the curvatures of $g$ and $h$. 

If $z = x+ i y$ is a local conformal coordinate on $\Sigma$ we denote by $\Delta_0$ the Laplace operator in this coordinate, namely $\Delta_0 = \frac{1}{2}\left(\frac{\partial^2}{\partial x^2}+ \frac{\partial^2}{\partial y^2}\right)$. If we write $g = e^{\sigma} \vert d z \vert^2$, then the operator $\Delta_g = e^{-\sigma} \Delta_0$ is the Laplace--Beltrami operator of the metric $g$ (in particular, it is independent of the coordinate). Moreover, we have
\[\kappa(g) = - \Delta_g \sigma~.\]

\begin{lem} \label{l:ConformalChangeCurvature}
Let $u:\Sigma \to \R$ be the function of class $\mathcal{C}^2$ such that $h = e^{u} g$. Then we have
\[\Delta_g u = \kappa(g) - e^u \kappa(h)~.\]
\end{lem}

\begin{proof}
In a local coordinate $z$ such that $g = e^\sigma \vert d z\vert^2$, we have $h = e^{\sigma + u} \vert d z\vert^2$ and 
\begin{eqnarray*}
\kappa(h) &=& -e^{-(\sigma + u)}\Delta_0(\sigma + u) \\
&=& - e^{-u} \Delta_g \sigma - e^{-u} \Delta_g u \\
&=&e^{-u} \kappa(g) - e^{-u} \Delta_g u~.
\end{eqnarray*}
Multiplying by $e^u$, we get
\[e^u \kappa(h) = \kappa(g) - \Delta_g u~.\]
\end{proof}

The weak inequality between $h$ and $g$ will easily follow from applying a maximum principle to the function $u$. To obtain a strict inequality, we will need the following strong version of the maximum principle that dates back to Picard. The proof we give here is based on notes of Sweers \cite{NotesStrongMaxPrinciple}. These notes include references about the history of this result.

\begin{lem} \label{l:StrongMaxPrinciple}
Let $u$ be a function of class $\mathcal{C}^2$ on a connected open set $U\subset \C$. Assume that $u\geq 0$ and that there exists a constant $K>0$ so that $\Delta_0 u \leq K u$. Then either $u$ vanishes identically, or $u>0$ everywhere.
\end{lem}

\begin{proof}
Let us prove that the set $W$ where $u$ vanishes is open in $U$. Since it is obviously closed, it will be either empty or the whole domain $U$.

Assume by contradiction that $W$ is not open. Then we can find a point $p$ in $U$, a small radius $r>0$ and a point $q$ at distance $r$ from $p$ such that $\overline{B}(p,r)\subset U$ and $W\cap \overline{B}(p,r)= \{q\}$. Indeed, let $a$ be a point in $W$ which is not in the interior of $W$. Let $r_0$ be such that $B(a,r_0) \subset U$. Let $b$ be a point in $B(a,\frac{r_0}{2})$ such that $u(b) >0$. Set $r_1 = d(b,W)$. Then $0 < r_1 < \frac{r_0}{2}$ and thus $\overline{B}(b,r_1) \subset B(b,\frac{r_0}{2}) \subset B(a,r_0) \subset U$. Moreover, $\overline{B}(b,r_1)$ intersects $W$ on its boundary. Take $q$ a point in $\overline{B}(b,r_1)\cap W$, $p = \frac{b+q}{2}$ and $r = \frac{r_1}{2}$. Then $p$, $q$ and $r$ satisfy the required hypotheses.

Let $r'\in (0,r)$ be such that $\overline{B}(q,r')\subset U$. Define $f_\alpha(z) = e^{\alpha r^2}- e^{\alpha |z-p|^2}$. Note that $f_\alpha$ is positive outside the ball of radius $r$ centered at $p$. A simple computation shows that
\[\Delta_0 f_\alpha = (4\alpha - 4\alpha^2 |z-p|^2)e^{-\alpha |z-p|^2}~,\]
and one verifies that, for $\alpha$ large enough, $\Delta_0 f_\alpha(z) < K f_\alpha(z)$ for $z\in \overline{B}(q,r')$. Choose such an $\alpha$. Let $S(q,r')$ denote the circle of and radius $r'$ about $q$. Then, by construction of $p$, $q$ and $r$, $u$ is positive on $\overline{B}(p,r)\cap S(q,r')$, and so is $u+\epsilon f_\alpha$ for $\epsilon >0$ small enough. Choose such an $\epsilon$. 

Since $u\geq 0$ and $f_\alpha >0$ on $U \backslash \overline{B}(p,r)$, we actually have $u + \epsilon f_\alpha >0$ on $S(q,r')$. Moreover, we have 
\[\Delta_0 (u+ \epsilon f_\alpha) < K (u+ \epsilon f_\alpha)\]
on $\overline{B}(q,r')$. Let $m$ be a point in $\overline{B}(q,r')$ such that
\[u(m)+ \epsilon f_\alpha(m) = \inf \{ u(x) +\epsilon f_\alpha(x),~x\in \overline{B}(q,r')\}.\]
If $m$ belongs to $\overset{\circ}{B}(q,r')$, then 
\[u(m) + \epsilon f_\alpha(m) > \frac{1}{K} \Delta_0 (u+ \epsilon f_\alpha)(m) \geq 0~.\]
If $m$ belongs to $S(q,r')$, then $u(m) + \epsilon f_\alpha(m)>0$. In any case, we obtain that $u+\epsilon f_\alpha >0$. This contradicts the fact that
\[u(q)+ \epsilon f_\alpha(q) = u(q) = 0~.\]

In conclusion $W$ is both closed and open, hence it is either empty or the whole domain $U$.
\end{proof}

We can now prove Theorem \ref{t:CurvatureIneq=>MetricIneq}.

\begin{proof}[Proof of Theorem \ref{t:CurvatureIneq=>MetricIneq}]
Let $u: \Sigma \to \R$ be the $\mathcal{C}^2$ function such that $h= e^u g$. Let $m\in \Sigma$ be a point at which $u$ achieves its minimum. We then have
\[\Delta_g u (m) \geq 0~,\]
and therefore
\[\kappa(g)(m) - e^{u(m)} \kappa(h)(m) \geq 0\]
by Lemma \ref{l:ConformalChangeCurvature}. We thus have
\[e^{-u(m)} \kappa(g)(m) \geq \kappa(h)(m) \geq \kappa(g)(m) \quad \textrm{(using $\kappa(h) \geq \kappa(g)$).}\]
Finally, since $\kappa(g) <0$, we deduce that $e^{-u(m)} \leq 1$. Hence $u(m)\geq 0$ and thus $u\geq 0$ on $\Sigma$.

If $u$ is identically $0$, then $h=g$. Otherwise, let us prove that $u$ is positive everywhere. By compactness of $\Sigma$ we have $u\geq a$ of some $a >0$ and $h\geq e^{a} g$. Assume by contradiction that $u$ vanishes somewhere but not identically. Then one can find a point $x\in \Sigma$ such that $u(x)=0$ but $u$ does not vanish identically on any neighborhood of $x$. Let $z$ be a local holomorphic coordinate defined in a compact connected neighborhood $V$ of $x$, and write $g = e^\sigma |d z|^2$. In the coordinate $z$, we have
\begin{eqnarray*}
\Delta_0 u &=& e^{\sigma}\left(\kappa(g) - e^u \kappa(h)\right) \\
&\leq & e^\sigma \kappa(g) (1-e^u) \\
& \leq & K u
\end{eqnarray*}
for some constant $K>0$ (depending on $V$). Thus Lemma \ref{l:StrongMaxPrinciple} applies, contradicting the fact that $u$ does not vanish identically on $V$. 
\end{proof}

\bibliographystyle{alpha}
\bibliography{biblio.bib}
\end{document}